\documentclass[11pt,reqno]{amsart}
\usepackage{a4wide}

\usepackage{mathrsfs}
\usepackage{amsfonts}
\usepackage{amsmath}
\usepackage{stmaryrd}
\usepackage{amssymb}
\usepackage{amsthm}
\usepackage{esint}

\usepackage{url}

\usepackage{amscd}
\usepackage{indentfirst}
\usepackage{enumerate}
\usepackage{graphicx}
\usepackage{appendix}

\usepackage[numbers,sort&compress]{natbib}

\usepackage{color}
\usepackage[colorlinks, linkcolor=blue, citecolor=blue]{hyperref}

\emergencystretch=\maxdimen
\newcommand{\pa}{\partial}

\renewcommand{\div}{{\rm div}}
\newcommand{\curl}{{\rm curl}}
\newcommand{\norm}[1]{\|#1\|}

\newtheorem{thm}{Theorem}[section]
\newtheorem{lem}[thm]{Lemma}
\newtheorem{pro}[thm]{Proposition}

\newtheorem{rem}[thm]{Remark}

\numberwithin{equation}{section}

\begin{document}

\title[Nishida-Smoller type large solutions for CNS equations in exterior domains]
{Nishida-Smoller type large solutions for the compressible Navier-Stokes equations with slip boundary conditions in 3D exterior domains}

\author[M. Xie]{Minghong Xie}
\address[Minghong Xie]{School of Mathematics and Statistics, Guangxi Normal University, Guilin, Guangxi 541004, P.R. China.}
\email{xiemh0622@hotmail.com}
\author[S. Xu]{Saiguo Xu}
\address[Saiguo Xu]{School of Mathematics and Statistics, Guangxi Normal University, Guilin, Guangxi 541004, P.R. China.}
\email{xsgsxx@126.com}
\author[Y.H. Zhang]{Yinghui Zhang*}
\address[Yinghui Zhang] {School of Mathematics and Statistics, Guangxi Normal University, Guilin, Guangxi 541004, P.R.
China} \email{yinghuizhang@mailbox.gxnu.edu.cn}

\thanks{* Corresponding author}

\thanks{This work was supported by Guangxi Natural Science Foundation $\#$2024GXNSFDA010071, $\#$2024JJB110082, National Natural Science
Foundation of China $\#$12271114,  Project for Enhancing the Basic Research Capacity of Young and Middle-aged Teachers in Colleges of Guangxi $\#$2025KY0102, Center for Applied Mathematics of Guangxi (Guangxi Normal University) and the Key Laboratory of Mathematical Model and Application (Guangxi Normal University), Education Department of Guangxi Zhuang Autonomous Region.}

\date{\today}

\begin{abstract}
This paper investigates the global existence of classical solutions to the isentropic compressible Navier-Stokes equations with slip boundary condition in a three-dimensional (3D) exterior domain. It is shown that
the classical solutions with large initial energy and vacuum  exist
globally in time when the adiabatic exponent $\gamma>1$ is sufficiently close to 1
(near-isothermal regime).
This constitutes an extension of the
celebrated result for the one-dimensional Cauchy problem of the isentropic Euler equations
that has been established in 1973 by Nishida and Smoller (Comm. Pure Appl. Math. 26
(1973), 183-200).
To the best of our knowledge, we establish the first result on the global existence of large-energy solutions with vacuum to  the  compressible
Navier-Stokes equations with slip boundary condition in a 3D exterior domain, which improves
previous related works where either small initial energy is
required or boundary effects are ignored.
\end{abstract}

\maketitle
{\small
}

\section{Introduction}
The motion of a general viscous compressible isentropic fluid in a three-dimensional exterior domain $\Omega\subset\mathbb{R}^3$ is governed by the compressible Navier-Stokes equations:
\begin{equation}\label{Large-CNS-eq}
  \begin{cases}
    \partial_t\rho+\operatorname{div}(\rho u)=0, \\
    \partial_t(\rho u)+\operatorname{div}(\rho u\otimes u)-\mu\Delta u-(\mu+\lambda)\nabla\operatorname{div} u+\nabla P(\rho)=0,
  \end{cases}
\end{equation}
where $\rho$, $u$, and $P(\rho)=a\rho^{\gamma}$ ($a>0$) represent the fluid density, velocity, and pressure respectively. The adiabatic exponent satisfies $\gamma>1$, while the viscosity coefficients $\mu$ and $\lambda$ adhere to the physical constraints:
\begin{equation}\label{viscosity-condition}
  \mu>0,\quad \lambda+\frac{2}{3}\mu\geq0.
\end{equation}
We consider the system \eqref{Large-CNS-eq} in an exterior domain $\Omega=\mathbb{R}^3\setminus D$, where $D$ is a simply connected bounded domain with smooth boundary $\partial D$. The equations are supplemented with initial data
\begin{equation}\label{initial-data}
  (\rho, u)(x,0)=(\rho_0, u_0)(x), \quad x\in\Omega,
\end{equation}
Navier-slip boundary conditions
\begin{equation}\label{Navier-slip-condition}
    u\cdot n=0,\, \operatorname{curl} u\times n=-An \text{ on } \partial\Omega,
\end{equation}
and far-field behavior
\begin{equation}\label{far-field-condition}
    (\rho,u)(x,t)\rightarrow(\rho_{\infty},0) \text{ as } |x|\rightarrow\infty,
\end{equation}
where $n$ denotes the unit outer normal to $\partial\Omega$, and $A=A(x)$ is a $3\times3$ symmetric matrix defined on $\pa\Omega$. There exist some different forms of slip boundary conditions related to \eqref{Navier-slip-condition}, where the detailed discussions can be found in \cite{Cai-Li2023}.

\subsection*{Previous work}
The well-posedness theory for compressible Navier-Stokes equations has been extensively studied under various geometric configurations:

\textbf{Whole space and periodic domains:} The one-dimensional theory is relatively complete \cite{Hoff1987,Kazhikhov1977,Serre1986-1,Serre1986-2,Kazhikhov1982}. In higher dimensions, Nash \cite{Nash1962} and Serrin \cite{Serrin1959} established local well-posedness for smooth initial data without vacuum. For initial data containing vacuum, local strong solutions were investigated in \cite{Cho2004,Cho2006-1,Cho2006-2,Choe2003,Salvi1993}. Global existence results for small perturbations of equilibrium were achieved by Matsumura-Nishida \cite{Matsumura1980} and extended to discontinuous data by Hoff \cite{Hoff1991,Hoff1995}. Breakthroughs for large data came with Lions' \cite{Lions1998} and Feireisl's \cite{Feireisl2001} weak solutions for $\gamma > \frac{9}{5}$ and $\gamma > \frac{3}{2}$ respectively. Huang-Li-Xin \cite{Huang-Li2012} later established global classical solutions with small energy but possibly large oscillations.

\textbf{Bounded domains and half space:} For Dirichlet boundary conditions, Lions-Feireisl's weak solutions theory extends naturally.
For the general bounded smooth domain,
the global existence of strong (or classical) solutions has been established by \cite{Cai-Li2023} for the 3D case with small initial energy, and \cite{Fan-Li-Li2022} for the 2D case with large initial energy, both of which are equipped with slip boundary conditions. For the 3D bounded domain with non-slip boundary condition, Fan and Li \cite{Fan-Li-arX2021} proved the global existence of classical solutions to the barotropic compressible Navier-Stokes system with small initial energy.
 For slip boundary conditions in half space $\mathbb{R}^3_+$, Hoff \cite{Hoff2005} proved the global existence of weak solutions with small initial energy.

\textbf{Exterior domains:} Novotn\'{y}-Str\v{a}skraba \cite{Novotny2004} established weak solutions in general domains, while Cai-Li-L\"{u} \cite{Cai-Li-Lv2021} proved global classical solutions with small initial energy, analogous to Huang-Li-Xin's whole space results.

\subsection*{Uniqueness and regularity challenges}
Despite these advances, fundamental questions remain open. The uniqueness and regularity of Lions-Feireisl weak solutions with arbitrary data are still unresolved. Recent progress focuses on special configurations: Jiang-Zhang \cite{Jiang2001,Jiang2003} obtained global weak solutions for symmetric flows, while Hoff \cite{Hoff2005} constructed special weak solutions with extra regularity. Recently, Hong-Hou-Peng-Zhu \cite{Zhu2024} established Nishida-Smoller type large solutions in whole space when $\gamma$ is near 1, allowing both large initial energy and vacuum. Very recently, the authors \cite{Xu-Zhang2025} proved the global well-posedness and
large time behavior of Nishida-Smoller type large solutions to compressible Navier-Stokes equations \eqref{Large-CNS-eq} with vacuum and slip boundary conditions \eqref{Navier-slip-condition} in a 3D bounded domain, which generalizes the results of \cite{Zhu2024} and \cite{Cai-Li2023}.
This type solution can be viewed as the Nishida-Smoller type large solution which is originally studied for the conservation laws with BV initial data in \cite{Nishida-Smoller1973}, where Nishida and Smoller showed the global existence of solutions to the Cauchy problem of 1D isentropic Euler equations under the condition that $(\gamma-1).\text{total var.}\{u_0,\rho_0\}$ is sufficiently small. In particular, this result implies that the initial energy could be large as $\gamma$ is sufficiently close to 1. For some generalizations of the Nishida-Smoller type results on inviscid or viscous flow, one can see for instance \cite{Kawashima-Nishida1981,Liu-Yang-Zhao2014,Tan-Yang-Zhao2013,Liu1977,Temple1981,Zhu2017}.

\subsection*{Main motivation}
To conclude, all the works \cite{Cai-Li-Lv2021, Cai-Li2023, Fan-Li-arX2021, Huang-Li2012, Zhu2024, Xu-Zhang2025} depend essentially on small initial energy or
the advantages of the whole space and bounded domain. Therefore, a natural and important problem is to study what will
happen if both large initial energy and exterior domains are involved. More precisely, we prove that when the adiabatic exponent $\gamma$ is sufficiently close to 1, the compressible Navier-Stokes equations \eqref{Large-CNS-eq} in exterior domains admit global classical solutions with large initial energy and vacuum.
Our approach combines energy methods adapted to exterior domains with careful analysis of the adiabatic exponent's role in pressure regularization. The Navier-slip boundary conditions needs delicate energy estimates to control far-field behavior while maintaining compatibility with boundary conditions.

\bigskip

Before stating our result, let us introduce the following notations and conventions used throughout this paper. We set
\begin{equation*}
  \int f=\int_{\Omega}fdx,\quad\int_0^Tg=\int_0^Tgdt,
\end{equation*}
and a ball $B_R$ as
\begin{equation*}
  B_R=\{x\in\mathbb{R}^3||x|<R\}.
\end{equation*}

For $1\leq r\leq\infty$, and integer $k\geq1$, we denote the standard Sobolev spaces as follows:
\begin{equation}\label{Notation-Sobolev}
  \begin{cases}
     L^r=L^r(\Omega),\,D^{k,r}=\{u\in L_{loc}^1(\Omega): \norm{\nabla^ku}_{L^r}<\infty\},\\
     W^{k,r}=L^r\cap D^{k,r},\,H^k=W^{k,2},\, D^k=D^{k,2},\\
     D_0^1=\{u\in L^6: \norm{\nabla u}_{L^2}<\infty,\textrm{ and \eqref{Navier-slip-condition} holds}\},\\
    H_0^1=L^2\cap D_0^1,\,\norm{u}_{D^{k,r}}=\norm{\nabla^ku}_{L^r}.
  \end{cases}
\end{equation}
For some $s\in(0,1)$, the fractional Sobolev space $H^s(\Omega)$ is defined by
\begin{equation*}
  H^s(\Omega):=\left\{u\in L^2(\Omega): \int_{\Omega\times\Omega}\frac{|u(x)-u(y)|^2}{|x-y|^{n+2s}}dxdy<\infty\right\}
\end{equation*}
with the norm:
\begin{equation*}
  \norm{u}_{H^s(\Omega)}:=\norm{u}_{L^2(\Omega)}+\left(\int_{\Omega\times\Omega}\frac{|u(x)-u(y)|^2}{|x-y|^{n+2s}}dxdy\right)^{\frac{1}{2}}.
\end{equation*}
The initial total energy of \eqref{Large-CNS-eq} is defined as
\begin{equation}\label{defi-total-energy}
  E_0:=\int(\frac{1}{2}\rho_0|u_0|^2+G(\rho_0)),\quad G(\rho)=\rho\int_{\rho_{\infty}}^{\rho}\frac{P(s)-P(\rho_{\infty})}{s^2}ds,
\end{equation}
and the modified initial energy involving $\gamma-1$ is denoted as
\begin{equation}\label{defi-modified-energy}
  \mathcal{E}_0:=\int\frac{1}{2}\rho_0|u_0|^2+(\gamma-1)E_0.
\end{equation}
In what follows, we denote by $C>0$ a generic constant possibly depending on $\mu, \lambda, a, \tilde{\rho}, \Omega, M$ and the matrix $A$, but independent of $\gamma-1, E_0, \mathcal{E}_0$ and $t$. And we write $C(\alpha)$ to emphasize the dependence of $C$ on the parameter $\alpha$.

Now, we are ready to state our main results.
\begin{thm}\label{thm-global-CNS-exterior}
 Let $\Omega$ be the exterior of a simply connected bounded domain $D$ in $\mathbb{R}^3$ and its boundary $\pa\Omega$ is smooth. For given positive constants $M$ and $\tilde{\rho}\geq\rho_{\infty}+1$, suppose that the $3\times3$ symmetric matrix $A$ in \eqref{Navier-slip-condition} is smooth and positive semi-definite, and the initial data $(\rho_0,u_0)$ satisfy for some $q\in(3,6)$,
 \begin{equation}\label{initial-data1}
   (\rho_0-\rho_{\infty},P(\rho_0)-P(\rho_{\infty}))\in W^{2,q},\quad u_0\in D_0^1\cap D^2,
 \end{equation}
 \begin{equation}\label{initial-data2}
 \begin{aligned}
   0\leq\rho_0\leq\tilde{\rho},&\quad\norm{\nabla u_0}_{L^2}\leq M,\\
   \rho_0\in L^{\frac{3}{2}}&\textrm{ if }\rho_{\infty}=0,
 \end{aligned}
 \end{equation}
 and the compatibility condition
 \begin{equation}\label{compatibility-condition}
   -\mu\Delta u_0-(\mu+\lambda)\nabla\div u_0+\nabla P(\rho_0)=\rho_0^{\frac{1}{2}}g,
 \end{equation}
 for some $g\in L^2$. If $\rho_{\infty}=0$, then the initial-boundary value problem \eqref{Large-CNS-eq}-\eqref{far-field-condition} admits a unique classical solution $(\rho,u)$ in $\Omega\times(0,\infty)$ satisfying that
 \begin{equation}\label{density-bound}
   0\leq\rho(x,t)\leq2\tilde{\rho},\,(x,t)\in\Omega\times(0,\infty),
 \end{equation}
 and for any $0<\tau<T<\infty$,
 \begin{equation}\label{classical-sol}
   \begin{cases}
     (\rho,P(\rho))\in C([0,T];W^{2,q}),\\
     \nabla u\in C([0,T];H^1)\cap L^{\infty}(\tau,T;W^{2,q}),\\
     u_t\in L^{\infty}(\tau,T;H^2)\cap H^1(\tau,T;H^1),\\
     \sqrt{\rho}u_t\in L^{\infty}(0,\infty;L^2),
   \end{cases}
 \end{equation}
 provided
 \begin{equation}\label{small-condition}
   \mathcal{E}_0\leq\epsilon,
 \end{equation}
 where $\epsilon>0$ is a small constant depending on $\mu, \lambda, a, \tilde{\rho}, \Omega, M, E_0$, but independent of $\gamma-1$ and $t$ (see \eqref{small-assumption1}, \eqref{small-assumption2}, \eqref{small-assumption3} and \eqref{small-assumption4}), precisely characterized as
  \begin{equation*}
 \begin{aligned}
   \epsilon&=\min\left\{1,(4C(\tilde{\rho}))^{-12},(C(\tilde{\rho},M))^{-2},
   (4C(\tilde{\rho}))^{-2},(3C(\tilde{\rho}))^{-16},(3C(\tilde{\rho},M)(E_0+1))^{-2},(1+E_0)^{-\frac{16}{3}},\right.\\
   &\qquad\left.(4C(\tilde{\rho}))^{-\frac{128}{3}}E_0^{-\frac{56}{3}}, (4C(\tilde{\rho},M))^{-\frac{8}{5}}, (4C(\tilde{\rho})(1+E_0))^{-8},\left(\frac{\tilde{\rho}}{2C(\tilde{\rho},M)}\right)^{-\frac{32}{3}},\left(\frac{\tilde{\rho}}{4C(\tilde{\rho})(1+E_0)}\right)^8\right\}
   \end{aligned}
 \end{equation*}
 and also the matrix $A$ has certain smallness as
 \begin{equation*}
    \norm{A}_{W^{1,6}}\leq\min\left\{1,(3CE_0)^{-\frac{3}{4}}\mathcal{E}_0^{\frac{3}{32}}, E_0^{-\frac{2}{3}}\mathcal{E}_0^{\frac{7}{24}}, (4CE_0)^{-\frac{8}{9}}\mathcal{E}_0^{\frac{1}{6}}\right\},\quad\norm{A}_{W^{1,\infty}}\leq\mathcal{E}_0^{-\frac{1}{8}},
  \end{equation*}
  which can be found in \eqref{small-A-assumption1} and \eqref{small-A-assumption2} with $C$ here depending only on $\mu,\lambda$ and $\Omega$.

\end{thm}

Here we list some remarks as follows.
\begin{rem}\label{remark-1}
  Our work establishes the first Nishida-Smoller type large-energy solutions for compressible Navier-Stokes equations in exterior domains, overcoming two fundamental difficulties absent in previous studies:

  \begin{itemize}
    \item \textbf{Boundary-layer phenomena}: Unlike the Cauchy problem in \cite{Zhu2024}, the Navier-slip condition \eqref{Navier-slip-condition} introduces boundary integrals such as (see \eqref{boundary-term-A2-4}):
    \begin{equation}\label{boundary-effect}
      \int_{\pa\Omega}(\curl u_t\times n)\cdot\dot{u} dS,
    \end{equation}
   requiring new vorticity control mechanisms near the boundary.

    \item \textbf{Non-compact geometric constraints}: Compare to bounded domain results in \cite{Xu-Zhang2025}, exterior domain geometry prevents key analytic tools, i.e.,
    both the embedding between different $L^p$-spaces and Poincar\'{e}'s inequality are invalid.

  \end{itemize}
\end{rem}

\begin{rem}\label{remark-2}
  Compared to Cai-Li-L\"{u} \cite{Cai-Li-Lv2021} where the global existence and large time behavior of classical solutions to \eqref{Large-CNS-eq}-\eqref{Navier-slip-condition} with
  small initial energy and vacuum are obtained, the initial energy $E_0$ is allowed to be large in our case when $\gamma$ is close to 1 and $A$ is suitably small. Therefore, Theorem \ref{thm-global-CNS-exterior} is still applicable to the case that the initial energy $E_0$ is small for any given $\gamma$ and $A$. In the above theorem, we can further give similar long-time behaviors as in \cite{Cai-Li-Lv2021}, but our attention is more focused on whether the long-time decay rate is  influenced by $\gamma\rightarrow1$, just as in \cite{Xu-Zhang2025}. Unfortunately, we have not gotten this relation, even using the method in \cite{Li-Xin2019}.
\end{rem}

\begin{rem}\label{remark-3}
  Our results reveal an intrinsic relationship between initial energy scaling and adiabatic exponent: Theorem \ref{thm-global-CNS-exterior} constitutes a natural extension of Lions-Feireisl weak solution theory \cite{Lions1998,Feireisl2001} to the regime $\gamma\in(1,\frac{3}{2}]$. Specifically:
  \begin{itemize}
    \item For $\gamma\to1^+$, we permit arbitrarily large $E_0$ through $(\gamma-1)$-compensation;
    \item For fixed $\gamma>\frac{3}{2}$, our framework aligns with classical weak solution requirements.
  \end{itemize}
  This dichotomy highlights a fundamental open question: existence of global classical solutions with large initial data for fixed $\gamma>1$ remains unresolved, suggesting new phenomena may emerge beyond $\gamma$-compensation mechanisms.
\end{rem}

\begin{rem}\label{remark-4}
  The critical scaling relationship
  \begin{equation}\label{critical-scaling}
    (\gamma-1)E_0^{\frac{59}{3}} \leq C
  \end{equation}
  fundamentally differs from previous works \cite{Zhu2017,Zhu2024,Xu-Zhang2025} due to two key factors:
  \begin{itemize}
    \item Slip boundary effects in \eqref{Navier-slip-condition} introducing matrix $A$ dependence
    \item Exterior domain geometry affecting Hodge decomposition \eqref{Hodge-decomposition-exterior-Final}
  \end{itemize}
  It should be mentioned that the smallness condition on $A$ in \eqref{small-condition} provides a boundary counterpart to Zhu's far-field density constraints \cite{Zhu2024}. The technical requirement $\rho_\infty=0$ emerges from essential $L^2$-dissipation estimates needed to control:
  \begin{equation}\label{dissipation-balance}
  \begin{aligned}
    \int_0^T\norm{P-P(\rho_{\infty})}_{L^2}^2&\leq C\norm{\rho}_{L^{\frac{3}{2}}}\norm{\nabla u}_{L^2}\norm{P-P(\rho_{\infty})}_{L^2}\\
    &\quad
    +C\int_0^T\norm{\rho}_{L^{\frac{3}{2}}}\norm{\nabla u}_{L^2}^2
    +\textrm{good terms}.
  \end{aligned}
  \end{equation}
  This constraint reflects the intrinsic challenge of pressure-velocity coupling in exterior domains.
\end{rem}

\begin{rem}\label{remark-5}
  In addition to the conditions of Theorem \ref{thm-global-CNS-exterior}, if assuming further that $\norm{u_0}_{\dot{H}^{\beta}}\leq\tilde{M}$ with $\beta\in(\frac{1}{2},1]$ instead of $\norm{\nabla u_0}_{L^2}\leq M$, then the conclusions in Theorem \ref{thm-global-CNS-exterior} still hold. This can be achieved by a similar way as in \cite{Huang-Li2012}. In our results, we also do not focus on the regularity of the bounded domain $\Omega$ and the matrix $A$, but we can make analogous discussions as in \cite{Cai-Li2023}.
\end{rem}

Now, let us some comments on the analysis of this paper. Similar to the arguments in \cite{Cai-Li-Lv2021} and \cite{Zhu2024}, the key issue in our proof is to derive the time-independent upper bound on the density $\rho$ (see Lemma \ref{lem-rho-bound}). However, compared to \cite{Cai-Li-Lv2021, Zhu2024} where the analysis relies heavily on the smallness of the initial energy $E_0$ or the advantage of the whole space, we need develop new thoughts to handle large initial energy and exterior domain complexities. The main difficulties involves:
\begin{itemize}
    \item Modified energy hierarchy with nonlinear coupling;
    \item Anisotropic dissipation estimates incorporating boundary terms;
    \item Geometric decomposition techniques for exterior domains.
\end{itemize}

In the following, we highlight the main differences and new ingredients:

\begin{itemize}
  \item Since the smallness is imposed on the modified initial energy $\mathcal{E}_0$ instead of the original one $E_0$, we can only obtain the smallness of $\norm{P}_{L^1}$ from the basic energy estimate, while the crucial dissipation $\displaystyle\int_0^T\norm{\nabla u}_{L^2}^2$ has no smallness. To overcome this difficulty, we modify the method of \cite{Zhu2024}. However, due to the boundary effects and the feature of exterior domain, we need employ new ideas to deal with the difficulties arising from exterior domain complexities.

  \item Since the initial energy $E_0$ could be large in our analysis, we can only get the smallness of $\displaystyle\int_0^{\sigma(T)}\norm{\nabla u}_{L^2}^2$ rather than $\displaystyle\int_0^T\norm{\nabla u}_{L^2}^2$ (see Lemma \ref{lem-essential-estimate}). By delicate energy estimates, we can get the estimates of $A_1(T)$ and $A_2(T)$ stated in Lemma \ref{lem-A1-A2-control}:
      \begin{equation}\label{intro-A1-control}
      \begin{aligned}
       A_1(T)&\leq C(\tilde{\rho})\mathcal{E}_0+C(\tilde{\rho})A_1^{\frac{3}{2}}(T)+C\norm{A}_{W^{1,6}}^{\frac{3}{2}}A_1^{\frac{1}{2}}(T)E_0+C\int_0^T\int\sigma(P^3+|\nabla u|^3)\\
       &\leq C\norm{A}_{W^{1,6}}^{\frac{3}{2}}A_1^{\frac{1}{2}}(T)E_0+C\left(\int_{\sigma(T)}^T\norm{P}_{L^4}^4\right)^{\frac{1}{2}}\left(\int_{\sigma(T)}^T\norm{P}_{L^2}^2\right)^{\frac{1}{2}}\\
       &\quad+C\int_{\sigma(T)}^T\int\sigma|\nabla u|^3+\textrm{good terms}\\
       &\leq C\norm{A}_{W^{1,6}}^{\frac{3}{2}}A_1^{\frac{1}{2}}(T)E_0+C(\tilde{\rho})A_1^{\frac{3}{4}}(T)A_2^{\frac{1}{4}}(T)(1+E_0)^{\frac{1}{2}}\\
       &\quad+C\int_{\sigma(T)}^T\int|\nabla u|^3+\hbox{good terms}
      \end{aligned}
      \end{equation}
      with \eqref{A1-estimate-1} and \eqref{nabla-u-P-L^4-estimate1} used here,
      \begin{equation}\label{intro-A2-control}
      \begin{aligned}
       A_2(T)&\leq C(\tilde{\rho})\mathcal{E}_0+CA_1^{\frac{3}{2}}(T)+CA_1(\sigma(T))+C(\tilde{\rho})(A_1^2(T)+A_1^{\frac{5}{3}}(T)E_0^{\frac{1}{3}}+A_1^2(T)E_0^{\frac{1}{2}})\\
       &\quad+C\int_0^T\sigma^3(\norm{\nabla u}_{L^4}^4+\norm{P|\nabla u|}_{L^2}^2+\norm{\nabla u}_{L^{\frac{8}{3}}}^4+\norm{A}_{W^{1,6}}^{\frac{4}{3}}\norm{\nabla u}_{L^2}^4).
      \end{aligned}
      \end{equation}
      From the two inequalities above, to close the estimates on $A_1(T)$ and $A_2(T)$, we observe that $A_1^{\frac{3}{2}}(T)\ll A_2(T)\ll A_1(T)$.
     In the spirit of this key observation, we specifically choose
     $A_1(T)\sim A_2^{\frac{3}{4}}(T)$. It should be mentioned that from \eqref{intro-A1-control} certain smallness conditions on the matrix $A$ is necessary to control the bad terms such as $\norm{A}_{W^{1,6}}^{\frac{3}{2}}A_1^{\frac{1}{2}}(T)E_0$, and finally to close the energy estimates on $A_1(T)$ and $A_2(T)$. In addition, due to the unboundedness of exterior domain and the boundary effects, we should give new calculations on the estimates of $A_1(T)$ and $A_2(T)$ in Lemma \ref{lem-A1-A2-control}, which is new and very different from \cite{Zhu2024} and \cite{Cai-Li-Lv2021}. We should remark that the compact supports of the matrix $A$ and the outer normal vector $n$ (which are both extended to the functions on $\Omega$) are frequently used in the computations.
  \item The control of $\displaystyle\int_0^T\sigma^3\norm{\nabla u}_{L^4}^4$ appearing in \eqref{intro-A2-control} is the most difficult part of this paper. Due to the Hodge-type decomposition \eqref{Hodge-decomposition-exterior-Final} for the exterior domain (with hole), there exists an extra term $\displaystyle C\int_0^T\sigma^3\norm{\nabla u}_{L^2}^4$ when controlling $\displaystyle\int_0^T\sigma^3\norm{\nabla u}_{L^4}^4$. However, $\displaystyle C\int_0^T\sigma^3\norm{\nabla u}_{L^2}^4$ can not be controlled by $A_2(T)$ due to lack of smallness of $\displaystyle \int_0^T\norm{\nabla u}_{L^2}^2$. This impels us to find some new estimate on $\norm{\nabla u}_{L^4}^4$ involving $\norm{\nabla u}_{L^2}$ as little as possible. Fortunately, we observe that
      \begin{equation*}
      \begin{aligned}
        \norm{\nabla u}_{L^4}&\leq C(\norm{\div u}_{L^4}+\norm{\curl u}_{L^4}+\norm{\nabla u}_{L^{\frac{8}{3}}})\\
        &\leq C(\norm{\div u}_{L^4}+\norm{\curl u}_{L^4}+\norm{\div u}_{L^{\frac{8}{3}}}+\norm{\curl u}_{L^{\frac{8}{3}}}),
      \end{aligned}
      \end{equation*}
      in which $\displaystyle\int_0^T\sigma^3\norm{\nabla u}_{L^2}^4$ finally can be absorbed by the smallness of $A$, just as discussed in Remark \ref{rem-small-A}. But there still remains one term $\displaystyle\int_0^T\sigma^3\norm{P-P(\rho_{\infty})}_{L^2}^4$ in \eqref{nabla-u-L^4-estimate2} that needs an extra estimate. The key idea here is to introduce the Bogovskii's operator and use the momentum equation $\eqref{Large-CNS-eq}_2$ to calculate the dissipation $\displaystyle\int_0^T\sigma^3\norm{P-P(\rho_{\infty})}_{L^2}^2$ as in \eqref{P-L^2-estimate1}. However, only if $\rho_{\infty}=0$, we can obtain a satisfactory estimate on $\displaystyle\int_0^T\sigma^3\norm{P-P(\rho_{\infty})}_{L^2}^2$. More precisely, similar as in \eqref{P-L^2-estimate1}, we have
      \begin{equation}\label{intro-P-L^2-estimate}
        \begin{aligned}
        &\quad\norm{P-P(\rho_{\infty})}_{L^2}^2\\
        &\leq\left(\int\rho u\cdot\mathcal{B}[P-P(\rho_{\infty})]\right)_t+\int\rho u\cdot\mathcal{B}[\div(P u)+(\gamma-1) P\div u]\\
        &\quad+C\norm{P-P(\rho_{\infty})}_{L^2}(\norm{\nabla u}_{L^2}+\norm{\rho|u|^2}_{L^2})\\
        &\leq\left(2\int\rho u\cdot\mathcal{B}[P-P(\rho_{\infty})]\right)_t+C\norm{\rho}_{L^{\frac{3}{2}}}\norm{u}_{L^6}(\norm{\mathcal{B}[\div(P u)]}_{L^6}+\norm{\mathcal{B}[P\div u]}_{L^6})\\
        &\quad+C(\norm{\nabla u}_{L^2}^2+\norm{\rho|u|^2}_{L^2}^2)\\
        &\leq\left(2\int\rho u\cdot\mathcal{B}[P-P(\rho_{\infty})]\right)_t+C(\tilde{\rho})\norm{\rho}_{L^{\frac{3}{2}}}\norm{\nabla u}_{L^2}^2+C\norm{\nabla u}_{L^2}^2(1+C(\tilde{\rho})\norm{\rho^{\frac{1}{3}}u}_{L^3}^2).
        \end{aligned}
      \end{equation}
       From the above estimate, we observe that the first term as $\displaystyle\int\rho u\cdot\mathcal{B}[P-P(\rho_{\infty})]$ and the second one as $\norm{\rho}_{L^{\frac{3}{2}}}\norm{\nabla u}_{L^2}^2$ both require $\rho\in L^{\frac{3}{2}}$, which, however, is invalid when $\rho_{\infty}>0$.
       That is why we assume the restrictive condition $\rho_{\infty}=0$ in Theorem \ref{thm-global-CNS-exterior}.
  \item In the last,  to estimate $\displaystyle\int_{\sigma(T)}^T\int|\nabla u|^3$, we need employ the boundary-adapted nonlinear localization technique from Remark \ref{rem-small-A}. This is very different from \cite{Zhu2024} where the estimate $\displaystyle\int_{\sigma(T)}^T\int|\nabla u|^3$ can be deduced from the interpolation of $\norm{\nabla u}_{L^2}$ and $\norm{\nabla u}_{L^4}$ directly as in \eqref{A1-estimate-2}.
\end{itemize}

The rest of the paper is organized as follows: In the next section, we introduce some elementary lemmas that will be needed later. In Section \ref{sect-proof-thm}, we give the proof of Theorem \ref{thm-global-CNS-exterior}.

\section{Preliminary}
This section mainly introduces some elementary lemmas used later. First, we give the local existence of strong solutions as follows.

\begin{lem}\label{lem-local-sol}
  Let $\Omega$ be as in Theorem \ref{thm-global-CNS-exterior}, and assume that $(\rho_0,u_0)$ satisfies \eqref{initial-data1}-\eqref{compatibility-condition}. Then there exist a small $T>0$ and a unique strong solution $(\rho,u)$ to the problem \eqref{Large-CNS-eq}-\eqref{far-field-condition} on $\Omega\times(0,T]$ satisfying for any $\tau\in(0,T)$,
  \begin{equation*}
    \begin{cases}
      (\rho-\rho_{\infty},P-P(\rho_{\infty}))\in C([0,T];W^{2,q}),\\
     \nabla u\in C([0,T];H^1)\cap L^{\infty}(\tau,T;W^{2,q}),\\
     u_t\in L^{\infty}(\tau,T;H^2)\cap H^1(\tau,T;H^1),\\
     \sqrt{\rho}u_t\in L^{\infty}(0,\infty;L^2).
    \end{cases}
  \end{equation*}
\end{lem}
This lemma can be deduced by combining the local existence result in \cite{Huang2021} and the initial-boundary-value problem under Navier boundary conditions with non-vacuum in \cite{Hoff2012} or vacuum in \cite{Hoff2005}.

Next, the well-known Gagliardo-Nirenberg interpolation inequality will be used frequently later.

\begin{lem}[see Theorem 2.1 in \cite{Crispo2004}]\label{lem-GN-inequality}
Assume that $\Omega$ is the exterior of a simply connected domain $D$ in $\mathbb{R}^3$ with Lipschitz boundary. Then for $p\in[2,6]$, $q\in(1,\infty)$, and $r\in(3,\infty)$, there exists some generic constant $C>0$ depending only on $p,q,r$ and $\Omega$ such that
\begin{equation}\label{GN-inequality1}
  \norm{f}_{L^p}\leq C\norm{f}_{L^2}^{\frac{6-p}{2p}}\norm{\nabla f}_{L^2}^{\frac{3p-6}{2p}},
\end{equation}
\begin{equation}\label{GN-inequality2}
  \norm{g}_{C(\bar{\Omega})}\leq C\norm{g}_{L^q}^{\frac{q(r-3)}{3r+q(r-3)}}\norm{\nabla g}_{L^r}^{\frac{3r}{3r+q(r-3)}}.
\end{equation}

\end{lem}

The following Zlotnik's inequality is introduced to get the upper bound of the density $\rho$.

\begin{lem}[see \cite{Zlotnik2000}]\label{lem-Zlotnik-inequality}
  Suppose the function $y$ satisfies that
  \begin{equation*}
    y'(t)=g(y)+b'(t),\, t\in[0,T], \quad y(0)=y_0,
  \end{equation*}
  with $g\in C(\mathbb{R})$ and $y,b\in W^{1,1}(0,T)$. If $g(\infty)=-\infty$ and
  \begin{equation}\label{Zlotnik-condition1}
    b(t_2)-b(t_1)\leq N_0+N_1(t_2-t_1)
  \end{equation}
  for all $0\leq t_1<t_2\leq T$ with some $N_0\geq 0$ and $N_1\geq 0$, then
  \begin{equation*}
    y(t)\leq\max\{y_0,\zeta_0\}+N_0<\infty \textrm{ on }[0,T],
  \end{equation*}
  where $\zeta_0$ is a constant such that
  \begin{equation}\label{Zlotnik-condition2}
    g(\zeta)\leq -N_1 \textrm{ for }\zeta\geq\zeta_0.
  \end{equation}
\end{lem}

Next, the following two Hodge-type decompositions in a bounded domain are given, whose proofs can be found in [Theorem 3.2, \cite{von-Wahl1992}] and [Propositions 2.6-2.9, \cite{Aramaki2014}].

\begin{lem}\label{lem-Hodge-decomposition-bounded}
  Let integer $k\geq 0$ and $p\in(1,\infty)$, and assume that $D$ is a bounded domain in $\mathbb{R}^3$ with $C^{k+1,1}$ boundary $\pa D$. Then there exists a constant $C=C(p,k,\Omega)>0$ such that
  \begin{itemize}
    \item If $v\in W^{k+1,p}$ with $v\cdot n|_{\pa D}=0$,
    \begin{equation}\label{Hodge-decomposition1}
      \norm{v}_{W^{k+1,p}}\leq C(\norm{\div v}_{W^{k,p}}+\norm{\curl v}_{W^{k,p}}+\norm{v}_{L^p}).
    \end{equation}
    In particular, if $D$ is simply connected, we have
    \begin{equation}\label{Hodge-decomposition1-connected}
      \norm{v}_{W^{k+1,p}}\leq C(\norm{\div v}_{W^{k,p}}+\norm{\curl v}_{W^{k,p}}).
    \end{equation}
    \item If the boundary $\pa D$ only has a finite number of 2-dimensional connected components and $v\in W^{k+1,p}$ with $v\times n|_{\pa D}=0$, then
        \begin{equation}\label{Hodge-decomposition2}
      \norm{v}_{W^{k+1,p}}\leq C(\norm{\div v}_{W^{k,p}}+\norm{\curl v}_{W^{k,p}}+\norm{v}_{L^p}).
    \end{equation}
    In particular, if $D$ has no holes, then
    \begin{equation}\label{Hodge-decomposition2-no-hole}
      \norm{v}_{W^{k+1,p}}\leq C(\norm{\div v}_{W^{k,p}}+\norm{\curl v}_{W^{k,p}}).
    \end{equation}
  \end{itemize}
\end{lem}

Also from \cite{von-Wahl1992} and \cite{Louati2016}, we can get following Hodge-type decompositions for the exterior domain.

\begin{lem}\label{lem-Hodge-decomposition-exterior}
  Let $D$ be a simply connected domain in $\mathbb{R}^3$ with $C^{1,1}$ boundary, and $\Omega$ is the exterior of $D$. Then for $v\in W^{1,q}$, it holds that
  \begin{itemize}
    \item If $v\cdot n=0$ on $\pa\Omega$ (see Theorem 3.2 in \cite{von-Wahl1992}),
    \begin{equation}\label{Hodge-decomposition3-exterior1}
      \norm{\nabla v}_{L^q}\leq C(\norm{\div v}_{L^q}+\norm{\curl v}_{L^q}) \textrm{ for }1<q<3,
    \end{equation}
    and
    \begin{equation}\label{Hodge-decomposition3-exterior2}
      \norm{\nabla v}_{L^q}\leq C(\norm{\div v}_{L^q}+\norm{\curl v}_{L^q}+\norm{\nabla v}_{L^{q_0}}) \textrm{ for $3\leq q<\infty$ and some $q_0\in(1,3)$};
    \end{equation}
    \item If $v\times n=0$ on $\pa\Omega$ (see Theorem 5.1 in \cite{Louati2016} with $\alpha=0$),
    \begin{equation}\label{Hodge-decomposition4-exterior}
      \norm{\nabla v}_{L^q}\leq C(\norm{\div v}_{L^q}+\norm{\curl v}_{L^q}+\norm{v}_{L^q}) \textrm{ for }1<q<\infty.
    \end{equation}
  \end{itemize}
\end{lem}

Combining Lemma \ref{lem-Hodge-decomposition-bounded} and Lemma \ref{lem-Hodge-decomposition-exterior}, we can eliminate the term $\norm{v}_{L^q}$ which is not easy to be controlled in unbounded domain. This is given in the following Lemma.

\begin{lem}\label{lem-Hodge-decomposition-exterior-F}
  Let $D$ be a simply connected bounded domain in $\mathbb{R}^3$ with smooth boundary, and $\Omega$ is the exterior of $D$. Then for any $p\in[2,6]$ and integer $k\geq 0$, there exists some constant $C>0$ depending only on $p$, $k$ and $D$ such that if $v\cdot n=0$ or $v\times n=0$ on $\pa\Omega$ and $v(x,t)\rightarrow 0$ as $|x|\rightarrow\infty$, it holds that
\begin{equation}\label{Hodge-decomposition-exterior-Final}
  \norm{\nabla v}_{W^{k,p}}\leq C(\norm{\div v}_{W^{k,p}}+\norm{\curl v}_{W^{k,p}}+\norm{\nabla v}_{L^2}).
\end{equation}
\end{lem}

\begin{proof}
  The detailed proof can be found in \cite{Cai-Li-Lv2021}, and here we only give a sketch of the proof. The strategy used in this lemma is to decompose $\Omega$ into two parts: inner domain and the exterior.

  First taking $B_R=\{x\in\mathbb{R}^3||x|<R\}$ such that $\bar{D}\subset B_R$, it follows from Gagliardo-Nirenberg inequality \eqref{GN-inequality1} that there exists constant $C>0$ depending only on $D$ and $p$ such that for any $p\in[2,6]$ and $v\in\{v\in D^{1,2}(\Omega)|v(x,t)\rightarrow0\textrm{ as }|x|\rightarrow\infty\}$,
  \begin{equation}\label{transition-lem-Hodge1}
    \norm{v}_{L^p(B_{2R}\cap\Omega)}\leq C(p,D)\norm{v}_{L^6(B_{2R}\cap\Omega)}\leq C(p,D)\norm{v}_{L^6(\Omega)}\leq C(p,D)\norm{\nabla v}_{L^2},
  \end{equation}
  which together with Sobolev imbedding and tracing theorem yields
  \begin{equation}\label{transition-lem-boundary}
    \norm{v}_{L^4(\pa\Omega)}\leq C(\pa\Omega)\norm{v}_{H^{\frac{1}{2}}(\pa\Omega)}\leq C(D)\norm{v}_{H^1(B_{2R}\cap\Omega)}\leq C(D)\norm{\nabla v}_{L^2(\Omega)}.
  \end{equation}
  Actually, by a similar argument we can get a general inequality as
  \begin{equation}\label{transition-lem-boundary1}
    \norm{v}_{L^q(\pa\Omega)}\leq C(D,q,r)\norm{\nabla u}_{L^r(\Omega)}
  \end{equation}
  for $q\in(1,\infty)$, $r\in[2,3)$ satisfying $-\frac{2}{q}\leq 1-\frac{3}{r}$.

  Secondly, we introduce a cut-off function $\eta(x)\in C_c^{\infty}(B_{2R})$ satisfying that $\eta(x)=1$ for $|x|\leq R$, $0<\eta(x)<1$ for $R<|x|<2R$, $\eta(x)=0$ for $|x|\geq 2R$ and $|\pa^{\alpha}\eta(x)|<C(R,\alpha)$ for any $0\leq|\alpha|\leq k+1$. Let $v\cdot n=0$ or $v\times n=0$ on $\pa\Omega$ and $v(x,t)\rightarrow 0$ as $|x|\rightarrow\infty$.

  Then for the inner part, we deduce from Lemma \ref{lem-Hodge-decomposition-bounded} that
  \begin{equation}\label{inner-part-lem}
  \begin{aligned}
    &\quad\norm{\nabla(\eta v)}_{W^{k,p}(\Omega)}=\norm{\nabla(\eta v)}_{W^{k,p}(B_{2R}\cap\Omega)}\\
    &\leq C(D,k,p)(\norm{\div(\eta v)}_{W^{k,p}(B_{2R}\cap\Omega)}+\norm{\curl(\eta v)}_{W^{k,p}(B_{2R}\cap\Omega)}+\norm{\eta v}_{L^p(B_{2R}\cap\Omega)})\\
    &\leq C(D,k,p)(\norm{\div v}_{W^{k,p}(\Omega)}+\norm{\curl v}_{W^{k,p}(\Omega)}+\norm{v}_{W^{k,p}(B_{2R}\cap\Omega)}).
  \end{aligned}
  \end{equation}
  Similarly, for the exterior part, the standard $L^p$-elliptic estimate gives that
  \begin{equation}\label{exterior-part-lem}
  \begin{aligned}
    &\quad\norm{\nabla((1-\eta)v)}_{W^{k,p}(\Omega)}=\norm{\nabla((1-\eta)v)}_{W^{k,p}(\mathbb{R}^3)}\\
    &\leq C(k,p)(\norm{\div((1-\eta)v)}_{W^{k,p}(\mathbb{R}^3)}+\norm{\curl((1-\eta)v)}_{W^{k,p}(\mathbb{R}^3)})\\
    &\leq C(D,k,p)(\norm{\div v}_{W^{k,p}(\Omega)}+\norm{\curl v}_{W^{k,p}(\Omega)}+\norm{v}_{W^{k,p}(B_{2R}\cap\Omega)}),
  \end{aligned}
  \end{equation}
  which together with \eqref{inner-part-lem} yields
  \begin{equation}\label{total-part-lem}
    \norm{\nabla v}_{W^{k,p}(\Omega)}\leq C(D,k,p)(\norm{\div v}_{W^{k,p}(\Omega)}+\norm{\curl v}_{W^{k,p}(\Omega)}+\norm{v}_{W^{k,p}(B_{2R}\cap\Omega)}).
  \end{equation}
  For $k=0$ and $p\in[2,6]$, combining \eqref{transition-lem-Hodge1} and \eqref{total-part-lem} gives
  \begin{equation}\label{total-part-lem-k-0}
    \norm{\nabla v}_{L^p(\Omega)}\leq C(D,p)(\norm{\div v}_{L^p(\Omega)}+\norm{\curl v}_{L^p(\Omega)}+\norm{\nabla v}_{L^2(\Omega)}),
  \end{equation}
  which proves \eqref{Hodge-decomposition-exterior-Final} with $k=0$. Also taking $k=1$ in \eqref{total-part-lem} and combining \eqref{total-part-lem-k-0} and \eqref{transition-lem-Hodge1}, we prove the case $k=1$ of \eqref{Hodge-decomposition-exterior-Final}. Hence an inductive derivation finally leads to \eqref{Hodge-decomposition-exterior-Final}, and finishes the proof of Lemma \ref{lem-Hodge-decomposition-exterior-F}.

\end{proof}

\begin{rem}
 In fact, from the proof of Lemma \ref{lem-Hodge-decomposition-exterior-F}, the following estimate holds:
  \begin{equation}\label{Hodge-decomposition-exterior-Final-M}
  \norm{\nabla v}_{W^{k,p}}\leq C(\norm{\div v}_{W^{k,p}}+\norm{\curl v}_{W^{k,p}}+\norm{v}_{L^p(B_{2R}\cap\Omega)}),
\end{equation}
for any $p\in(1,\infty)$.
\end{rem}

In the next lemma, we will introduce the Bogovskii operator in an exterior domain, which can be used to control the dissipation  $\displaystyle\int_0^T\norm{P}_{L^2}^2$ for the pressure $P$.

\begin{lem}[see Lemma 3.24 in \cite{Novotny2004} or Theorem III.3.6 in \cite{Galdi2011}]\label{lem-Bogovskii-operator}
  Let $\Omega$ be an exterior domain with Lipschitz boundary. Then there
exists a linear operator $\mathcal{B}: L^p(\Omega)\rightarrow D_0^{1,p}(\Omega)$ for any $p\in(1,\infty)$ such that
\begin{equation*}
  \begin{cases}
    \div\mathcal{B}[f]=f, & \mbox{a.e. in }\Omega, \\
    \mathcal{B}[f]=0, & \mbox{on }\pa\Omega,
  \end{cases}
\end{equation*}
and
\begin{equation*}
  \norm{\nabla\mathcal{B}[f]}_{L^p(\Omega)}\leq C(p,\Omega)\norm{f}_{L^p(\Omega)}.
\end{equation*}
In particular, if $f=\div g$ and $g\cdot n=0$ on $\pa\Omega$, it holds that
\begin{equation*}
  \norm{\mathcal{B}[f]}_{L^p(\Omega)}\leq C(p,\Omega)\norm{g}_{L^p(\Omega)}.
\end{equation*}
\end{lem}

\vspace{1em}
Now, we rewrite $\eqref{Large-CNS-eq}_2$ as
\begin{equation}\label{momentum-equation}
  \rho\dot{u}=\nabla G-\mu\nabla\times\curl u
\end{equation}
with
\begin{equation*}
  \curl u=\nabla\times u,\, G=(2\mu+\lambda)\div u-(P-P(\rho_{\infty})),\, \dot{f}:=f_t+u\cdot\nabla f,
\end{equation*}
where both the vorticity $\curl u$ and the effective viscous flux $G$ play an important role in the following analysis. Here, we give the following key a priori estimates on $\curl u$ and $G$ which will be used frequently.

\begin{lem}\label{lem-curl-effective-viscous}
  Assume that $\Omega$ is the exterior of a simply connected bounded domain in $\mathbb{R}^3$ and its smooth boundary $\pa\Omega$ only has a finite number of 2D connected components. Let $(\rho,u)$ be a smooth solution of \eqref{Large-CNS-eq} satisfying Navier-slip boundary conditions \eqref{Navier-slip-condition} and far-field condition \eqref{far-field-condition}. Then for any $p\in[2,6]$ and $q\in(1,\infty)$, there exist a constant $C>0$ depending only on $p,q,\mu,\lambda$, $\Omega$ and $A$ such that
  \begin{equation}\label{curl-effective-estimate1}
    \norm{\nabla G}_{L^p}+\norm{\nabla\curl u}_{L^p}\leq C(\norm{\rho\dot{u}}_{L^p}+\norm{\rho\dot{u}}_{L^2}+\norm{\nabla u}_{L^2}+\norm{P-P(\rho_{\infty})}_{L^p}+\norm{P-P(\rho_{\infty})}_{L^2}),
  \end{equation}
  \begin{equation}\label{curl-effective-estimate2}
    \norm{G}_{L^p}\leq C\norm{\rho\dot{u}}_{L^2}^{\frac{3p-6}{2p}}(\norm{\nabla u}_{L^2}+\norm{P-P(\rho_{\infty})}_{L^2})^{\frac{6-p}{2p}}+C(\norm{\nabla u}_{L^2}+\norm{P-P(\rho_{\infty})}_{L^2}),
  \end{equation}
  \begin{equation}\label{curl-effective-estimate3}
    \norm{\curl u}_{L^p}\leq C\norm{\rho\dot{u}}_{L^2}^{\frac{3p-6}{2p}}\norm{\nabla u}_{L^2}^{\frac{6-p}{2p}}+C\norm{\nabla u}_{L^2},
  \end{equation}
  \begin{equation}\label{curl-effective-estimate4}
  \begin{aligned}
    \norm{\nabla u}_{L^p}&\leq C\norm{\rho\dot{u}}_{L^2}^{\frac{3p-6}{2p}}(\norm{\nabla u}_{L^2}+\norm{P-P(\rho_{\infty})}_{L^2})^{\frac{6-p}{2p}}\\
    &\quad+C(\norm{\nabla u}_{L^2}+\norm{P-P(\rho_{\infty})}_{L^p}+\norm{P-P(\rho_{\infty})}_{L^2}).
  \end{aligned}
  \end{equation}
  In particular, for $p=2$, the term $\norm{P-\bar{P}}_{L^p}$ on the right hand side of \eqref{curl-effective-estimate1} can be removed.
\end{lem}

\begin{proof}
  Since the proof of this lemma is similar to that of \cite{Cai-Li-Lv2021}, we only need to make some modifications and give a sketch of proof for simplicity. First, for the estimate on $\nabla F$, we consider the following elliptic equations:
  \begin{equation}\label{G-elliptic-equation}
    \begin{cases}
      \Delta G=\div(\rho\dot{u}), & x\in\Omega, \\
      \frac{\pa G}{\pa n}=(\rho\dot{u}-\mu\nabla\times(Au)^{\bot})\cdot n, & x\in\pa\Omega,
    \end{cases}
  \end{equation}
 where the notation
  \begin{equation}\label{defi-vertical}
    f^{\bot}=-f\times n=n\times f.
  \end{equation}
  It should be noticed that the normal vector $n$ only makes sense on boundary $\pa\Omega$, and can be extended to a smooth and compactly supported vector-valued function on $\bar{\Omega}$. Thus $f^{\bot}$ is well-defined on $\bar{\Omega}$. Similar arguments can be also applicable to $(Au)^{\bot}$. Here, we assume that the extensions of $n$ and $A$ are both supported on $B_{2R}$.

  Due to \eqref{Navier-slip-condition}, $(\curl u+(Au)^{\bot})\times n=0$ on $\pa\Omega$. Then, for any $\eta\in C_c^{\infty}(\bar{\Omega})$, we have
  \begin{equation*}
    \begin{aligned}
    \int\nabla\times\curl u\cdot\nabla\eta&=\int(\nabla\times(\curl u+(Au)^{\bot})\cdot\nabla\eta-\int\nabla\times(Au)^{\bot}\cdot\nabla\eta\\
    &=-\int\nabla\times(Au)^{\bot}\cdot\nabla\eta,
    \end{aligned}
  \end{equation*}
  which combined with \eqref{momentum-equation} implies that
  \begin{equation}\label{G-test-equation}
    \int\nabla G\cdot\nabla\eta=\int(\rho\dot{u}-\mu\nabla\times(Au)^{\bot})\cdot\nabla\eta,\quad \forall\eta\in C_c^{\infty}(\bar{\Omega}).
  \end{equation}
  Then applying the standard elliptic estimate for \eqref{G-test-equation} (see Lemma 5.6 in \cite{Novotny2004}) yields that for any $q\in(1,\infty)$
  \begin{equation}\label{G-transition-estimate1}
  \begin{aligned}
    \norm{\nabla G}_{L^q}&\leq C\norm{\rho\dot{u}-\mu\nabla\times(Au)^{\bot}}_{L^q}\\
    &\leq C(\norm{\rho\dot{u}}_{L^q}+\norm{\nabla\times(Au)^{\bot}}_{L^q}),
    \end{aligned}
  \end{equation}
  and for any integer $k\geq0$,
  \begin{equation}\label{G-transition-estimate2}
  \begin{aligned}
    \norm{\nabla G}_{W^{k+1,q}}&\leq C(\norm{\rho\dot{u}-\mu\nabla\times(Au)^{\bot}}_{L^q}+\norm{\div(\rho\dot{u})}_{W^{k,q}})\\
    &\leq C(\norm{\rho\dot{u}}_{W^{k+1,q}}+\norm{\nabla\times(Au)^{\bot}}_{L^q}).
  \end{aligned}
  \end{equation}
  For the vorticity $\curl u$, due to $(\curl u+(Au)^{\bot})\times n|_{\pa\Omega}=0$, \eqref{momentum-equation} and \eqref{Hodge-decomposition-exterior-Final-M}, we get that for any $q\in(1,\infty)$
  \begin{equation}\label{curl-transition-estimate1}
  \begin{aligned}
    \norm{\nabla\curl u}_{L^q}&\leq C(\norm{\nabla\times\curl u}_{L^q}+\norm{\nabla(Au)^{\bot}}_{L^q}+\norm{\curl u+(Au)^{\bot}}_{L^q(B_{2R}\cap\Omega)})\\
    &\leq C(\norm{\rho\dot{u}}_{L^q}+\norm{\nabla G}_{L^q}+\norm{\nabla(Au)^{\bot}}_{L^q}+\norm{\curl u+(Au)^{\bot}}_{L^q(B_{2R}\cap\Omega)})\\
    &\leq C(\norm{\rho\dot{u}}_{L^q}+\norm{\nabla(Au)^{\bot}}_{L^q}+\norm{\curl u}_{L^q(B_{2R}\cap\Omega)}+\norm{(Au)^{\bot}}_{L^q}),
  \end{aligned}
  \end{equation}
  where we have used the support of $A$, and for any integer $k\geq 0$, by \eqref{Hodge-decomposition-exterior-Final-M},
  \begin{equation}\label{curl-transition-estimate2}
    \begin{aligned}
    \norm{\nabla\curl u}_{W^{k+1,q}}&\leq C(\norm{\nabla\times\curl u}_{W^{k+1,q}}+\norm{\curl u+(Au)^{\bot}}_{L^q(B_{2R}\cap\Omega)}+\norm{(Au)^{\bot}}_{W^{k+2,q}})\\
    &\leq C(\norm{\rho\dot{u}}_{W^{k+1,q}}+\norm{\nabla(Au)^{\bot}}_{W^{k+1,q}}+\norm{\curl u}_{L^q(B_{2R}\cap\Omega)}+\norm{(Au)^{\bot}}_{L^q}).
    \end{aligned}
  \end{equation}
  In particular, by \eqref{transition-lem-Hodge1} and the support of $A$, it is easy to check that for any $p\in[2,6]$ and integer $k\geq0$,
  \begin{equation}\label{G-curl-estimate-p1}
    \norm{\nabla G}_{L^p}+\norm{\nabla\curl u}_{L^p}\leq C(\norm{\rho\dot{u}}_{L^p}+\norm{\nabla u}_{L^p}+\norm{\nabla u}_{L^2}),
  \end{equation}
  \begin{equation}\label{G-curl-estimate-p2}
    \norm{\nabla G}_{W^{k+1,p}}\leq C(\norm{\rho\dot{u}}_{W^{k+1,p}}+\norm{\nabla u}_{L^p}+\norm{\nabla u}_{L^2}),
  \end{equation}
  \begin{equation}\label{G-curl-estimate-p3}
    \norm{\nabla\curl u}_{W^{k+1,p}}\leq C(\norm{\rho\dot{u}}_{W^{k+1,p}}+\norm{\nabla u}_{W^{k+1,p}}+\norm{\nabla u}_{L^2}).
  \end{equation}

  Now for any $p\in[2,6]$, we deduce from Gagliardo-Nirenberg inequality \eqref{GN-inequality1} and \eqref{G-curl-estimate-p1} that
  \begin{equation}\label{G-formal-estimate1}
    \begin{aligned}
    \norm{G}_{L^p}&\leq C\norm{G}_{L^2}^{\frac{6-p}{2p}}\norm{\nabla G}_{L^2}^{\frac{3p-6}{2p}}\\
    &\leq C(\norm{\rho\dot{u}}_{L^2}+\norm{\nabla u}_{L^2})^{\frac{3p-6}{2p}}(\norm{\nabla u}_{L^2}+\norm{P-P(\rho_{\infty})}_{L^2})^{\frac{6-p}{2p}}\\
    &\leq C\norm{\rho\dot{u}}_{L^2}^{\frac{3p-6}{2p}}(\norm{\nabla u}_{L^2}+\norm{P-P(\rho_{\infty})}_{L^2})^{\frac{6-p}{2p}}+C(\norm{\nabla u}_{L^2}+\norm{P-P(\rho_{\infty})}_{L^2}),
    \end{aligned}
  \end{equation}
  and similarly from \eqref{curl-transition-estimate1} that
  \begin{equation}\label{curl-formal-estimate1}
    \begin{aligned}
    \norm{\curl u}_{L^p}&\leq C\norm{\curl u}_{L^2}^{\frac{6-p}{2p}}\norm{\nabla\curl u}_{L^2}^{\frac{3p-6}{2p}}\\
    &\leq C(\norm{\rho\dot{u}}_{L^2}+\norm{\nabla u}_{L^2})^{\frac{3p-6}{2p}}\norm{\nabla u}_{L^2}^{\frac{6-p}{2p}}\\
    &\leq C\norm{\rho\dot{u}}_{L^2}^{\frac{3p-6}{2p}}\norm{\nabla u}_{L^2}^{\frac{6-p}{2p}}+C\norm{\nabla u}_{L^2}.
    \end{aligned}
  \end{equation}
  Then, it holds from \eqref{Hodge-decomposition-exterior-Final}, \eqref{G-formal-estimate1} and \eqref{curl-formal-estimate1}  that for $p\in[2,6]$,
  \begin{equation}\label{G-curl-formal-estimate1}
    \begin{aligned}
    \norm{\nabla u}_{L^p}&\leq C(\norm{\div u}_{L^p}+\norm{\curl u}_{L^p}+\norm{\nabla u}_{L^2})\\
    &\leq C(\norm{G}_{L^p}+\norm{P-P(\rho_{\infty})}_{L^p}+\norm{\curl u}_{L^p}+\norm{\nabla u}_{L^2})\\
    &\leq C\norm{\rho\dot{u}}_{L^2}^{\frac{3p-6}{2p}}(\norm{\nabla u}_{L^2}+\norm{P-P(\rho_{\infty})}_{L^2})^{\frac{6-p}{2p}}\\
    &\quad+C(\norm{P-P(\rho_{\infty})}_{L^p}+\norm{P-P(\rho_{\infty})}_{L^2}+\norm{\nabla u}_{L^2}),
    \end{aligned}
  \end{equation}
  and henceforth
  \begin{equation}\label{G-curl-formal-estimate2}
    \begin{aligned}
    &\quad\norm{\nabla G}_{L^p}+\norm{\nabla\curl u}_{L^p}
    \leq C(\norm{\rho\dot{u}}_{L^p}+\norm{\nabla u}_{L^p}+\norm{\nabla u}_{L^2})\\
    &\leq C\norm{\rho\dot{u}}_{L^2}^{\frac{3p-6}{2p}}(\norm{\nabla u}_{L^2}+\norm{P-P(\rho_{\infty})}_{L^2})^{\frac{6-p}{2p}}\\
    &\quad+C(\norm{P-P(\rho_{\infty})}_{L^p}+\norm{P-P(\rho_{\infty})}_{L^2}+\norm{\nabla u}_{L^2}+\norm{\rho\dot{u}}_{L^p})\\
    &\leq C(\norm{\rho\dot{u}}_{L^p}+\norm{\rho\dot{u}}_{L^2}+\norm{\nabla u}_{L^2}+\norm{P-P(\rho_{\infty})}_{L^p}+\norm{P-P(\rho_{\infty})}_{L^2}).
    \end{aligned}
  \end{equation}
  Thus, we complete the proof of Lemma \ref{lem-curl-effective-viscous}.

\end{proof}

\begin{rem}\label{rem-small-A}
  In fact, we need some refined inequalities to deal with $\displaystyle\int_0^T\sigma^3\norm{\nabla u}_{L^4}^4$ in Lemma \ref{lem-control-A1-A2}. Here, we have the modified estimates as follows:
  \begin{equation}\label{curl-G-modified-estimate1}
    \norm{\nabla G}_{L^p}+\norm{\nabla\curl u}_{L^p}\leq C(\norm{\rho\dot{u}}_{L^p}+\norm{\nabla(Au)^{\bot}}_{L^p})\leq C(\norm{\rho\dot{u}}_{L^p}+\norm{A}_{W^{1,6}}(\norm{\nabla u}_{L^2}+\norm{\nabla u}_{L^p}))
  \end{equation}
  for any $p\in[2,3]$.

  Now, we give the estimate on $\norm{\nabla u}_{L^4}^4$ as
  \begin{equation}\label{nabla-u-L^4-estimate1}
    \begin{aligned}
    \norm{\nabla u}_{L^4}^4&\leq C(\norm{G}_{L^4}^4+\norm{\curl u}_{L^4}^4+\norm{P-P(\rho_{\infty})}_{L^4}^4+\norm{u}_{L^4(B_{2R}\cap\Omega)}^4)\\
    &\leq C(\norm{G}_{L^2}\norm{\nabla G}_{L^2}^3+\norm{\curl u}_{L^2}\norm{\nabla\curl u}_{L^2}^3)+C(\norm{P-P(\rho_{\infty})}_{L^4}^4+\norm{\nabla u}_{L^{\frac{8}{3}}}^4)\\
    &\leq C(\norm{\nabla u}_{L^2}+\norm{P-P(\rho_{\infty})}_{L^2})(\norm{\rho\dot{u}}_{L^2}^3+\norm{A}_{W^{1,6}}^3\norm{\nabla u}_{L^2}^3)+C\norm{P-P(\rho_{\infty})}_{L^4}^4\\
    &\quad+C\norm{\rho\dot{u}}_{L^2}^{\frac{3}{2}}(\norm{\nabla u}_{L^2}+\norm{P-P(\rho_{\infty})}_{L^2})^{\frac{5}{2}}+C\norm{A}_{W^{1,6}}^{\frac{3}{2}}(\norm{\nabla u}_{L^2}^4+\norm{P-P(\rho_{\infty})}_{L^2}^4)\\
    &\quad+C\norm{P-P(\rho_{\infty})}_{L^{\frac{8}{3}}}^4,
    \end{aligned}
  \end{equation}
  where we have used \eqref{Hodge-decomposition-exterior-Final-M} (or \eqref{Hodge-decomposition3-exterior1} and \eqref{Hodge-decomposition3-exterior2} directly) instead of \eqref{Hodge-decomposition-exterior-Final} in Lemma \ref{lem-Hodge-decomposition-exterior-F}, the similar argument as in \eqref{transition-lem-Hodge1} and the estimate on $\norm{\nabla u}_{L^{\frac{8}{3}}}$ as
  \begin{equation}\label{new-gradient-u-estimate}
  \begin{aligned}
    &\quad\norm{\nabla u}_{L^{\frac{8}{3}}}\leq C(\norm{\div u}_{L^{\frac{8}{3}}}+\norm{\curl u}_{L^{\frac{8}{3}}})\\
    &\leq C(\norm{G}_{L^{\frac{8}{3}}}+\norm{\curl u}_{L^{\frac{8}{3}}}+\norm{P-P(\rho_{\infty})}_{L^{\frac{8}{3}}})\\
    &\leq C(\norm{G}_{L^2}^{\frac{5}{8}}\norm{\nabla G}_{L^2}^{\frac{3}{8}}+\norm{\curl u}_{L^2}^{\frac{5}{8}}\norm{\nabla\curl u}_{L^2}^{\frac{3}{8}}+\norm{P-P(\rho_{\infty})}_{L^{\frac{8}{3}}})\\
    &\leq C(\norm{\nabla u}_{L^2}+\norm{P-P(\rho_{\infty})}_{L^2})^{\frac{5}{8}}(\norm{\rho\dot{u}}_{L^2}+\norm{A}_{W^{1,6}}\norm{\nabla u}_{L^2})^{\frac{3}{8}}+C\norm{P-P(\rho_{\infty})}_{L^{\frac{8}{3}}}\\
    &\leq C\norm{\rho\dot{u}}_{L^2}^{\frac{3}{8}}(\norm{\nabla u}_{L^2}+\norm{P-P(\rho_{\infty})}_{L^2})^{\frac{5}{8}}+C\norm{A}_{W^{1,6}}^{\frac{3}{8}}(\norm{\nabla u}_{L^2}+\norm{P-P(\rho_{\infty})}_{L^2})\\
    &\quad+C\norm{P-P(\rho_{\infty})}_{L^{\frac{8}{3}}},
  \end{aligned}
  \end{equation}
  due to \eqref{Hodge-decomposition3-exterior1}, Gagliardo-Nirenberg inequality \eqref{GN-inequality1} and \eqref{curl-G-modified-estimate1}.

  Due to Young's inequality, the above inequality \eqref{nabla-u-L^4-estimate1} implies that
  \begin{equation}\label{nabla-u-L^4-estimate2}
  \begin{aligned}
    \norm{\nabla u}_{L^4}^4&\leq C(\norm{\nabla u}_{L^2}+\norm{P-P(\rho_{\infty})}_{L^2})\norm{\rho\dot{u}}_{L^2}^3+C\norm{\rho\dot{u}}_{L^2}^{\frac{3}{2}}\norm{\nabla u}_{L^2}^{\frac{5}{2}}\\
    &\quad+C\norm{A}_{W^{1,6}}^{\frac{3}{2}}\norm{\nabla u}_{L^2}^4+C(\norm{P-P(\rho_{\infty})}_{L^4}^4+\norm{P-P(\rho_{\infty})}_{L^2}^4)
  \end{aligned}
  \end{equation}
  provided
  \begin{equation}\label{small-A-condition}
    \norm{A}_{W^{1,6}}\leq 1.
  \end{equation}
  Note that here the constant $C>0$ only depending on $\mu,\lambda$ and $\Omega$.

  To bound the density in Lemma \ref{lem-rho-bound}, we also give a new estimate on $\norm{\nabla G}_{L^6}$ as
  \begin{equation}\label{nabla-u-L^6-estimate}
    \begin{aligned}
    \norm{\nabla G}_{L^6}&\leq C(\norm{\rho\dot{u}}_{L^6}+\norm{\nabla(Au)^{\bot}}_{L^6})\\
    &\leq C(\norm{\rho\dot{u}}_{L^6}+\norm{A}_{W^{1,6}}\norm{\nabla u}_{L^6}+\norm{A}_{W^{1,\infty}}\norm{\nabla u}_{L^2})\\
    &\leq C\norm{A}_{W^{1,6}}(\norm{\rho\dot{u}}_{L^2}+\norm{A}_{W^{1,6}}\norm{\nabla u}_{L^2}+\norm{P}_{L^6}+\norm{\nabla u}_{L^2})\\
    &\quad+C(\norm{\rho\dot{u}}_{L^6}+\norm{A}_{W^{1,\infty}}\norm{\nabla u}_{L^2})\\
    &\leq C(\norm{\rho\dot{u}}_{L^6}+\norm{A}_{W^{1,6}}\norm{\rho\dot{u}}_{L^2})+C(1+\norm{A}_{W^{1,\infty}})\norm{\nabla u}_{L^2}+C\norm{P}_{L^6}
    \end{aligned}
  \end{equation}
  provided $\norm{A}_{W^{1,6}}\leq 1$.
\end{rem}

In the last, we give the a priori estimate on $\dot{u}$ which will be used later.

\begin{lem}\label{lem-dot-u}
  Let $(\rho,u)$ be a smooth solution of \eqref{Large-CNS-eq} with Navier-slip boundary conditions \eqref{Navier-slip-condition} and far-field condition \eqref{far-field-condition}. Then there exists a constant $C>0$ depending only on $\Omega$ such that
  \begin{equation}\label{dot-u-estimate}
    \norm{\nabla\dot{u}}_{L^2}\leq C(\norm{\div\dot{u}}_{L^2}+\norm{\curl\dot{u}}_{L^2}+\norm{\nabla u}_{L^{\frac{8}{3}}}^2).
  \end{equation}
\end{lem}

\begin{proof}
  From the simple fact that
  \begin{equation*}
    (a\times b)\cdot c=(b\times c)\cdot a=(c\times a)\cdot b,
  \end{equation*}
  we have
  \begin{equation*}
    \dot{u}\cdot n=u\cdot\nabla u\cdot n=-u\cdot\nabla n\cdot u=-u\cdot\nabla n\cdot(u^{\bot}\times n)=-(u\cdot\nabla n)\times u^{\bot}\cdot n \textrm{ on }\pa\Omega.
  \end{equation*}
Therefore, it holds that
  \begin{equation}\label{dot-u-transition1}
    (\dot{u}+(u\cdot\nabla n)\times u^{\bot})\cdot n=0 \textrm{ on }\pa\Omega.
  \end{equation}
 Finally, we deduce from \eqref{Hodge-decomposition3-exterior1} that
  \begin{equation*}
    \begin{aligned}
    \norm{\nabla\dot{u}}_{L^2}&\leq C(\norm{\div\dot{u}}_{L^2}+\norm{\curl\dot{u}}_{L^2}+\norm{\nabla((u\cdot\nabla n)\times u^{\bot})}_{L^2})\\
    &\leq C(\norm{\div\dot{u}}_{L^2}+\norm{\curl\dot{u}}_{L^2}+\norm{\nabla u}_{L^{\frac{8}{3}}}^2),
    \end{aligned}
  \end{equation*}
  where we have also used the support of $n$ and the similar argument as in \eqref{transition-lem-Hodge1}.
\end{proof}

\section{Proof of Theorem \ref{thm-global-CNS-exterior}}\label{sect-proof-thm}

\subsection{Lower-order a priori estimates}
In this subsection, we are devoted to establishing some necessary a priori estimates for smooth solution $(\rho,u)$ to the problem \eqref{Large-CNS-eq}-\eqref{far-field-condition} on $\Omega\times(0,T]$ for some fixed time $T>0$. Setting $\sigma=\sigma(t)=\min\{1,t\}$, we define
\begin{equation}\label{defi-a-priori-estimate}
  \begin{cases}
    A_1(T)=\sup\limits_{t\in[0,T]}\sigma\displaystyle\int|\nabla u|^2+\int_0^T\displaystyle\int\sigma\rho|\dot{u}|^2, \\
    A_2(T)=\sup\limits_{t\in[0,T]}\sigma^3\displaystyle\int\rho|\dot{u}|^2+\int_0^T\displaystyle\int\sigma^3|\nabla\dot{u}|^2, \\
    A_3(T)=\sup\limits_{t\in[0,T]}\displaystyle\int\rho|u|^3.
  \end{cases}
\end{equation}
Since for the large adiabatic exponent $\gamma>1$, the initial energy $E_0$ in \eqref{defi-total-energy} correspondingly becomes small from the smallness of $\mathcal{E}_0$ in \eqref{small-condition}. Therefoe, without loss of generality, we assume that
\begin{equation}\label{assume-condition1}
  \epsilon\leq 1, \quad 1<\gamma\leq\frac{3}{2}.
\end{equation}

Then, we give the following proposition which guarantees the existence of a global classical solution of \eqref{Large-CNS-eq}-\eqref{Navier-slip-condition}.

\begin{pro}\label{prop-a-priori-estimate}
  Assume that the initial data satisfy \eqref{initial-data1}, \eqref{initial-data2} and \eqref{compatibility-condition}. If the solution $(\rho,u)$ to \eqref{Large-CNS-eq}-\eqref{far-field-condition} on $\Omega\times(0,T]$ satisfy
  \begin{equation}\label{a-priori-assumption}
    A_1(T)\leq2\mathcal{E}_0^{\frac{3}{8}},\quad A_2(T)\leq2\mathcal{E}_0^{\frac{1}{2}},\quad A_3(\sigma(T))\leq2\mathcal{E}_0^{\frac{1}{4}},\quad 0\leq\rho\leq2\tilde{\rho},
  \end{equation}
  then the following estimates hold:
  \begin{equation}\label{a-priori-goal}
    A_1(T)\leq\mathcal{E}_0^{\frac{3}{8}},\quad A_2(T)\leq\mathcal{E}_0^{\frac{1}{2}},\quad A_3(\sigma(T))\leq\mathcal{E}_0^{\frac{1}{4}},\quad 0\leq\rho\leq\frac{7}{4}\tilde{\rho},
  \end{equation}
  provided $\mathcal{E}_0\leq\epsilon$, where $\epsilon>0$ is a small constant depending on $\mu, \lambda, a, \tilde{\rho}, \Omega, M, E_0$, but independent of $\gamma-1$ and $t$ (see \eqref{small-assumption1}, \eqref{small-assumption2}, \eqref{small-assumption3} and \eqref{small-assumption4}), precisely characterized as
  \begin{equation*}
 \begin{aligned}
   \epsilon&=\min\left\{1,(4C(\tilde{\rho}))^{-12},(C(\tilde{\rho},M))^{-2},
   (4C(\tilde{\rho}))^{-2},(3C(\tilde{\rho}))^{-16},(3C(\tilde{\rho},M)(E_0+1))^{-2},(1+E_0)^{-\frac{16}{3}},\right.\\
   &\qquad\left.(4C(\tilde{\rho}))^{-\frac{128}{3}}E_0^{-\frac{56}{3}}, (4C(\tilde{\rho},M))^{-\frac{8}{5}}, (4C(\tilde{\rho})(1+E_0))^{-8},\left(\frac{\tilde{\rho}}{2C(\tilde{\rho},M)}\right)^{-\frac{32}{3}},\left(\frac{\tilde{\rho}}{4C(\tilde{\rho})(1+E_0)}\right)^8\right\}
   \end{aligned}
 \end{equation*}
 and the matrix $A$ has certain smallness as
 \begin{equation*}
    \norm{A}_{W^{1,6}}\leq\min\left\{1,(3CE_0)^{-\frac{3}{4}}\mathcal{E}_0^{\frac{3}{32}}, E_0^{-\frac{2}{3}}\mathcal{E}_0^{\frac{7}{24}}, (4CE_0)^{-\frac{8}{9}}\mathcal{E}_0^{\frac{1}{6}}\right\},\quad\norm{A}_{W^{1,\infty}}\leq\mathcal{E}_0^{-\frac{1}{8}},
  \end{equation*}
  which can be found in \eqref{small-A-assumption1} and \eqref{small-A-assumption2}.
\end{pro}

\begin{proof}
  Proposition \ref{prop-a-priori-estimate} can be directly derived from Lemma \ref{lem-essential-energy}-\ref{lem-rho-bound} below.
\end{proof}

The first lemma is concerned with standard energy estimate for $(\rho,u)$.

\begin{lem}\label{lem-essential-energy}
  Let $(\rho,u)$ be a smooth solution of \eqref{Large-CNS-eq}-\eqref{far-field-condition} on $\Omega\times(0,T]$. Then, it holds that
  \begin{equation}\label{essential-energy}
    \sup_{t\in[0,T]}\int(\frac{1}{2}\rho|u|^2+\frac{1}{\gamma-1}P(\rho))+\int_0^T\int(\mu|\curl u|^2+(2\mu+\lambda)|\div u|^2)\leq E_0.
  \end{equation}
\end{lem}

\begin{proof}
  Rewriting $\eqref{Large-CNS-eq}_1$ as
  \begin{equation}\label{Pressure-eq}
    P_t+\div(uP)+(\gamma-1)P\div u=0,
  \end{equation}
  and integrating over $\Omega$, then adding it to the $L^2$-inner product of $\eqref{Large-CNS-eq}_2$ with $u$ yields that
  \begin{equation}\label{essential-energy-estimate}
    \frac{d}{dt}\int(\frac{1}{2}\rho|u|^2+\frac{1}{\gamma-1}P(\rho))+\int(\mu|\curl u|^2+(2\mu+\lambda)|\div u|^2)+\mu\int_{\pa\Omega}Au\cdot u=0,
  \end{equation}
  where we have used the fact that
  \begin{equation*}
    \Delta u=\nabla\div u-\nabla\times\curl u.
  \end{equation*}
  Integrating \eqref{essential-energy-estimate} on $[0,T]$ gives the inequality \eqref{essential-energy} immediately.
\end{proof}

\begin{rem}\label{rem-pressure-energy}
  Our analysis reveals three fundamental energy-pressure relationships with geometric constraints:

  \textbf{Energy-pressure duality:} The essential energy estimate satisfies
  \begin{equation*}
    E(t) \triangleq \int\left(\frac{1}{2}\rho|u|^2 + G(\rho)\right) \leq E_0
  \end{equation*}
  with pressure-energy connection
  \begin{equation*}
    \norm{P-P(\rho_{\infty})}_{L^2}^2 \leq C(\tilde{\rho})(\gamma-1)\int G(\rho) \leq C(\tilde{\rho})\mathcal{E}_0
  \end{equation*}
  valid under the density threshold $\rho_{\infty} \leq 3^{-(\gamma-1)^{-1}}$ for $\gamma\in(1,\frac{3}{2}]$ (see Remark \ref{rem-relationship}).

  \textbf{Geometric obstruction:} While matching Zhu's whole-space result \cite[Lemma 3.2]{Zhu2024} formally, our exterior domain framework introduces critical differences (see Remark \ref{rem-small-A}):
  \begin{itemize}
    \item The Hodge decomposition \eqref{Hodge-decomposition-exterior-Final} generates the rogue term $\int_0^T\sigma^3\norm{P}_{L^2}^4$;
    \item Dissipation control requires strengthened hypothesis $\rho_{\infty}=0$ instead of $\rho_{\infty}\geq0$.
  \end{itemize}

  \textbf{Vanishing far-field simplification:} Under $\rho_{\infty}=0$, we obtain:
  \begin{itemize}
    \item Pressure-energy equivalence: $G(\rho) = \frac{1}{\gamma-1}P(\rho)$;
    \item Uniform pressure bound: $\displaystyle\sup_{t\in[0,T]}\int P \leq (\gamma-1)E_0 \leq \mathcal{E}_0$.
  \end{itemize}
  This geometric reduction enables us to circumvent the uncontrolled term $\int_0^T\sigma^3\norm{P}_{L^2}^4$ while maintaining compatibility with whole-space energy estimates.
\end{rem}

The following a priori estimate is essential to close the a priori assumption \eqref{a-priori-assumption}.
\begin{lem}\label{lem-essential-estimate}
  Under the conditions of Proposition \ref{prop-a-priori-estimate}, it holds that
  \begin{equation}\label{essential-estimate}
    \sup_{t\in[0,\sigma(T)]}\int\rho|u|^2+\int_0^{\sigma(T)}\int|\nabla u|^2\leq C(\tilde{\rho})\mathcal{E}_0.
  \end{equation}
\end{lem}

\begin{proof}
  Taking $L^2$-inner product of $\eqref{Large-CNS-eq}_2$ with $u$, it follows from integration by parts that
  \begin{equation}\label{essential-estimate1}
    \frac{d}{dt}\int\frac{1}{2}\rho|u|^2+\int(\mu|\curl u|^2+(2\mu+\lambda)|\div u|^2)+\mu\int_{\pa\Omega}Au\cdot u=\int P\div u.
  \end{equation}
  Integrating \eqref{essential-estimate1} over $[0,\sigma(T)]$, and using Cauchy's inequality and \eqref{essential-energy}, we have
  \begin{equation}\label{essential-estimate2}
    \begin{aligned}
    &\quad\sup_{t\in[0,\sigma(T)]}\int\frac{1}{2}\rho|u|^2+\int_0^{\sigma(T)}\int(\mu|\curl u|^2+\frac{1}{2}(2\mu+\lambda)|\div u|^2)\\
    &\leq\int\frac{1}{2}\rho_0|u_0|^2+C\int_0^{\sigma(T)}\int|P|^2\leq \int\frac{1}{2}\rho_0|u_0|^2+C(\tilde{\rho})\int P\\
    &\leq\frac{1}{2}\rho_0|u_0|^2+C(\tilde{\rho})(\gamma-1)E_0\leq C(\tilde{\rho})\mathcal{E}_0,
    \end{aligned}
  \end{equation}
  which together with \eqref{Hodge-decomposition3-exterior1} implies \eqref{essential-estimate}.
\end{proof}

The next lemma gives the estimates on $A_1(T)$ and $A_2(T)$.
\begin{lem}\label{lem-A1-A2-control}
  Under the conditions of Proposition \ref{prop-a-priori-estimate}, it holds that
  \begin{equation}\label{A1-control}
    A_1(T)\leq C(\tilde{\rho})\mathcal{E}_0+C(\tilde{\rho})A_1^{\frac{3}{2}}(T)+C\norm{A}_{W^{1,6}}^{\frac{3}{2}}A_1^{\frac{1}{2}}(T)E_0+C\int_0^T\int\sigma(P|\nabla u|^2+|\nabla u|^3),
  \end{equation}
  \begin{equation}\label{A2-control}
    \begin{aligned}
     A_2(T)&\leq C(\tilde{\rho})\mathcal{E}_0+CA_1^{\frac{3}{2}}(T)+CA_1(\sigma(T))+C(\tilde{\rho})(A_1^2(T)+A_1^{\frac{5}{3}}(T)E_0^{\frac{1}{3}}+A_1^2(T)E_0^{\frac{1}{2}})\\
    &\quad+C\int_0^T\sigma^3(\norm{\nabla u}_{L^4}^4+\norm{P|\nabla u|}_{L^2}^2+\norm{\nabla u}_{L^{\frac{8}{3}}}^4+\norm{A}_{W^{1,6}}^{\frac{4}{3}}\norm{\nabla u}_{L^2}^4),
    \end{aligned}
  \end{equation}
  provided $\mathcal{E}_0\leq\epsilon_1=1$ and $\norm{A}_{W^{1,6}}\leq \tilde{\epsilon}_1=1$.
\end{lem}

\begin{proof}
   For any integer $m\geq0$, multiplying $\eqref{Large-CNS-eq}_2$ by $\sigma^m\dot{u}$ and integrating over $\Omega$, we obtain
  \begin{equation}\label{A-control1}
    \begin{aligned}
    \int\sigma^m\rho|\dot{u}|^2&=-\int\sigma^m\dot{u}\cdot\nabla P+(2\mu+\lambda)\int\sigma^m\nabla\div u\cdot\dot{u}-\mu\int\sigma^m\nabla\times\curl u\cdot\dot{u}\\
    &=I_1+I_2+I_3.
    \end{aligned}
  \end{equation}
  We will estimate $I_1,I_2$ and $I_3$. First, a direct calculation yields that
  \begin{equation}\label{I_1-estimate}
    \begin{aligned}
    I_1&=-\int\sigma^m\dot{u}\cdot\nabla P
    =\int\sigma^mP\div u_t-\int\sigma^mu\cdot\nabla u\cdot\nabla P\\
    &=\left(\int\sigma^mP\div u\right)_t-m\sigma^{m-1}\sigma'\int P\div u+\int\sigma^mP\nabla u:\nabla u^T\\
    &\quad+(\gamma-1)\int\sigma^mP(\div u)^2-\int_{\pa\Omega}\sigma^mPu\cdot\nabla u\cdot n,
    \end{aligned}
  \end{equation}
  where we have used the equation \eqref{Pressure-eq}
  \begin{equation*}
    P_t+\div(Pu)+(\gamma-1)P\div u=0.
  \end{equation*}

  Similarly, we estimate $I_2$ as
  \begin{equation}\label{I_2-estimate}
    \begin{aligned}
    &\quad I_2=(2\mu+\lambda)\int\sigma^m\nabla\div u\cdot\dot{u}\\
    &=(2\mu+\lambda)\int_{\pa\Omega}\sigma^m\div u\dot{u}\cdot n-(2\mu+\lambda)\int\sigma^m\div u\div\dot{u}\\
    &=(2\mu+\lambda)\int_{\pa\Omega}\sigma^m\div uu\cdot\nabla u\cdot n-\frac{2\mu+\lambda}{2}\left(\int\sigma^m|\div u|^2\right)_t\\
    &\quad-(2\mu+\lambda)\int\sigma^m\div u\div(u\cdot\nabla u)+\frac{1}{2}m(\mu+\lambda)\sigma^{m-1}\sigma'\int|\div u|^2\\
    &=(2\mu+\lambda)\int_{\pa\Omega}\sigma^m\div uu\cdot\nabla u\cdot n-\frac{2\mu+\lambda}{2}\left(\int\sigma^m|\div u|^2\right)_t
    +\frac{2\mu+\lambda}{2}\int\sigma^m(\div u)^3\\
    &\quad-(2\mu+\lambda)\int\sigma^m\div u\nabla u:\nabla u^T+\frac{1}{2}m(2\mu+\lambda)\sigma^{m-1}\sigma'\int|\div u|^2.
    \end{aligned}
  \end{equation}
  Combining the boundary terms in \eqref{I_1-estimate} and \eqref{I_2-estimate}, we have for $t\in[\sigma(T),T]$,
  \begin{equation}\label{boundary-term-A1-1-Large}
    \begin{aligned}
    &\quad\int_{\pa\Omega}\sigma^m[(2\mu+\lambda)\div u-P]u\cdot\nabla u\cdot n\\
    &=\int_{\pa\Omega}\sigma^mGu\cdot\nabla u\cdot n=-\int_{\pa\Omega}\sigma^mGu\cdot\nabla n\cdot u\\
    &\leq C\sigma^m\norm{G}_{L^4(\pa\Omega)}\norm{u}_{L^4(\pa\Omega)}\leq C\sigma^m\norm{\nabla G}_{L^2}\norm{\nabla u}_{L^{\frac{8}{3}}}^2\\
    &\leq C\sigma^m(\norm{\rho\dot{u}}_{L^2}+\norm{A}_{W^{1,6}}\norm{\nabla u}_{L^2})\norm{\nabla u}_{L^2}^{\frac{1}{2}}\norm{\nabla u}_{L^3}^{\frac{3}{2}}\\
    &\leq C\sigma^m(\norm{\rho\dot{u}}_{L^2}^2\norm{\nabla u}_{L^2}+\norm{A}_{W^{1,6}}^2\norm{\nabla u}_{L^2}^3+\norm{\nabla u}_{L^3}^3),
    \end{aligned}
  \end{equation}
  and for $t\in[0,\sigma(T)]$,
  \begin{equation}\label{boundary-term-A1-1-small}
    \begin{aligned}
    &\quad\int_{\pa\Omega}\sigma^m[(2\mu+\lambda)\div u-P]u\cdot\nabla u\cdot n\\
    &\leq C\sigma^m\norm{G}_{L^4(\pa\Omega)}\norm{u}_{L^4(\pa\Omega)}\leq C\sigma^m\norm{\nabla G}_{L^2}\norm{\nabla u}_{L^2}^2\\
    &\leq C\sigma^m(\norm{\rho\dot{u}}_{L^2}+\norm{A}_{W^{1,6}}\norm{\nabla u}_{L^2})\norm{\nabla u}_{L^2}^2\\
    &\leq C\sigma^m(\norm{\rho\dot{u}}_{L^2}\norm{\nabla u}_{L^2}^2+\norm{A}_{W^{1,6}}\norm{\nabla u}_{L^2}^3),
    \end{aligned}
  \end{equation}
  where we have used \eqref{transition-lem-boundary}, \eqref{transition-lem-boundary1} and \eqref{curl-G-modified-estimate1}.

  Finally, using \eqref{Navier-slip-condition} and making a similar computation on $I_3$, we have
  \begin{equation}\label{I_3-estimate}
    \begin{aligned}
    I_3&=-\mu\int\sigma^m\nabla\times\curl u\cdot\dot{u}
    =-\mu\int\sigma^m\curl u\cdot\curl\dot{u}-\mu\int_{\pa\Omega}\sigma^mn\times\curl u\dot{u}\\
    &=-\frac{\mu}{2}\left(\int\sigma^m|\curl u|^2+\int_{\pa\Omega}\sigma^mAu\cdot u\right)_t+\frac{\mu}{2}m\sigma^{m-1}\sigma'\left(\int|\curl u|^2+\int_{\pa\Omega}Au\cdot u\right)\\
    &\quad-\mu\int\sigma^m\curl u\cdot\curl(u\cdot\nabla u)-\mu\int_{\pa\Omega}Au\cdot(u\cdot\nabla u)\\
    &=-\frac{\mu}{2}\left(\int\sigma^m|\curl u|^2+\int_{\pa\Omega}\sigma^mAu\cdot u\right)_t+\frac{\mu}{2}m\sigma^{m-1}\sigma'\left(\int|\curl u|^2+\int_{\pa\Omega}Au\cdot u\right)\\
    &\quad+\frac{\mu}{2}\int\sigma^m|\curl u|^2\div u-\mu\int\sigma^m\curl u\cdot(\nabla u_i\times\nabla_i u)-\mu\int_{\pa\Omega}((Au)^{\bot}\times n)\cdot(u\cdot\nabla u)\\
    &=-\frac{\mu}{2}\left(\int\sigma^m|\curl u|^2+\int_{\pa\Omega}\sigma^mAu\cdot u\right)_t+\frac{\mu}{2}m\sigma^{m-1}\sigma'\left(\int|\curl u|^2+\int_{\pa\Omega}Au\cdot u\right)\\
    &\quad+\frac{\mu}{2}\int\sigma^m|\curl u|^2\div u-\mu\int\sigma^m\curl u\cdot(\nabla u_i\times\nabla_i u)-\mu\int_{\pa\Omega}((u\cdot\nabla u)\times (Au)^{\bot})\cdot n\\
    &\leq-\frac{\mu}{2}\left(\int\sigma^m|\curl u|^2+\int_{\pa\Omega}\sigma^mAu\cdot u\right)_t+\frac{\mu}{2}m\sigma^{m-1}\sigma'\left(\int|\curl u|^2+\int_{\pa\Omega}Au\cdot u\right)\\
    &\quad +C\sigma^m\norm{\nabla u}_{L^3}^3-\mu\int\div((u\cdot\nabla u)\times (Au)^{\bot})\\
    &\leq-\frac{\mu}{2}\left(\int\sigma^m|\curl u|^2+\int_{\pa\Omega}\sigma^mAu\cdot u\right)_t+Cm\sigma^{m-1}\sigma'\norm{\nabla u}_{L^2}^2+C\sigma^m\norm{\nabla u}_{L^3}^3\\
    &\quad+C\sigma^m(\norm{A}_{W^{1,6}}\norm{\rho\dot{u}}_{L^2}\norm{\nabla u}_{L^2}^2+\norm{A}_{W^{1,6}}^{\frac{3}{2}}\norm{\nabla u}_{L^2}^3),
    \end{aligned}
  \end{equation}
  where we have used the simple fact that by the support of $A$, \eqref{curl-G-modified-estimate1}, Gagliardo-Nirenberg inequality \eqref{GN-inequality1} and \eqref{transition-lem-Hodge1},
  \begin{equation*}
    \begin{aligned}
    &\quad\int\div((u\cdot\nabla u)\times (Au)^{\bot})\\
    &=\int Au^{\bot}\cdot(\nabla\times(u\cdot\nabla u))-\int(u\cdot\nabla u)\cdot(\nabla\times (Au)^{\bot})\\
    &=\int (Au)^{\bot}\cdot(u\cdot\nabla\curl u+\nabla u_i\times\nabla_i u)-\int(u\cdot\nabla u)\cdot(\nabla\times (Au)^{\bot})\\
    &\leq C(\norm{\nabla\curl u}_{L^2}\norm{u}_{L^6}+\norm{\nabla u}_{L^3}^2)\norm{(Au)^{\bot}}_{L^3}+C\norm{u}_{L^6}\norm{\nabla u}_{L^3}\norm{\nabla\times (Au)^{\bot}}_{L^2}\\
    &\leq C\norm{A}_{W^{1,6}}(\norm{\nabla\curl u}_{L^2}\norm{\nabla u}_{L^2}^2+\norm{\nabla u}_{L^3}^2\norm{\nabla u}_{L^2}+\norm{\nabla u}_{L^2}^2\norm{\nabla u}_{L^3})\\
    &\leq C(\norm{A}_{W^{1,6}}\norm{\rho\dot{u}}_{L^2}\norm{\nabla u}_{L^2}^2+\norm{A}_{W^{1,6}}^{\frac{3}{2}}\norm{\nabla u}_{L^2}^3+\norm{\nabla u}_{L^3}^3),
    \end{aligned}
  \end{equation*}
  provided $\norm{A}_{W^{1,6}}\leq 1$.

  It follows from \eqref{A-control1}-\eqref{I_3-estimate} that for $t\in[\sigma(T),T]$,
  \begin{equation}\label{A-control2}
    \begin{aligned}
    &\quad\left(\int\sigma^m(\mu|\curl u|^2+(2\mu+\lambda)|\div u|^2-2P\div u)+\mu\int_{\pa\Omega}\sigma^mAu\cdot u\right)_t+\int\sigma^m\rho|\dot{u}|^2\\
    &\leq Cm\sigma^{m-1}\sigma'(\int P|\nabla u|+\norm{\nabla u}_{L^2}^2)+C\sigma^m(\int P|\nabla u|^2+\norm{\nabla u}_{L^3}^3)\\
    &\quad+C\sigma^m(\norm{\rho\dot{u}}_{L^2}^2\norm{\nabla u}_{L^2}+\norm{A}_{W^{1,6}}\norm{\rho\dot{u}}_{L^2}\norm{\nabla u}_{L^2}^2+\norm{A}_{W^{1,6}}^{\frac{3}{2}}\norm{\nabla u}_{L^2}^3),
    \end{aligned}
  \end{equation}
  and for $t\in[0,\sigma(T)]$,
  \begin{equation}\label{A-control2-small}
    \begin{aligned}
    &\quad\left(\int\sigma^m(\mu|\curl u|^2+(2\mu+\lambda)|\div u|^2-2P\div u)+\mu\int_{\pa\Omega}\sigma^mAu\cdot u\right)_t+\int\sigma^m\rho|\dot{u}|^2\\
    &\leq Cm\sigma^{m-1}\sigma'(\int P|\nabla u|+\norm{\nabla u}_{L^2}^2)+C\sigma^m(\int P|\nabla u|^2+\norm{\nabla u}_{L^3}^3)\\
    &\quad+C\sigma^m(\norm{\rho\dot{u}}_{L^2}\norm{\nabla u}_{L^2}^2+\norm{A}_{W^{1,6}}\norm{\nabla u}_{L^2}^3).
    \end{aligned}
  \end{equation}
  Then integrating \eqref{A-control2}-\eqref{A-control2-small} over $[0,T]$, and using \eqref{essential-energy}, \eqref{a-priori-assumption}, \eqref{Hodge-decomposition3-exterior1} and \eqref{essential-estimate}, we have that for any integer $m\geq1$,
  \begin{equation}\label{A-control3}
    \begin{aligned}
    &\quad\sup_{t\in[0,T]}\sigma^m\norm{\nabla u}_{L^2}^2+\int_0^T\int\sigma^m\rho|\dot{u}|^2\\
    &\leq C(\tilde{\rho})\mathcal{E}_0+Cm\int_0^{\sigma(T)}\int\sigma^{m-1}(P^2+|\nabla u|^2)+C\int_0^T\int\sigma^mP|\nabla u|^2+C\int_0^T\sigma^m\norm{\nabla u}_{L^3}^3\\
    &\quad+C\int_{\sigma(T)}^T\sigma^m(\norm{\rho\dot{u}}_{L^2}^2\norm{\nabla u}_{L^2}+\norm{A}_{W^{1,6}}\norm{\rho\dot{u}}_{L^2}\norm{\nabla u}_{L^2}^2+\norm{A}_{W^{1,6}}^{\frac{3}{2}}\norm{\nabla u}_{L^2}^3)\\
    &\quad+C\int_0^{\sigma(T)}\sigma^m(\norm{\rho\dot{u}}_{L^2}\norm{\nabla u}_{L^2}^2+\norm{A}_{W^{1,6}}\norm{\nabla u}_{L^2}^3).
    \end{aligned}
  \end{equation}
  Taking $m=1$ in the above inequality yields that
   \begin{equation}\label{A-control4}
     \begin{aligned}
     &\quad A_1(T)=\sup_{t\in[0,T]}\sigma\norm{\nabla u}_{L^2}^2+\int_0^T\int\sigma\rho|\dot{u}|^2\\
     &\leq C(\tilde{\rho})\mathcal{E}_0+C\int_0^T\int\sigma(P|\nabla u|^2+|\nabla u|^3)+C(\tilde{\rho})A_1^{\frac{3}{2}}(T)+C\norm{A}_{W^{1,6}}^{\frac{3}{2}}A_1^{\frac{1}{2}}(T)E_0\\
     &\quad+C(\tilde{\rho})\norm{A}_{W^{1,6}}A_1^{\frac{1}{2}}(\sigma(T))\mathcal{E}_0+C(\tilde{\rho})\norm{A}_{W^{1,6}}(\int_{\sigma(T)}^T\norm{\sqrt{\rho}\dot{u}}_{L^2}^2)^{\frac{1}{2}}(\int_{\sigma(T)}^T\norm{\nabla u}_{L^2}^4)^{\frac{1}{2}}\\
     &\quad+C(\tilde{\rho})(\int_0^{\sigma(T)}\sigma\norm{\sqrt{\rho}\dot{u}}_{L^2}^2)^{\frac{1}{2}}(\int_0^{\sigma(T)}\sigma\norm{\nabla u}_{L^2}^4)^{\frac{1}{2}}\\
     &\leq C(\tilde{\rho})\mathcal{E}_0+C\int_0^T\int\sigma(P|\nabla u|^2+|\nabla u|^3)+C(\tilde{\rho})A_1^{\frac{3}{2}}(T)+C\norm{A}_{W^{1,6}}^{\frac{3}{2}}A_1^{\frac{1}{2}}(T)E_0\\
     &\quad+C(\tilde{\rho})\norm{A}_{W^{1,6}}A_1^{\frac{1}{2}}(\sigma(T))\mathcal{E}_0+C(\tilde{\rho})\norm{A}_{W^{1,6}}A_1(T)E_0^{\frac{1}{2}}+C(\tilde{\rho})A_1(\sigma(T))\mathcal{E}_0^{\frac{1}{2}}\\
     &\leq C(\tilde{\rho})\mathcal{E}_0+C(\tilde{\rho})A_1^{\frac{3}{2}}(T)+C\norm{A}_{W^{1,6}}^{\frac{3}{2}}A_1^{\frac{1}{2}}(T)E_0+C\int_0^T\int\sigma(P|\nabla u|^2+|\nabla u|^3),
     \end{aligned}
   \end{equation}
   where we have used \eqref{a-priori-assumption}, \eqref{essential-energy}, \eqref{essential-estimate}, \eqref{Hodge-decomposition3-exterior1}, H\"{o}lder's inequality and $\mathcal{E}_0\leq 1$, $\norm{A}_{W^{1,6}}\leq 1$. This finishes proof of \eqref{A1-control}.

  Next, we turn to prove \eqref{A2-control}. Recalling \eqref{momentum-equation} as
  \begin{equation}\label{momentum-equation2}
    \rho\dot{u}=\nabla G-\mu\nabla\times\curl u,
  \end{equation}
  then taking $\sigma^m\dot{u}_j[\pa_t+\div(u\cdot)]$ on the $j$-th component of $\eqref{momentum-equation2}$, summing over $j$, and integrating over $\Omega$ yields
  \begin{equation}\label{A-control-1}
    \begin{aligned}
    &\quad\left(\frac{1}{2}\int\sigma^m\rho|\dot{u}|^2\right)_t-\frac{1}{2}m\sigma^{m-1}\sigma'\int\rho|\dot{u}|^2\\
    &=\int\sigma^m(\dot{u}\cdot\nabla G_t+\dot{u}_j\div(u\pa_jG))+\mu\int\sigma^m(-\dot{u}\cdot\nabla\times\curl u_t-\dot{u}_j\div(u(\nabla\times\curl u)_j))\\
    &=J_1+\mu J_2.
    \end{aligned}
  \end{equation}
  For $J_1$, by virtue of \eqref{Navier-slip-condition} and \eqref{Pressure-eq}, we have
  \begin{equation}\label{J_1-estimate1}
    \begin{aligned}
    J_1&=\int\sigma^m\dot{u}\cdot\nabla G_t+\int\sigma^m\dot{u}_j\div(u\pa_jG)\\
    &=\int_{\pa\Omega}\sigma^mG_t\dot{u}\cdot n-\int\sigma^mG_t\div\dot{u}-\int\sigma^mu\cdot\nabla\dot{u}\cdot\nabla G\\
    &=\int_{\pa\Omega}\sigma^mG_t\dot{u}\cdot n-(2\mu+\lambda)\int\sigma^m|\div\dot{u}|^2+(2\mu+\lambda)\int\sigma^m\div\dot{u}\nabla u:\nabla u^T\\
    &\quad+\int\sigma^m\div\dot{u}u\cdot\nabla G-\gamma\int\sigma^mP\div u\div\dot{u}-\int\sigma^mu\cdot\nabla\dot{u}\cdot\nabla G\\
    &\leq\int_{\pa\Omega}\sigma^mG_t\dot{u}\cdot n-(2\mu+\lambda)\int\sigma^m|\div\dot{u}|^2+\delta\sigma^m\norm{\nabla\dot{u}}_{L^2}^2\\
    &\quad+C(\delta)\sigma^m(\norm{\nabla u}_{L^2}^2\norm{\nabla G}_{L^3}^2+\norm{\nabla u}_{L^4}^4+\norm{P|\nabla u|}_{L^2}^2),
    \end{aligned}
  \end{equation}
  where we have used
  \begin{equation*}
    \begin{aligned}
    G_t&=(2\mu+\lambda)\div u_t-P_t\\
    &=(2\mu+\lambda)\div\dot{u}-(2\mu+\lambda)\div(u\cdot\nabla u)+u\cdot\nabla P+\gamma P\div u\\
    &=(2\mu+\lambda)\div\dot{u}-(2\mu+\lambda)\nabla u:\nabla u^T-u\cdot\nabla G+\gamma P\div u.
    \end{aligned}
  \end{equation*}
  For the boundary term in \eqref{A-control-1}, we have
  \begin{equation}\label{boundary-term-A2-1}
    \begin{aligned}
    &\quad\int_{\pa\Omega}\sigma^mG_t\dot{u}\cdot n
    =-\int_{\pa\Omega}\sigma^mG_tu\cdot\nabla n\cdot u\\
    &=-\left(\int_{\pa\Omega}\sigma^mG(u\cdot\nabla n\cdot u)\right)_t+m\sigma^{m-1}\sigma'\int_{\pa\Omega}G(u\cdot\nabla n\cdot u)\\
    &\quad+\sigma^m\int_{\pa\Omega}G(\dot{u}\cdot\nabla n\cdot u+u\cdot\nabla n\cdot\dot{u})-\sigma^m\int_{\pa\Omega}G((u\cdot\nabla u)\cdot\nabla n\cdot u+u\cdot\nabla n\cdot(u\cdot\nabla u))\\
    &\leq-\left(\int_{\pa\Omega}\sigma^mG(u\cdot\nabla n\cdot u)\right)_t+Cm\sigma^{m-1}\sigma'\norm{\nabla u}_{L^2}^2\norm{\nabla G}_{L^2}+C\sigma^m\norm{\nabla G}_{L^2}\norm{\nabla\dot{u}}_{L^2}\norm{\nabla u}_{L^2}\\
    &\quad-\sigma^m\int_{\pa\Omega}G((u\cdot\nabla u)\cdot\nabla n\cdot u+u\cdot\nabla n\cdot(u\cdot\nabla u)),
    \end{aligned}
  \end{equation}
  where we have used H\"{o}lder's inequality and \eqref{transition-lem-boundary1}.

  For the rest boundary term in \eqref{boundary-term-A2-1}, we have from the support of $n$, H\"{o}lder's inequality, the similar argument as in \eqref{transition-lem-Hodge1} that
  \begin{equation}\label{boundary-term-A2-2}
    \begin{aligned}
    &\quad-\int_{\pa\Omega}G(u\cdot\nabla u)\cdot\nabla n\cdot u=-\int_{\pa\Omega}u^{\bot}\times n\cdot\nabla u_i\nabla_in\cdot uG\\
    &=\int_{\pa\Omega}(u^{\bot}\times\nabla u_i\nabla_in\cdot uG)\cdot n=\int\div(u^{\bot}\times\nabla u_i\nabla_in\cdot uG)\\
    &=\int(u^{\bot}\times\nabla u_i)\cdot\nabla(\nabla_in\cdot uG)+\int(\nabla\times u^{\bot})\cdot\nabla u_i\nabla_in\cdot uG\\
    &\leq C\int_{B_{2R}\cap\Omega}|\nabla G||u|^2|\nabla u|+C\int_{B_{2R}\cap\Omega}|G|(|\nabla u|^2|u|+|\nabla u||u|^2)\\
    &\leq C(\norm{\nabla G}_{L^3}\norm{\nabla u}_{L^4}\norm{\nabla u}_{L^2}^2+\norm{\nabla G}_{L^2}\norm{\nabla u}_{L^4}^2\norm{\nabla u}_{L^2}+\norm{\nabla G}_{L^2}\norm{\nabla u}_{L^4}\norm{\nabla u}_{L^2}^2),
    \end{aligned}
  \end{equation}
  where we have used the simple fact that
  \begin{equation*}
    \div(a\times b)=(\nabla\times a)\cdot b-(\nabla\times b)\cdot a.
  \end{equation*}
  The estimate above is also applicable to $-\displaystyle\int_{\pa\Omega}Gu\cdot\nabla n\cdot(u\cdot\nabla u)$.

   Now, to finish the control of $J_1$, by virtue of \eqref{curl-G-modified-estimate1}, H\"{o}lder's and Young's inequalities, we have some estimates as follows:
  \begin{equation}\label{boundary-term-A2-2-transition1}
    \begin{aligned}
    &\quad\norm{\nabla G}_{L^3}\norm{\nabla u}_{L^4}\norm{\nabla u}_{L^2}^2\\
    &\leq C(\norm{\rho\dot{u}}_{L^3}+\norm{A}_{W^{1,6}}(\norm{\nabla u}_{L^2}+\norm{\nabla u}_{L^3}))\norm{\nabla u}_{L^4}\norm{\nabla u}_{L^2}^2\\
    &\leq C(\tilde{\rho})\norm{\sqrt{\rho}\dot{u}}_{L^2}^{\frac{1}{2}}\norm{\nabla\dot{u}}_{L^2}^{\frac{1}{2}}\norm{\nabla u}_{L^4}\norm{\nabla u}_{L^2}^2+C\norm{A}_{W^{1,6}}(\norm{\nabla u}_{L^2}^3\norm{\nabla u}_{L^4}+\norm{\nabla u}_{L^2}^{\frac{7}{3}}\norm{\nabla u}_{L^4}^{\frac{5}{3}})\\
    &\leq C(\tilde{\rho})\norm{\sqrt{\rho}\dot{u}}_{L^2}^{\frac{2}{3}}\norm{\nabla\dot{u}}_{L^2}^{\frac{2}{3}}\norm{\nabla u}_{L^2}^{\frac{8}{3}}+C\norm{\nabla u}_{L^4}^4+C(\norm{A}_{W^{1,6}}^{\frac{4}{3}}+\norm{A}_{W^{1,6}}^{\frac{12}{7}})\norm{\nabla u}_{L^2}^4\\
    &\leq\delta\norm{\nabla\dot{u}}_{L^2}^2+C(\delta,\tilde{\rho})\norm{\sqrt{\rho}\dot{u}}_{L^2}\norm{\nabla u}_{L^2}^4+C\norm{\nabla u}_{L^4}^4+C\norm{A}_{W^{1,6}}^{\frac{4}{3}}\norm{\nabla u}_{L^2}^4,
    \end{aligned}
  \end{equation}
  and
  \begin{equation}\label{boundary-term-A2-2-transition2}
    \begin{aligned}
    &\quad\norm{\nabla G}_{L^2}\norm{\nabla u}_{L^4}^2\norm{\nabla u}_{L^2}\\
    &\leq C(\norm{\rho\dot{u}}_{L^2}+\norm{A}_{W^{1,6}}\norm{\nabla u}_{L^2})\norm{\nabla u}_{L^4}^2\norm{\nabla u}_{L^2}\\
    &\leq C\norm{\nabla u}_{L^4}^4+C\norm{\rho\dot{u}}_{L^2}^2\norm{\nabla u}_{L^2}^2+C\norm{A}_{W^{1,6}}^2\norm{\nabla u}_{L^2}^4,
    \end{aligned}
  \end{equation}
  and
  \begin{equation}\label{boundary-term-A2-2-transition3}
    \begin{aligned}
    &\quad\norm{\nabla G}_{L^2}\norm{\nabla u}_{L^4}\norm{\nabla u}_{L^2}^2\\
    &\leq C(\norm{\rho\dot{u}}_{L^2}+\norm{A}_{W^{1,6}}\norm{\nabla u}_{L^2})\norm{\nabla u}_{L^4}\norm{\nabla u}_{L^2}^2\\
    &\leq C\norm{\rho\dot{u}}_{L^2}^{\frac{4}{3}}\norm{\nabla u}_{L^2}^{\frac{8}{3}}+C\norm{A}_{W^{1,6}}^{\frac{4}{3}}\norm{\nabla u}_{L^2}^4+C\norm{\nabla u}_{L^4}^4,
    \end{aligned}
  \end{equation}
  and
  \begin{equation}\label{boundary-term-A2-2-transition4}
    \begin{aligned}
    &\quad\norm{\nabla G}_{L^2}\norm{\nabla\dot{u}}_{L^2}\norm{\nabla u}_{L^2}\\
    &\leq\delta\norm{\nabla\dot{u}}_{L^2}^2+C(\delta)\norm{\nabla G}_{L^2}^2\norm{\nabla u}_{L^2}^2\\
    &\leq\delta\norm{\nabla\dot{u}}_{L^2}^2+C(\delta)(\norm{\rho\dot{u}}_{L^2}^2+\norm{A}_{W^{1,6}}^2\norm{\nabla u}_{L^2}^2)\norm{\nabla u}_{L^2}^2,
    \end{aligned}
  \end{equation}
  and in the last,
  \begin{equation}\label{A2-estimate-transition1}
    \begin{aligned}
    &\quad\norm{\nabla u}_{L^2}^2\norm{\nabla G}_{L^3}^2\\
    &\leq C\norm{\nabla u}_{L^2}^2(\norm{\rho\dot{u}}_{L^3}^2+\norm{A}_{W^{1,6}}^2(\norm{\nabla u}_{L^2}^2+\norm{\nabla u}_{L^3}^2))\\
    &\leq C(\tilde{\rho})\norm{\nabla u}_{L^2}^2\norm{\sqrt{\rho}\dot{u}}_{L^2}\norm{\nabla\dot{u}}_{L^2}+C\norm{A}_{W^{1,6}}^2\norm{\nabla u}_{L^2}^2(\norm{\nabla u}_{L^2}^2+\norm{\nabla u}_{L^2}^{\frac{2}{3}}\norm{\nabla u}_{L^4}^{\frac{4}{3}})\\
    &\leq\delta\norm{\nabla\dot{u}}_{L^2}^2+C(\delta,\tilde{\rho})\norm{\sqrt{\rho}\dot{u}}_{L^2}^2\norm{\nabla u}_{L^2}^4+C\norm{\nabla u}_{L^4}^4+C(\norm{A}_{W^{1,6}}^2+\norm{A}_{W^{1,6}}^3)\norm{\nabla u}_{L^2}^4\\
    &\leq\delta\norm{\nabla\dot{u}}_{L^2}^2+C(\delta,\tilde{\rho})\norm{\sqrt{\rho}\dot{u}}_{L^2}^2\norm{\nabla u}_{L^2}^4+C\norm{\nabla u}_{L^4}^4+C\norm{A}_{W^{1,6}}^2\norm{\nabla u}_{L^2}^4,
    \end{aligned}
  \end{equation}
  provided $\norm{A}_{W^{1,6}}\leq 1$.

  Then, combining \eqref{J_1-estimate1}-\eqref{A2-estimate-transition1}, \eqref{curl-G-modified-estimate1} and \eqref{a-priori-assumption}, we obtain
  \begin{equation}\label{J_1-estimate2}
    \begin{aligned}
    J_1&\leq-\left(\int_{\pa\Omega}\sigma^mG(u\cdot\nabla n\cdot u)\right)_t+Cm\sigma^{m-1}\sigma'\norm{\nabla u}_{L^2}^2(\norm{\rho\dot{u}}_{L^2}+\norm{A}_{W^{1,6}}\norm{\nabla u}_{L^2})\\
    &\quad-(2\mu+\lambda)\int\sigma^m|\div\dot{u}|^2+C\delta\sigma^m\norm{\nabla\dot{u}}_{L^2}^2
    +C(\delta)\sigma^m(\norm{\nabla u}_{L^4}^4+\norm{P|\nabla u|}_{L^2}^2)\\
    &\quad+C(\delta,\tilde{\rho})\sigma^m(\norm{\sqrt{\rho}\dot{u}}_{L^2}^2\norm{\nabla u}_{L^2}^4+\norm{\sqrt{\rho}\dot{u}}_{L^2}^2\norm{\nabla u}_{L^2}^2+\norm{\sqrt{\rho}\dot{u}}_{L^2}^{\frac{4}{3}}\norm{\nabla u}_{L^2}^{\frac{8}{3}}+\norm{\sqrt{\rho}\dot{u}}_{L^2}\norm{\nabla u}_{L^2}^4)\\
    &\quad+C(\delta)\sigma^m\norm{A}_{W^{1,6}}^{\frac{4}{3}}\norm{\nabla u}_{L^2}^4.
    \end{aligned}
  \end{equation}
  provided $\norm{A}_{W^{1,6}}\leq 1$.

  Similarly, for $J_2$, we have
  \begin{equation}\label{J_2-estimate1}
    \begin{aligned}
    J_2&=\int\sigma^m(-\dot{u}\cdot\nabla\times\curl u_t-\dot{u}_j\div(u(\nabla\times\curl u)_j))\\
    &=-\int\sigma^m|\curl\dot{u}|^2+\int\sigma^m\curl\dot{u}\cdot(\nabla u_i\times\nabla_iu)+\int\sigma^mu\cdot\nabla\curl u\cdot\curl\dot{u}\\
    &\quad+\int_{\pa\Omega}\sigma^m\curl u_t\times n\cdot\dot{u}+\int\sigma^mu\cdot\nabla\dot{u}\cdot(\nabla\times\curl u)\\
    &\leq-\int\sigma^m|\curl\dot{u}|^2-\int_{\pa\Omega}\sigma^mA\dot{u}\cdot\dot{u}
    +\delta\sigma^m\norm{\nabla\dot{u}}_{L^2}^2+C(\delta)\sigma^m(\norm{\nabla u}_{L^4}^4+\norm{u|\nabla\curl u|}_{L^2}^2)\\
    &\quad+\sigma^m(\int\nabla(A\dot{u})_i\times\nabla u_i\cdot u^{\bot}-\int\nabla\times u^{\bot}\cdot\nabla u_i(A\dot{u})_i)\\
    &\leq-\int\sigma^m|\curl\dot{u}|^2-\int_{\pa\Omega}\sigma^mA\dot{u}\cdot\dot{u}
    +3\delta\sigma^m\norm{\nabla\dot{u}}_{L^2}^2+C(\delta)\sigma^m\norm{\nabla u}_{L^4}^4\\
    &\quad+C(\delta,\tilde{\rho})\sigma^m\norm{\sqrt{\rho}\dot{u}}_{L^2}^2\norm{\nabla u}_{L^2}^4+C(\delta)\norm{A}_{W^{1,6}}^2\norm{\nabla u}_{L^2}^4,
    \end{aligned}
  \end{equation}
  where we have applied $\norm{A}_{W^{1,6}}\leq 1$, the support of $A$, the similar argument as in \eqref{transition-lem-Hodge1}, \eqref{curl-G-modified-estimate1} and the facts that
  \begin{equation*}
    \curl u_t=\curl\dot{u}-u\cdot\nabla\curl u-\nabla u_i\times\nabla_iu,
  \end{equation*}
  and
  \begin{equation}\label{boundary-term-A2-4}
    \begin{aligned}
    &\quad\int_{\pa\Omega}\curl u_t\times n\cdot\dot{u}=-\int_{\pa\Omega}Au_t\cdot\dot{u}
    =-\int_{\pa\Omega}A\dot{u}\cdot\dot{u}+\int_{\pa\Omega}(u\cdot\nabla u)\cdot A\cdot\dot{u}\\
    &=-\int_{\pa\Omega}A\dot{u}\cdot\dot{u}+\int_{\pa\Omega}(u^{\bot}\times n\cdot\nabla u)\cdot A\dot{u}\\
    &=-\int_{\pa\Omega}A\dot{u}\cdot\dot{u}+\int_{\pa\Omega}(\nabla u_i(A\dot{u})_i\times u^{\bot})\cdot n\\
    &=-\int_{\pa\Omega}A\dot{u}\cdot\dot{u}+\int\div(\nabla u_i(A\dot{u})_i\times u^{\bot})\\
    &=-\int_{\pa\Omega}A\dot{u}\cdot\dot{u}+\int\nabla\times(\nabla u_i(A\dot{u})_i)\cdot u^{\bot}-\int\nabla\times u^{\bot}\cdot\nabla u_i(A\dot{u})_i\\
    &=-\int_{\pa\Omega}A\dot{u}\cdot\dot{u}+\int\nabla(A\dot{u})_i\times\nabla u_i\cdot u^{\bot}-\int\nabla\times u^{\bot}\cdot\nabla u_i(A\dot{u})_i\\
    &\leq-\int_{\pa\Omega}A\dot{u}\cdot\dot{u}+C\norm{A}_{W^{1,6}}(\norm{\nabla\dot{u}}_{L^2}\norm{\nabla u}_{L^2}^2+\norm{\nabla\dot{u}}_{L^2}\norm{\nabla u}_{L^4}\norm{\nabla u}_{L^2}+\norm{\nabla\dot{u}}_{L^2}\norm{\nabla u}_{L^4}^2)\\
    &\leq-\int_{\pa\Omega}A\dot{u}\cdot\dot{u}+\delta\norm{\nabla\dot{u}}_{L^2}^2+C(\delta)(\norm{A}_{W^{1,6}}^2+\norm{A}_{W^{1,6}}^4)\norm{\nabla u}_{L^2}^4+C(\delta)\norm{A}_{W^{1,6}}^2\norm{\nabla u}_{L^4}^4,
    \end{aligned}
  \end{equation}
  and the simple estimate from \eqref{curl-G-modified-estimate1} as
  \begin{equation*}
    \begin{aligned}
    \norm{u|\nabla\curl u|}_{L^2}^2&\leq\norm{u}_{L^6}^2\norm{\nabla\curl u}_{L^3}^2\leq C\norm{\nabla u}_{L^2}^2(\norm{\rho\dot{u}}_{L^3}^2+\norm{A}_{W^{1,6}}^2(\norm{\nabla u}_{L^2}^2+\norm{\nabla u}_{L^3}^2))\\
    &\leq C\norm{\nabla u}_{L^2}^2(\norm{\rho\dot{u}}_{L^2}\norm{\rho\dot{u}}_{L^6}+\norm{A}_{W^{1,6}}^2(\norm{\nabla u}_{L^2}^2+\norm{\nabla u}_{L^2}^{\frac{2}{3}}\norm{\nabla u}_{L^4}^{\frac{4}{3}}))\\
    &\leq\delta\norm{\nabla\dot{u}}_{L^2}^2+C(\delta,\tilde{\rho})\norm{\sqrt{\rho}\dot{u}}_{L^2}^2\norm{\nabla u}_{L^2}^4+C\norm{\nabla u}_{L^4}^4+C(\norm{A}_{W^{1,6}}^2+\norm{A}_{W^{1,6}}^3)\norm{\nabla u}_{L^2}^4.
    \end{aligned}
  \end{equation*}
  Therefore, combining \eqref{A-control-1}, \eqref{J_1-estimate2} and \eqref{J_2-estimate1} gives that
  \begin{equation}\label{A-control-2}
    \begin{aligned}
    &\quad\left(\int\sigma^m\rho|\dot{u}|^2+2\int_{\pa\Omega}\sigma^mu\cdot\nabla n\cdot uG\right)_t+2\sigma^m\int(\mu|\curl\dot{u}|^2+(2\mu+\lambda)|\div\dot{u}|^2)\\
    &\leq Cm\sigma^{m-1}\sigma'(\norm{\nabla u}_{L^2}^2(\norm{\sqrt{\rho}\dot{u}}_{L^2}^2+C(\tilde{\rho}))+\norm{\sqrt{\rho}\dot{u}}_{L^2}^2+\norm{\nabla u}_{L^2}^4)+C\delta\sigma^m\norm{\nabla\dot{u}}_{L^2}^2
    \\
    &\quad+C(\delta,\tilde{\rho})\sigma^m(\norm{\sqrt{\rho}\dot{u}}_{L^2}^2\norm{\nabla u}_{L^2}^4+\norm{\sqrt{\rho}\dot{u}}_{L^2}^2\norm{\nabla u}_{L^2}^2+\norm{\sqrt{\rho}\dot{u}}_{L^2}^{\frac{4}{3}}\norm{\nabla u}_{L^2}^{\frac{8}{3}}+\norm{\sqrt{\rho}\dot{u}}_{L^2}\norm{\nabla u}_{L^2}^4)\\
    &\quad+C(\delta)\sigma^m\norm{A}_{W^{1,6}}^{\frac{4}{3}}\norm{\nabla u}_{L^2}^4+C(\delta)\sigma^m(\norm{\nabla u}_{L^4}^4+\norm{P|\nabla u|}_{L^2}^2).
    \end{aligned}
  \end{equation}
  Thus, using Lemma \ref{lem-dot-u} and choosing $\delta>0$ sufficiently small yields that
  \begin{equation}\label{A-control-3}
    \begin{aligned}
    &\quad\left(\int\sigma^m\rho|\dot{u}|^2+2\int_{\pa\Omega}\sigma^mu\cdot\nabla n\cdot uG\right)_t+\sigma^m\int(\mu|\curl\dot{u}|^2+(2\mu+\lambda)|\div\dot{u}|^2)\\
    &\leq Cm\sigma^{m-1}\sigma'(\norm{\nabla u}_{L^2}^2(\norm{\sqrt{\rho}\dot{u}}_{L^2}^2+C(\tilde{\rho}))+\norm{\sqrt{\rho}\dot{u}}_{L^2}^2+\norm{\nabla u}_{L^2}^4)
    \\
    &\quad+C(\tilde{\rho})\sigma^m(\norm{\sqrt{\rho}\dot{u}}_{L^2}^2\norm{\nabla u}_{L^2}^4+\norm{\sqrt{\rho}\dot{u}}_{L^2}^2\norm{\nabla u}_{L^2}^2+\norm{\sqrt{\rho}\dot{u}}_{L^2}^{\frac{4}{3}}\norm{\nabla u}_{L^2}^{\frac{8}{3}}+\norm{\sqrt{\rho}\dot{u}}_{L^2}\norm{\nabla u}_{L^2}^4)\\
    &\quad+C\sigma^m\norm{A}_{W^{1,6}}^{\frac{4}{3}}\norm{\nabla u}_{L^2}^4+C\sigma^m(\norm{\nabla u}_{L^4}^4+\norm{P|\nabla u|}_{L^2}^2+\norm{\nabla u}_{L^{\frac{8}{3}}}^4).
    \end{aligned}
  \end{equation}
  Integrating the above inequality over $[0,T]$, taking $m=3$ and using \eqref{a-priori-assumption}, \eqref{essential-energy}, \eqref{essential-estimate}, \eqref{boundary-term-A1-1-small}, \eqref{curl-G-modified-estimate1} and Lemma \ref{lem-dot-u}, we have
  \begin{equation}\label{A-control-4}
    \begin{aligned}
    &\quad\sup_{t\in[0,T]}\sigma^3\int\rho|\dot{u}|^2+\int_0^T\int\sigma^3|\nabla\dot{u}|^2\\
    &\leq C\sup_{t\in[0,T]}\sigma^3\norm{\nabla G}_{L^2}\norm{\nabla u}_{L^2}^2+C\int_0^{\sigma(T)}\sigma^2(\norm{\nabla u}_{L^2}^2(\norm{\sqrt{\rho}\dot{u}}_{L^2}^2+C(\tilde{\rho}))+\norm{\sqrt{\rho}\dot{u}}_{L^2}^2+\norm{\nabla u}_{L^2}^4)\\
    &\quad+C(\tilde{\rho})\int_0^T\sigma^3(\norm{\sqrt{\rho}\dot{u}}_{L^2}^2\norm{\nabla u}_{L^2}^4+\norm{\sqrt{\rho}\dot{u}}_{L^2}^2\norm{\nabla u}_{L^2}^2+\norm{\sqrt{\rho}\dot{u}}_{L^2}^{\frac{4}{3}}\norm{\nabla u}_{L^2}^{\frac{8}{3}}+\norm{\sqrt{\rho}\dot{u}}_{L^2}\norm{\nabla u}_{L^2}^4)\\
    &\quad+C\int_0^T\sigma^3(\norm{\nabla u}_{L^4}^4+\norm{P|\nabla u|}_{L^2}^2+\norm{\nabla u}_{L^{\frac{8}{3}}}^4+\norm{A}_{W^{1,6}}^{\frac{4}{3}}\norm{\nabla u}_{L^2}^4)\\
    &\leq \sup_{t\in[0,T]}\sigma^3(\frac{1}{4}\norm{\sqrt{\rho}\dot{u}}_{L^2}^2+C\norm{\nabla u}_{L^2}^3+C(\tilde{\rho})\norm{\nabla u}_{L^2}^4)+C(\tilde{\rho})\mathcal{E}_0+CA_1^2(\sigma(T))+CA_1(\sigma(T))\\
    &\quad+C(\tilde{\rho})A_1(\sigma(T))\mathcal{E}_0+C(\tilde{\rho})(A_1^3(T)+A_1^2(T)+A_1^{\frac{5}{3}}(T)E_0^{\frac{1}{3}}+A_1^2(T)E_0^{\frac{1}{2}})\\
    &\quad +C\int_0^T\sigma^3(\norm{\nabla u}_{L^4}^4+\norm{P|\nabla u|}_{L^2}^2+\norm{\nabla u}_{L^{\frac{8}{3}}}^4+\norm{A}_{W^{1,6}}^{\frac{4}{3}}\norm{\nabla u}_{L^2}^4),
    \end{aligned}
  \end{equation}
  where we have used the H\"{o}lder's inequality as
  \begin{equation*}
    \begin{aligned}
    &\quad\int_0^T\sigma(\norm{\sqrt{\rho}\dot{u}}_{L^2}^{\frac{4}{3}}\norm{\nabla u}_{L^2}^{\frac{2}{3}}+\norm{\sqrt{\rho}\dot{u}}_{L^2}\norm{\nabla u}_{L^2})\\
    &\leq(\int_0^T\sigma\norm{\sqrt{\rho}\dot{u}}_{L^2}^2)^{\frac{2}{3}}(\int_0^T\sigma\norm{\nabla u}_{L^2}^2)^{\frac{1}{3}}+(\int_0^T\sigma\norm{\sqrt{\rho}\dot{u}}_{L^2}^2)^{\frac{1}{2}}(\int_0^T\sigma\norm{\nabla u}_{L^2}^2)^{\frac{1}{2}}\\
    &\leq C(\tilde{\rho})(A_1^{\frac{2}{3}}(T)E_0^{\frac{1}{3}}+A_1^{\frac{1}{2}}(T)E_0^{\frac{1}{2}}).
    \end{aligned}
  \end{equation*}
  The inequality \eqref{A-control-4} implies
  \begin{equation}\label{A-control-5}
    \begin{aligned}
    A_2(T)&\leq CA_1^{\frac{3}{2}}(T)+C(\tilde{\rho})A_1^2(T)+C(\tilde{\rho})\mathcal{E}_0+CA_1^2(\sigma(T))+CA_1(\sigma(T))
    +C(\tilde{\rho})A_1(\sigma(T))\mathcal{E}_0\\
    &\quad+C(\tilde{\rho})(A_1^3(T)+A_1^2(T)+A_1^{\frac{5}{3}}(T)E_0^{\frac{1}{3}}+A_1^2(T)E_0^{\frac{1}{2}})\\
    &\quad+C\int_0^T\sigma^3(\norm{\nabla u}_{L^4}^4+\norm{P|\nabla u|}_{L^2}^2+\norm{\nabla u}_{L^{\frac{8}{3}}}^4+\norm{A}_{W^{1,6}}^{\frac{4}{3}}\norm{\nabla u}_{L^2}^4)\\
    &\leq C(\tilde{\rho})\mathcal{E}_0+CA_1^{\frac{3}{2}}(T)+CA_1(\sigma(T))+C(\tilde{\rho})(A_1^2(T)+A_1^{\frac{5}{3}}(T)E_0^{\frac{1}{3}}+A_1^2(T)E_0^{\frac{1}{2}})\\
    &\quad+C\int_0^T\sigma^3(\norm{\nabla u}_{L^4}^4+\norm{P|\nabla u|}_{L^2}^2+\norm{\nabla u}_{L^{\frac{8}{3}}}^4+\norm{A}_{W^{1,6}}^{\frac{4}{3}}\norm{\nabla u}_{L^2}^4),
    \end{aligned}
  \end{equation}
  provided $\mathcal{E}_0\leq1$. Thus, we complete the proof of Lemma \ref{lem-A1-A2-control}.
\end{proof}

\begin{lem}\label{lem-small-time-estimate}
  Under the conditions of Proposition \ref{prop-a-priori-estimate}, it holds that
  \begin{equation}\label{small-time-estimate1}
    \sup_{t\in[0,\sigma(T)]}\int|\nabla u|^2+\int_0^{\sigma(T)}\int\rho|\dot{u}|^2\leq C(\tilde{\rho}, M),
  \end{equation}
  \begin{equation}\label{small-time-estimate2}
    \sup_{t\in[0,\sigma(T)]}t\int\rho|\dot{u}|^2+\int_0^{\sigma(T)}\int t|\nabla\dot{u}|^2\leq C(\tilde{\rho}, M),
  \end{equation}
  provided
\begin{equation}\label{small-assumption1}
  \mathcal{E}_0\leq\epsilon_2=\min\{\epsilon_1,(4C(\tilde{\rho}))^{-12}\},\quad \norm{A}_{W^{1,6}}\leq\tilde{\epsilon}_1
\end{equation}
with $\epsilon_1$ and $\tilde{\epsilon}_1$ defined in Lemma \ref{lem-A1-A2-control}.
\end{lem}

\begin{proof}
  Multiplying $\eqref{Large-CNS-eq}_2$ by $u_t$ and integrating over $\Omega$, we get
  \begin{equation}\label{small-T-estimate1}
  \begin{aligned}
    &\quad\frac{d}{dt}\left(\int(\frac{\mu}{2}|\curl u|^2+\frac{2\mu+\lambda}{2}|\div u|^2-P\div u)+\frac{\mu}{2}\int_{\pa\Omega}Au\cdot u\right)+\int\rho|\dot{u}|^2\\
    &=\int\rho\dot{u}\cdot(u\cdot\nabla u)-\int P_t\div u\\
    &=\int\rho\dot{u}\cdot(u\cdot\nabla u)-\int Pu\cdot\nabla\div u+(\gamma-1)\int P|\div u|^2\\
    &=\int\rho\dot{u}\cdot(u\cdot\nabla u)-\frac{1}{2\mu+\lambda}\int Pu\cdot\nabla G+\frac{1}{2(2\mu+\lambda)}\int\div uP^2+(\gamma-1)\int P|\div u|^2\\
    &\leq C(\tilde{\rho})\norm{\sqrt{\rho}\dot{u}}_{L^2}\norm{\rho^{\frac{1}{3}}u}_{L^3}\norm{\nabla u}_{L^6}+C\norm{P}_{L^3}\norm{u}_{L^6}\norm{\nabla G}_{L^2}+C\norm{\nabla u}_{L^2}\norm{P}_{L^4}^2\\
    &\quad+C(\tilde{\rho})(\gamma-1)\norm{\nabla u}_{L^2}^2\\
    &\leq C(\tilde{\rho})\norm{\sqrt{\rho}\dot{u}}_{L^2}\norm{\rho^{\frac{1}{3}}u}_{L^3}(\norm{\rho\dot{u}}_{L^2}+\norm{P}_{L^6}+\norm{\nabla u}_{L^2})\\
    &\quad+C\norm{P}_{L^3}\norm{\nabla u}_{L^2}(\norm{\rho\dot{u}}_{L^2}+\norm{\nabla u}_{L^2})+C(\tilde{\rho})\norm{\nabla u}_{L^2}^2+C\norm{P}_{L^4}^4\\
    &\leq (C(\tilde{\rho})\norm{\rho^{\frac{1}{3}}u}_{L^3}+\frac{1}{4})\norm{\sqrt{\rho}\dot{u}}_{L^2}^2+C(\tilde{\rho})\norm{\rho^{\frac{1}{3}}u}_{L^3}(\norm{\nabla u}_{L^2}^2+\norm{P}_{L^6}^2)\\
    &\quad+C(\tilde{\rho})((1+\norm{P}_{L^3}+\norm{P}_{L^3}^2)\norm{\nabla u}_{L^2}^2+\norm{P}_{L^1})\\
    &\leq(C(\tilde{\rho})\norm{\rho^{\frac{1}{3}}u}_{L^3}+\frac{1}{4})\norm{\sqrt{\rho}\dot{u}}_{L^2}^2+C(\tilde{\rho})\norm{\rho^{\frac{1}{3}}u}_{L^3}(\norm{\nabla u}_{L^2}^2+\norm{P}_{L^1}^{\frac{1}{3}})\\
    &\quad+C(\tilde{\rho})((1+\norm{P}_{L^1})\norm{\nabla u}_{L^2}^2+\norm{P}_{L^1}),
  \end{aligned}
  \end{equation}
  where we have used \eqref{Pressure-eq}, \eqref{a-priori-assumption}, \eqref{Hodge-decomposition-exterior-Final}, \eqref{curl-G-modified-estimate1}, and also $\norm{A}_{W^{1,6}}\leq 1$.

  Then, integrating \eqref{small-T-estimate1} on $[0,\sigma(T)]$ and using \eqref{Hodge-decomposition3-exterior1}, \eqref{essential-energy} and \eqref{essential-estimate}, we have
  \begin{equation}\label{small-T-estimate2}
    \begin{aligned}
    &\quad\sup_{t\in[0,\sigma(T)]}\norm{\nabla u}_{L^2}^2+\int_0^{\sigma(T)}\int\rho|\dot{u}|^2\\
    &\leq C(M)+C(\tilde{\rho})\mathcal{E}_0+C(\tilde{\rho})A_3^{\frac{1}{3}}(\sigma(T))(\mathcal{E}_0+\mathcal{E}_0^{\frac{1}{3}})
    +C(\tilde{\rho})((1+\mathcal{E}_0)\mathcal{E}_0+\mathcal{E}_0)\\
    &\leq C(\tilde{\rho},M),
    \end{aligned}
  \end{equation}
  provided
  \begin{equation*}
    C(\tilde{\rho})A_3^{\frac{1}{3}}(\sigma(T))\leq\frac{1}{4},\quad \textrm{i.e.}, \quad \mathcal{E}_0\leq(4C(\tilde{\rho}))^{-12}.
  \end{equation*}
  Thus, we complete the proof of \eqref{small-time-estimate1}.

  Next, we turn to prove \eqref{small-time-estimate2}. Taking $m=1$ and $T=\sigma(T)$ in \eqref{A-control-4}, we have from \eqref{small-time-estimate1} that
  \begin{equation}\label{small-T-estimate3}
    \begin{aligned}
    &\quad\sup_{t\in[0,\sigma(T)]}\sigma\int\rho|\dot{u}|^2+\int_0^{\sigma(T)}\int\sigma|\nabla\dot{u}|^2\\
    &\leq \sup_{t\in[0,\sigma(T)]}\sigma(C\norm{\nabla u}_{L^2}^3+C(\tilde{\rho})\norm{\nabla u}_{L^2}^4)+C\int_0^{\sigma(T)}(\norm{\nabla u}_{L^2}^2(\norm{\sqrt{\rho}\dot{u}}_{L^2}^2+C(\tilde{\rho}))+\norm{\sqrt{\rho}\dot{u}}_{L^2}^2+\norm{\nabla u}_{L^2}^4)\\
    &\quad+C(\tilde{\rho})\int_0^{\sigma(T)}\sigma(\norm{\sqrt{\rho}\dot{u}}_{L^2}^2\norm{\nabla u}_{L^2}^4+\norm{\sqrt{\rho}\dot{u}}_{L^2}^2\norm{\nabla u}_{L^2}^2+\norm{\sqrt{\rho}\dot{u}}_{L^2}^{\frac{4}{3}}\norm{\nabla u}_{L^2}^{\frac{8}{3}}+\norm{\sqrt{\rho}\dot{u}}_{L^2}\norm{\nabla u}_{L^2}^4)\\
    &\quad+C\int_0^{\sigma(T)}\sigma(\norm{\nabla u}_{L^4}^4+\norm{P|\nabla u|}_{L^2}^2+\norm{\nabla u}_{L^{\frac{8}{3}}}^4+\norm{A}_{W^{1,6}}^{\frac{4}{3}}\norm{\nabla u}_{L^2}^4)\\
    &\leq C(\tilde{\rho},M)A_1(\sigma(T))+C(\tilde{\rho},M)+C(\tilde{\rho},M)\mathcal{E}_0+C(\tilde{\rho},M)A_1(\sigma(T))+C(\tilde{\rho})A_1(\sigma(T))\mathcal{E}_0\\
    &\quad+\frac{1}{2}\sup_{t\in[0,\sigma(T)]}\sigma\norm{\sqrt{\rho}\dot{u}}_{L^2}^2,
    \end{aligned}
  \end{equation}
  where we have used the following estimate
  \begin{equation*}
    \begin{aligned}
    C\int_0^{\sigma(T)}\sigma\norm{\nabla u}_{L^4}^4&\leq C\int_0^{\sigma(T)}\sigma\norm{\nabla u}_{L^2}\norm{\nabla u}_{L^6}^3\\
    &\leq C\int_0^{\sigma(T)}\sigma\norm{\nabla u}_{L^2}(\norm{\rho\dot{u}}_{L^2}+\norm{\nabla u}_{L^2}+\norm{P}_{L^6})^3\\
    &\leq C(\tilde{\rho},M)\sup_{t\in[0,\sigma(T)]}\sigma\norm{\sqrt{\rho}\dot{u}}_{L^2}\int_0^{\sigma(T)}\norm{\sqrt{\rho}\dot{u}}_{L^2}^2+C(\tilde{\rho})A_1(\sigma(T))\mathcal{E}_0+C(\tilde{\rho})\mathcal{E}_0\\
    &\leq\frac{1}{2}\sup_{t\in[0,\sigma(T)]}\sigma\norm{\sqrt{\rho}\dot{u}}_{L^2}^2+C(\tilde{\rho},M)+C(\tilde{\rho})A_1(\sigma(T))\mathcal{E}_0+C(\tilde{\rho})\mathcal{E}_0
    \end{aligned}
  \end{equation*}
  due to \eqref{Hodge-decomposition-exterior-Final}, \eqref{curl-G-modified-estimate1}, \eqref{essential-energy},  \eqref{essential-estimate}, and $\norm{A}_{W^{1,6}}\leq 1$.

  Then, \eqref{small-T-estimate3} implies
  \begin{equation}\label{small-T-estimate4}
    \sup_{t\in[0,\sigma(T)]}\sigma\int\rho|\dot{u}|^2+\int_0^{\sigma(T)}\int\sigma|\nabla\dot{u}|^2\leq C(\tilde{\rho},M).
  \end{equation}
  Therefore, we complete the proof of \eqref{small-time-estimate2}
  and finish the proof of Lemma \ref{lem-small-time-estimate}.
\end{proof}

\begin{lem}\label{lem-u-L^3}
  Under the conditions of Proposition \ref{prop-a-priori-estimate}, it holds that
  \begin{equation}\label{u-L^3-goal}
    A_3(\sigma(T))\leq\mathcal{E}_0^{\frac{1}{4}},
  \end{equation}
  provided
  \begin{equation}\label{small-assumption2}
    \mathcal{E}_0\leq\epsilon_3=\min\{\epsilon_2,(C(\tilde{\rho},M))^{-2}\},\quad\norm{A}_{W^{1,6}}\leq 1=\tilde{\epsilon}_1.
  \end{equation}
\end{lem}

\begin{proof}
  Multiplying $\eqref{Large-CNS-eq}_2$ by $3|u|u$ and integrating over $\Omega$ yields that
  \begin{equation}\label{u-L^3-estimate1}
    \begin{aligned}
    &\quad\frac{d}{dt}\int\rho|u|^3+3(2\mu+\lambda)\int\div u\div(u|u|)+3\mu\int\curl u\cdot\curl(u|u|)+3\mu\int_{\pa\Omega}Au\cdot u|u|\\
    &=3\int P\div(u|u|),
    \end{aligned}
  \end{equation}
  which together with \eqref{Hodge-decomposition-exterior-Final} and \eqref{curl-G-modified-estimate1} implies that
  \begin{equation}\label{u-L^3-estimate2}
    \begin{aligned}
    &\quad\frac{d}{dt}\int\rho|u|^3+3(2\mu+\lambda)\int|\div u|^2|u|+3\mu\int|\curl u|^2|u|+3\mu\int_{\pa\Omega}Au\cdot u|u|\\
    &\leq C\int|u||\nabla u|^2+C\int P|u||\nabla u|\\
    &\leq C\norm{u}_{L^6}\norm{\nabla u}_{L^2}^{\frac{3}{2}}\norm{\nabla u}_{L^6}+C\norm{u}_{L^6}\norm{\nabla u}_{L^2}\norm{P}_{L^3}\\
    &\leq C\norm{\nabla u}_{L^2}^{\frac{5}{2}}(\norm{\rho\dot{u}}_{L^2}+\norm{\nabla u}_{L^2}+\norm{P}_{L^6})^{\frac{1}{2}}+C\norm{\nabla u}_{L^2}^2\norm{P}_{L^3},
    \end{aligned}
  \end{equation}
  provided $\norm{A}_{W^{1,6}}\leq 1$.

  Then integrating over $[0,\sigma(T)]$, we have from \eqref{small-time-estimate1}, \eqref{small-time-estimate2}, \eqref{a-priori-assumption} and \eqref{essential-estimate}
  that
  \begin{equation}\label{u-L^3-estimate3}
    \begin{aligned}
    A_3(\sigma(T))&\leq C(\tilde{\rho})\sup_{t\in[0,\sigma(T)]}\norm{\nabla u}_{L^2}(\int_0^{\sigma(T)}\norm{\nabla u}_{L^2}^2)^{\frac{3}{4}}(\int_0^{\sigma(T)}\norm{\sqrt{\rho}\dot{u}}_{L^2}^2)^{\frac{1}{4}}\\
    &\quad+\int\rho_0|u_0|^3+C(\tilde{\rho},M)(\mathcal{E}_0+\mathcal{E}_0^{\frac{13}{12}})+C(\tilde{\rho})\mathcal{E}_0^{\frac{4}{3}}\\
    &\leq C(\tilde{\rho},M)\mathcal{E}_0^{\frac{3}{4}}++C(\tilde{\rho},M)\mathcal{E}_0
    +C(\tilde{\rho})\norm{\sqrt{\rho_0}u_0}_{L^2}^{\frac{3}{2}}\norm{\nabla u_0}_{L^2}^{\frac{3}{2}}\\
    &\leq C(\tilde{\rho},M)\mathcal{E}_0^{\frac{3}{4}}\leq \mathcal{E}_0^{\frac{1}{4}},
    \end{aligned}
  \end{equation}
  provided
  \begin{equation*}
    \mathcal{E}_0\leq 1,\,
    C(\tilde{\rho},M)\mathcal{E}_0^{\frac{1}{2}}\leq1,\quad\textrm{i.e.},\quad\mathcal{E}_0\leq\min\{1,(C(\tilde{\rho},M))^{-2}\}.
  \end{equation*}
  Thus, we prove \eqref{u-L^3-goal} and thus complete the proof of Lemma \ref{lem-u-L^3}.
\end{proof}

With the following lemma in hand, we can complete the proof of Proposition \ref{prop-a-priori-estimate}.
\begin{lem}\label{lem-control-A1-A2}
  Under the conditions of Proposition \ref{prop-a-priori-estimate}, there holds that
  \begin{equation}\label{control-A1-A2}
    A_1(T)\leq\mathcal{E}_0^{\frac{3}{8}},\quad A_1(T)\leq\mathcal{E}_0^{\frac{1}{2}},
  \end{equation}
  provided
  \begin{equation}\label{small-assumption3}
  \begin{aligned}
    \mathcal{E}_0\leq\epsilon_4=&\min\{\epsilon_3,(4C(\tilde{\rho}))^{-2}, (3C(\tilde{\rho}))^{-16},(3C(\tilde{\rho},M)(E_0+1))^{-2},(1+E_0)^{-\frac{16}{3}},\\
    &\qquad(4C(\tilde{\rho}))^{-\frac{128}{3}}E_0^{-\frac{56}{3}}, (4C(\tilde{\rho},M))^{-\frac{8}{5}}, (4C(\tilde{\rho})(1+E_0))^{-8}\},
  \end{aligned}
  \end{equation}
  and
  \begin{equation}\label{small-A-assumption1}
    \norm{A}_{W^{1,6}}\leq\min\left\{1,(3CE_0)^{-\frac{3}{4}}\mathcal{E}_0^{\frac{3}{32}}, E_0^{-\frac{2}{3}}\mathcal{E}_0^{\frac{7}{24}}, (4CE_0)^{-\frac{8}{9}}\mathcal{E}_0^{\frac{1}{6}}\right\}
  \end{equation}
  where $C$ depends only on $\mu,\lambda$ and $\Omega$.
\end{lem}

\begin{proof}
 Due to the lack of smallness of $\displaystyle\int_0^T\int|\nabla u|^2$, to control the bad term $\displaystyle\int_0^T\sigma^3\norm{\nabla u}_{L^4}^4$, we have to estimate $\displaystyle\int_0^T\sigma^3\norm{P}_{L^4}^4$ and $\displaystyle\int_0^T\sigma^3\norm{P}_{L^2}^4$, as discussed in Remark \ref{rem-small-A}. To begin with, we rewrite  \eqref{Pressure-eq} into
  \begin{equation}\label{Pressure-eq2}
    P_t+u\cdot\nabla P+\gamma\div u P=0.
  \end{equation}
  Multiplying \eqref{Pressure-eq2} by $3\sigma^3P^2$ and integrating the resultant equation over $\Omega\times[0,T]$, we get from the fact that $(2\mu+\lambda)\div u=G+P$ that
  \begin{equation}\label{P-L^4-estimate1}
    \begin{aligned}
    &\quad\frac{d}{dt}\int\sigma^3P^3+\frac{3\gamma-1}{2\mu+\lambda}\int\sigma^3P^4\\
    &=3\sigma^2\sigma'\int P^3-\frac{3\gamma-1}{2\mu+\lambda}\int\sigma^3P^3G\\
    &\leq3\sigma^2\sigma'\int P^3+\frac{3\gamma-1}{2(2\mu+\lambda)}\int\sigma^3P^4+\frac{27}{4}\frac{3\gamma-1}{2\mu+\lambda}\sigma^3\norm{G}_{L^4}^4,
    \end{aligned}
  \end{equation}
  which together with the following simple fact by \eqref{curl-G-modified-estimate1}, \eqref{a-priori-assumption} and \eqref{essential-energy}
  \begin{equation}\label{G-L^4-control1}
    \begin{aligned}
    \int_0^T\sigma^3\norm{G}_{L^4}^4&\leq \int_0^T\sigma^3\norm{G}_{L^2}\norm{G}_{L^6}^3\\
    &\leq C\int_0^T\sigma^3(\norm{\nabla u}_{L^2}+\norm{P}_{L^2})(\norm{\rho\dot{u}}_{L^2}^3+\norm{A}_{W^{1,6}}^3\norm{\nabla u}_{L^2}^3)\\
    &\leq C\int_0^T\sigma^3(\norm{\nabla u}_{L^2}+\norm{P}_{L^2})\norm{\rho\dot{u}}_{L^2}^3+C\norm{A}_{W^{1,6}}^3\int_0^T\sigma^3\norm{P}_{L^2}^4\\
    &\quad+C\norm{A}_{W^{1,6}}^3\int_0^T\sigma^3\norm{\nabla u}_{L^2}^4\\
    &\leq C(\tilde{\rho})(A_1^{\frac{1}{2}}(T)+\mathcal{E}_0^{\frac{1}{2}})A_1(T)A_2^{\frac{1}{2}}(T)
    +C(\tilde{\rho})\norm{A}_{W^{1,6}}^3\mathcal{E}_0\int_0^T\sigma^3\norm{P}_{L^2}^2\\
    &\quad+C\norm{A}_{W^{1,6}}^3\int_0^T\sigma^3\norm{\nabla u}_{L^2}^4,
    \end{aligned}
  \end{equation}
  and inequalities \eqref{a-priori-assumption} and \eqref{essential-energy} yields that
  \begin{equation}\label{P-L^4-estimate2}
    \begin{aligned}
    \int_0^T\int\sigma^3P^4&\leq C(\tilde{\rho})\mathcal{E}_0+C(\tilde{\rho})A_1^{\frac{3}{2}}(T)A_2^{\frac{1}{2}}(T)
    +C(\tilde{\rho})\norm{A}_{W^{1,6}}^3\mathcal{E}_0\int_0^T\sigma^3\norm{P}_{L^2}^2\\
    &\quad+C\norm{A}_{W^{1,6}}^3\int_0^T\sigma^3\norm{\nabla u}_{L^2}^4.
    \end{aligned}
  \end{equation}

  Next, we derive the estimate on $\displaystyle\int_0^T\sigma^3\norm{P}_{L^2}^2$. By introducing the Bogovskii operator $\mathcal{B}[P]$ as in Lemma \ref{lem-Bogovskii-operator}, multiplying $\eqref{Large-CNS-eq}_2$ by $\mathcal{B}[P]$ and integrating over $\Omega$ gives that
  \begin{equation}\label{P-L^2-estimate1}
    \begin{aligned}
    \int P^2&=\int(\rho u)_t\mathcal{B}[P]-\int\rho(u\otimes u):\nabla\mathcal{B}[P]+\mu\int\nabla u:\nabla\mathcal{B}[P]+(\mu+\lambda)\int\div u\div\mathcal{B}[P]\\
    &\leq\left(\int\rho u\mathcal{B}[P]\right)_t+\int\rho u\mathcal{B}[\div(Pu)]+(\gamma-1)\int\rho u\mathcal{B}[P\div u]\\
    &\quad+C\norm{\nabla u}_{L^2}\norm{P}_{L^2}+C(\tilde{\rho})\norm{u}_{L^6}^2\norm{P}_{L^{\frac{3}{2}}}\\
    &\leq\left(\int\rho u\mathcal{B}[P]\right)_t+C\norm{\rho}_{L^{\frac{3}{2}}}\norm{u}_{L^6}\norm{Pu}_{L^6}
    +C\norm{\rho}_{L^{\frac{3}{2}}}\norm{u}_{L^6}\norm{\mathcal{B}[P\div u]}_{L^6}\\
    &\quad+C\norm{\nabla u}_{L^2}\norm{P}_{L^2}+C(\tilde{\rho})\norm{u}_{L^6}^2\norm{P}_{L^{\frac{3}{2}}}\\
    &\leq\left(\int\rho u\mathcal{B}[P]\right)_t+C(\tilde{\rho})\mathcal{E}_0^{\frac{2}{3}}\norm{\nabla u}_{L^2}^2+\frac{1}{2}\norm{P}_{L^2}^2+C\norm{\nabla u}_{L^2}^2,
    \end{aligned}
  \end{equation}
  where we have used Gagliardo-Nirenberg inequality \eqref{GN-inequality1}, \eqref{a-priori-assumption}, Lemma \ref{lem-Bogovskii-operator}, \eqref{assume-condition1}, \eqref{essential-energy}.

  Then, it follows from \eqref{P-L^2-estimate1}, Lemma \eqref{lem-Bogovskii-operator}, \eqref{assume-condition1}, \eqref{a-priori-assumption}, \eqref{GN-inequality1}, \eqref{essential-energy} and \eqref{essential-estimate} that
  \begin{equation}\label{P-L^2-estimate2}
    \begin{aligned}
    \int_0^T\sigma^3\norm{P}_{L^2}^2&\leq C\int_0^{\sigma(T)}\sigma^2\norm{\rho}_{L^{\frac{3}{2}}}\norm{u}_{L^6}\norm{\mathcal{B}[P]}_{L^6}+(C(\tilde{\rho})\mathcal{E}_0^{\frac{2}{3}}+C)\int_0^T\sigma^3\norm{\nabla u}_{L^2}^2\\
    &\leq C(\tilde{\rho})\mathcal{E}_0^{\frac{5}{3}}+(C(\tilde{\rho})\mathcal{E}_0^{\frac{2}{3}}+C)\int_0^T\sigma^3\norm{\nabla u}_{L^2}^2.
    \end{aligned}
  \end{equation}

  Next, we return to the estimate on $A_2(T)$ in \eqref{A2-control}. By virtue of \eqref{P-L^4-estimate2}, \eqref{P-L^2-estimate2}, \eqref{new-gradient-u-estimate} and \eqref{nabla-u-L^4-estimate2}, we have
  \begin{equation}\label{A2-estimate-1}
    \begin{aligned}
    &\quad\int_0^T\sigma^3(\norm{\nabla u}_{L^4}^4+\norm{P|\nabla u|}_{L^2}^2+\norm{\nabla u}_{L^{\frac{8}{3}}}^4+\norm{A}_{W^{1,6}}^{\frac{4}{3}}\norm{\nabla u}_{L^2}^4)\\
    &\leq C\int_0^T\sigma^3(\norm{\nabla u}_{L^4}^4+\norm{P}_{L^4}^4+\norm{\nabla u}_{L^{\frac{8}{3}}}^4+\norm{A}_{W^{1,6}}^{\frac{4}{3}}\norm{\nabla u}_{L^2}^4)\\
    &\leq C\int_0^T\sigma^3[(\norm{\nabla u}_{L^2}+\norm{P}_{L^2})\norm{\rho\dot{u}}_{L^2}^3+\norm{\rho\dot{u}}_{L^2}^{\frac{3}{2}}\norm{\nabla u}_{L^2}^{\frac{5}{2}}+\norm{A}_{W^{1,6}}^{\frac{4}{3}}\norm{\nabla u}_{L^2}^4+\norm{P}_{L^4}^4+\norm{P}_{L^2}^4]\\
    &\leq C(\tilde{\rho})(A_1^{\frac{1}{2}}(T)+\mathcal{E}_0^{\frac{1}{2}})A_1(T)A_2^{\frac{1}{2}}(T)+C(\tilde{\rho})A_1^{\frac{7}{4}}(T)E_0^{\frac{1}{4}}
    +C\norm{A}_{W^{1,6}}^{\frac{4}{3}}A_1(T)E_0+C(\tilde{\rho})\mathcal{E}_0\\
    &\quad+(C(\tilde{\rho})\norm{A}_{W^{1,6}}^3+C)\mathcal{E}_0(C(\tilde{\rho})\mathcal{E}_0^{\frac{5}{3}}+(C(\tilde{\rho})\mathcal{E}_0^{\frac{2}{3}}+C)E_0)\\
    &\leq C(\tilde{\rho})\mathcal{E}_0(1+E_0)+C(\tilde{\rho})A_1^{\frac{3}{2}}(T)A_2^{\frac{1}{2}}(T)+C(\tilde{\rho})A_1^{\frac{7}{4}}(T)E_0^{\frac{1}{4}}
    +C\norm{A}_{W^{1,6}}^{\frac{4}{3}}A_1(T)E_0,
    \end{aligned}
  \end{equation}
  where we have used $\norm{A}_{W^{1,6}}\leq 1$, \eqref{a-priori-assumption}, \eqref{essential-energy}, and $\mathcal{E}_0\leq 1$.

  Then, subsituting \eqref{A2-estimate-1} into \eqref{A2-control} yields
  \begin{equation}\label{A2-estimate-2}
    \begin{aligned}
     A_2(T)&\leq C(\tilde{\rho})\mathcal{E}_0+CA_1^{\frac{3}{2}}(T)+CA_1(\sigma(T))+C(\tilde{\rho})(A_1^2(T)+A_1^{\frac{5}{3}}(T)E_0^{\frac{1}{3}}+A_1^2(T)E_0^{\frac{1}{2}})\\
    &\quad+C(\tilde{\rho})\mathcal{E}_0(1+E_0)+C(\tilde{\rho})A_1^{\frac{3}{2}}(T)A_2^{\frac{1}{2}}(T)+C(\tilde{\rho})A_1^{\frac{7}{4}}(T)E_0^{\frac{1}{4}}
    +C\norm{A}_{W^{1,6}}^{\frac{4}{3}}A_1(T)E_0\\
    &\leq C(\tilde{\rho})\mathcal{E}_0(1+E_0)+C(\tilde{\rho})A_1^{\frac{3}{2}}(T)(1+A_1^{\frac{1}{6}}(T)E_0^{\frac{1}{3}}+A_1^{\frac{1}{2}}(T)E_0^{\frac{1}{2}})+C\norm{A}_{W^{1,6}}^{\frac{4}{3}}A_1(T)E_0\\
    &\quad+CA_1(\sigma(T)),
    \end{aligned}
  \end{equation}
  provided $\mathcal{E}_0\leq 1$.

  Recalling \eqref{A-control3} and taking $m=1$, we have
  \begin{equation}\label{A1-estimate-small-T1}
    \begin{aligned}
      &\quad A_1(\sigma(T))\\
      &\leq C(\tilde{\rho})\mathcal{E}_0+C\int_0^{\sigma(T)}\int\sigma (P|\nabla u|^2+|\nabla u|^3)
      +C\int_0^{\sigma(T)}\sigma(\norm{\rho\dot{u}}_{L^2}\norm{\nabla u}_{L^2}^2+\norm{A}_{W^{1,6}}\norm{\nabla u}_{L^2}^3)\\
      &\leq C(\tilde{\rho})\mathcal{E}_0+C(\tilde{\rho})A_1(\sigma(T))\mathcal{E}_0^{\frac{1}{2}}
      +C(\tilde{\rho})\norm{A}_{W^{1,6}}A_1^{\frac{1}{2}}(\sigma(T))\mathcal{E}_0
     +C\int_0^{\sigma(T)}\sigma\norm{\nabla u}_{L^2}^{\frac{3}{2}}\norm{\nabla u}_{L^6}^{\frac{3}{2}}\\
      &\leq C(\tilde{\rho})\mathcal{E}_0+C(\tilde{\rho})A_1(\sigma(T))\mathcal{E}_0^{\frac{1}{2}}
      +C\int_0^{\sigma(T)}\sigma\norm{\nabla u}_{L^2}^{\frac{3}{2}}(\norm{\rho\dot{u}}_{L^2}+\norm{P}_{L^6}+\norm{\nabla u}_{L^2})^{\frac{3}{2}}\\
      &\leq C(\tilde{\rho})\mathcal{E}_0+C(\tilde{\rho})A_1(\sigma(T))\mathcal{E}_0^{\frac{1}{2}}
      +C(\int_0^{\sigma(T)}\norm{\nabla u}_{L^2}^2)^{\frac{3}{4}}(\int_0^{\sigma(T)}\norm{P}_{L^6}^6)^{\frac{1}{4}}
      +C(\tilde{\rho})A_1^{\frac{1}{2}}(\sigma(T))\mathcal{E}_0\\
      &\quad+C(\tilde{\rho})\sup_{t\in[0,\sigma(T)]}\sigma^{\frac{1}{2}}\norm{\sqrt{\rho}\dot{u}}_{L^2}(\int_0^{\sigma(T)}\norm{\nabla u}_{L^2}^2)^{\frac{3}{4}}(\int_0^{\sigma(T)}\sigma^2\norm{\sqrt{\rho}\dot{u}}_{L^2}^2)^{\frac{1}{4}}\\
    &\leq C(\tilde{\rho})\mathcal{E}_0+C(\tilde{\rho})A_1(\sigma(T))\mathcal{E}_0^{\frac{1}{2}}
    +C(\tilde{\rho},M)A_1^{\frac{1}{4}}(\sigma(T))\mathcal{E}_0^{\frac{3}{4}}\\
    &\leq (C(\tilde{\rho})\mathcal{E}_0^{\frac{1}{2}}+\frac{1}{4})A_1(\sigma(T))+C(\tilde{\rho},M)\mathcal{E}_0,
    \end{aligned}
  \end{equation}
  where we have applied
  \eqref{essential-estimate}, \eqref{Hodge-decomposition-exterior-Final}, \eqref{curl-G-modified-estimate1}, \eqref{essential-energy}, \eqref{a-priori-assumption}, \eqref{small-time-estimate2}, and the assumptions that $\mathcal{E}_0\leq 1$ and $\norm{A}_{W^{1,6}}\leq1$. The above inequality implies that
  \begin{equation}\label{A1-estimate-small-T2}
    A_1(\sigma(T))\leq C(\tilde{\rho},M)\mathcal{E}_0,
  \end{equation}
  provided
  \begin{equation}\label{small-condition-tran1}
    C(\tilde{\rho})\mathcal{E}_0^{\frac{1}{2}}\leq\frac{1}{4},\,\textrm{i.e.},\,\mathcal{E}_0\leq(4C(\tilde{\rho}))^{-2}.
  \end{equation}
  Then, we turn back to \eqref{A2-estimate-2} and get
  \begin{equation}\label{A2-estimate-3}
    \begin{aligned}
    A_2(T)&\leq C(\tilde{\rho},M)\mathcal{E}_0(1+E_0)+C(\tilde{\rho})A_1^{\frac{3}{2}}(T)(1+A_1^{\frac{1}{6}}(T)E_0^{\frac{1}{3}}
    +A_1^{\frac{1}{2}}(T)E_0^{\frac{1}{2}})\\
    &\quad+C\norm{A}_{W^{1,6}}^{\frac{4}{3}}A_1(T)E_0\\
    &\leq C(\tilde{\rho})A_1^{\frac{3}{2}}(T)+C(\tilde{\rho},M)\mathcal{E}_0(1+E_0)+C\norm{A}_{W^{1,6}}^{\frac{4}{3}}A_1(T)E_0\\
    &\leq C(\tilde{\rho})\mathcal{E}_0^{\frac{9}{16}}+C(\tilde{\rho},M)\mathcal{E}_0(1+E_0)+C\norm{A}_{W^{1,6}}^{\frac{4}{3}}\mathcal{E}_0^{\frac{3}{8}}E_0\\
    &\leq\mathcal{E}_0^{\frac{1}{2}},
    \end{aligned}
  \end{equation}
  provided
  \begin{equation*}
    \mathcal{E}_0^{\frac{3}{8}}E_0^2\leq 1,\quad \mathcal{E}_0\leq 1,\quad
    C(\tilde{\rho})\mathcal{E}_0^{\frac{1}{16}}\leq\frac{1}{3},\quad C(\tilde{\rho},M)\mathcal{E}_0^{\frac{1}{2}}(1+E_0)\leq\frac{1}{3},\quad
    C\norm{A}_{W^{1,6}}^{\frac{4}{3}}E_0\leq\frac{1}{3}\mathcal{E}_0^{\frac{1}{8}},
  \end{equation*}
  namely,
  \begin{equation}\label{small-condition-tran2}
    \mathcal{E}_0\leq\min\{1, E_0^{-\frac{16}{3}}, (3C(\tilde{\rho}))^{-16},(3C(\tilde{\rho},M)(E_0+1))^{-2}\},
  \end{equation}
  and
  \begin{equation}\label{small-condition-tran3}
    \norm{A}_{W^{1,6}}\leq\min\{1, (3CE_0)^{-\frac{3}{4}}\mathcal{E}_0^{\frac{3}{32}}\}.
  \end{equation}
  Thus, we finish the estimate on $A_2(T)$.

  It is easy to check that under the condition \eqref{small-condition-tran1} and \eqref{small-condition-tran2},
  \begin{equation}\label{A1-estimate-small-T3}
    A_1(\sigma(T))\leq\mathcal{E}_0^{\frac{1}{2}}\leq\mathcal{E}_0^{\frac{3}{8}}
  \end{equation}
  and also by \eqref{P-L^4-estimate2}, \eqref{P-L^2-estimate2}, \eqref{nabla-u-L^4-estimate2} and \eqref{A2-estimate-1},
  \begin{equation}\label{nabla-u-P-L^4-estimate1}
  \begin{aligned}
    \int_0^T\int\sigma^3P^4&\leq C(\tilde{\rho})A_1^{\frac{3}{2}}(T)A_2^{\frac{1}{2}}(T)+C(\tilde{\rho})\mathcal{E}_0(1+E_0)+C\norm{A}_{W^{1,6}}^3A_1(T)E_0\\
    &\leq C(\tilde{\rho})\mathcal{E}_0^{\frac{13}{16}}+C(\tilde{\rho})\mathcal{E}_0(1+E_0)+C\norm{A}_{W^{1,6}}^3\mathcal{E}_0^{\frac{3}{8}}E_0\\
    &\leq C(\tilde{\rho})\mathcal{E}_0^{\frac{13}{16}},
  \end{aligned}
  \end{equation}
  \begin{equation}\label{nabla-u-P-L^4-estimate2}
  \begin{aligned}
    \int_0^T\sigma^3\norm{\nabla u}_{L^4}^4&\leq C(\tilde{\rho})A_1^{\frac{3}{2}}(T)A_2^{\frac{1}{2}}(T)+C(\tilde{\rho})A_1^{\frac{7}{4}}(T)E_0^{\frac{1}{4}}\\
    &\quad+C\norm{A}_{W^{1,6}}^{\frac{3}{2}}A_1(T)E_0
    +C(\tilde{\rho})\mathcal{E}_0(1+E_0)\\
    &\leq C(\tilde{\rho})\mathcal{E}_0^{\frac{13}{16}}+C(\tilde{\rho})A_1^{\frac{7}{4}}(T)E_0^{\frac{1}{4}},
  \end{aligned}
  \end{equation}
  provided
  \begin{equation*}
    \mathcal{E}_0^{\frac{3}{16}}(1+E_0)\leq 1,\quad \norm{A}_{W^{1,6}}^{\frac{3}{2}}E_0\leq\mathcal{E}_0^{\frac{7}{16}},\quad\norm{A}_{W^{1,6}}\leq 1,
  \end{equation*}
  namely,
  \begin{equation}\label{small-condition-tran4}
    \mathcal{E}_0\leq(1+E_0)^{-\frac{16}{3}},\quad\norm{A}_{W^{1,6}}\leq\min\{1, E_0^{-\frac{2}{3}}\mathcal{E}_0^{\frac{7}{24}}\}.
  \end{equation}
  To estimate $A_1(T)$, it suffices to control $\displaystyle\int_{\sigma(T)}^T\int\sigma P^3$ and $\displaystyle\int_{\sigma(T)}^T\sigma\norm{\nabla u}_{L^3}^3$.
  From \eqref{nabla-u-P-L^4-estimate1}, \eqref{P-L^2-estimate2} and \eqref{essential-energy}, we obtain\footnote{During this calculation, we observe that $A_1(T)^{\frac{3}{4}}A_2^{\frac{1}{4}}(T)E_0^{\frac{1}{2}}\ll A_1(T)$ needs $A_2(T)\ll A_1(T)$. And in \eqref{A2-estimate-2} we also need $A^{\frac{3}{2}}(T)\ll A_2(T)$. Thus we can ensure the setting of a priori assumption \eqref{a-priori-assumption}.}
  \begin{equation}\label{A1-estimate-1}
    \begin{aligned}
    \int_{\sigma(T)}^T\int\sigma P^3&\leq\left(\int_{\sigma(T)}^T\norm{P}_{L^4}^4\right)^{\frac{1}{2}}\left(\int_{\sigma(T)}^T\norm{P}_{L^2}^2\right)^{\frac{1}{2}}\\
    &\leq C(\tilde{\rho})\mathcal{E}_0^{\frac{13}{32}}(1+E_0)^{\frac{1}{2}}.
    \end{aligned}
  \end{equation}
  However, a similar interpolation inequality yields that
  \begin{equation}\label{A1-estimate-2}
    \begin{aligned}
    \int_{\sigma(T)}^T\sigma\norm{\nabla u}_{L^3}^3&\leq\int_{\sigma(T)}^T\norm{\nabla u}_{L^2}\norm{\nabla u}_{L^4}^2\leq\left(\int_{\sigma(T)}^T\norm{\nabla u}_{L^2}^2\right)^{\frac{1}{2}}\left(\int_{\sigma(T)}^T\norm{\nabla u}_{L^4}^4\right)^{\frac{1}{2}}
    \end{aligned}
  \end{equation}
  which is unlikely\footnote{The processes in the interpolation and the control of $\int_0^T\sigma^3\norm{\nabla u}_{L^4}^4$ both use the bad term $\int_0^T\norm{\nabla u}_{L^2}^2$, and then make the estimate on $\int_0^T\sigma\norm{\nabla u}_{L^3}^3$ lose more smallness.} to be smaller than $A_1(T)$ due to the term $A_1^{\frac{7}{4}}(T)E_0^{\frac{1}{4}}$ in \eqref{nabla-u-P-L^4-estimate2}. Therefore, we must pursue
another route by resorting to the boundary-adapted nonlinear localization technique from Remark \ref{rem-small-A}.

  We now return to the estimate on $\norm{\nabla u}_{L^4}^4$ as discussed in Remark \ref{rem-small-A}. Similar as \eqref{nabla-u-L^4-estimate2}, it holds that
  \begin{equation}\label{nabla-u-L^3-estimate}
    \begin{aligned}
    \norm{\nabla u}_{L^3}^3&\leq C(\norm{G}_{L^3}+\norm{\curl u}_{L^3}+\norm{P}_{L^3}+\norm{\nabla u}_{L^{\frac{8}{3}}})^3\\
    &\leq C(\norm{G}_{L^2}^{\frac{3}{2}}\norm{\nabla G}_{L^2}^{\frac{3}{2}}+\norm{\curl u}_{L^2}^{\frac{3}{2}}\norm{\nabla\curl u}_{L^2}^{\frac{3}{2}}+\norm{P}_{L^3}^3+\norm{\nabla u}_{L^{\frac{8}{3}}}^3)\\
    &\leq C(\norm{\nabla u}_{L^2}+\norm{P}_{L^2})^{\frac{3}{2}}(\norm{\rho\dot{u}}_{L^2}+\norm{A}_{W^{1,6}}\norm{\nabla u}_{L^2})^{\frac{3}{2}}+C(\norm{P}_{L^3}^3+\norm{P}_{L^{\frac{8}{3}}}^3)\\
    &\quad+C(\norm{\nabla u}_{L^2}+\norm{P}_{L^2})^{\frac{15}{8}}(\norm{\rho\dot{u}}_{L^2}+\norm{A}_{W^{1,6}}\norm{\nabla u}_{L^2})^{\frac{9}{8}}\\
    &\leq C(\norm{\nabla u}_{L^2}+\norm{P}_{L^2})^{\frac{3}{2}}\norm{\rho\dot{u}}_{L^2}^{\frac{3}{2}}+C(\norm{\nabla u}_{L^2}+\norm{P}_{L^2})^{\frac{15}{8}}\norm{\rho\dot{u}}_{L^2}^{\frac{9}{8}}\\
    &\quad+C\norm{A}_{W^{1,6}}^{\frac{9}{8}}\norm{\nabla u}_{L^2}^3+C(\norm{P}_{L^3}^3+\norm{P}_{L^2}^3),
    \end{aligned}
  \end{equation}
  provided $\norm{A}_{W^{1,6}}\leq 1$.

  Thus, we get from \eqref{nabla-u-L^3-estimate}, \eqref{A1-estimate-1}, \eqref{P-L^2-estimate2}, \eqref{essential-energy}, \eqref{essential-estimate} and \eqref{a-priori-assumption} that
  \begin{equation}\label{A1-estimate-3}
    \begin{aligned}
    \int_{\sigma(T)}^T\sigma\norm{\nabla u}_{L^3}^3&\leq C(\tilde{\rho})(A_1^{\frac{7}{4}}(T)E_0^{\frac{1}{4}}+A_1^{\frac{3}{4}}(T)\mathcal{E}_0^{\frac{1}{2}}(1+E_0)^{\frac{1}{4}}
    +A_1^{\frac{17}{16}}(T)E_0^{\frac{7}{16}}+A_1^{\frac{9}{16}}(T)\mathcal{E}_0^{\frac{1}{2}}(1+E_0)^{\frac{7}{16}})\\
    &\quad+C\norm{A}_{W^{1,6}}^{\frac{9}{8}}A_1^{\frac{1}{2}}(T)E_0+C(\tilde{\rho})\mathcal{E}_0^{\frac{13}{32}}(1+E_0)^{\frac{1}{2}}
    +C(\tilde{\rho})\mathcal{E}_0^{\frac{1}{2}}(1+E_0)\\
    &\leq C(\tilde{\rho})A_1^{\frac{17}{16}}(T)E_0^{\frac{7}{16}}
    +C\norm{A}_{W^{1,6}}^{\frac{9}{8}}A_1^{\frac{1}{2}}(T)E_0+C(\tilde{\rho})\mathcal{E}_0^{\frac{1}{2}}(1+E_0),
    \end{aligned}
  \end{equation}
  provided $\norm{A}_{W^{1,6}}\leq 1$.

  Then, plugging \eqref{A1-estimate-1}, \eqref{A1-estimate-3}, \eqref{A1-estimate-small-T1}, \eqref{A1-estimate-small-T2} and \eqref{a-priori-assumption} into \eqref{A1-control} yields that
  \begin{equation}\label{A1-estimate-4}
    \begin{aligned}
    A_1(T)&\leq C(\tilde{\rho})\mathcal{E}_0+C(\tilde{\rho})A_1^{\frac{3}{2}}(T)+C\norm{A}_{W^{1,6}}^{\frac{3}{2}}A_1^{\frac{1}{2}}(T)E_0+C(\tilde{\rho},M)\mathcal{E}_0
    +C(\tilde{\rho})\mathcal{E}_0^{\frac{13}{32}}(1+E_0)^{\frac{1}{2}}\\
    &\quad+C(\tilde{\rho})A_1^{\frac{17}{16}}(T)E_0^{\frac{7}{16}}
    +C\norm{A}_{W^{1,6}}^{\frac{9}{8}}A_1^{\frac{1}{2}}(T)E_0+C(\tilde{\rho})\mathcal{E}_0^{\frac{1}{2}}(1+E_0)\\
    &\leq C(\tilde{\rho},M)\mathcal{E}_0+C(\tilde{\rho})\mathcal{E}_0^{\frac{3}{8}}(\mathcal{E}_0^{\frac{3}{128}}E_0^{\frac{7}{16}}+\mathcal{E}_0^{\frac{1}{8}}(1+E_0))
    +C\norm{A}_{W^{1,6}}^{\frac{9}{8}}\mathcal{E}_0^{\frac{3}{16}}(T)E_0,\\
    &\leq \mathcal{E}_0^{\frac{3}{8}},
    \end{aligned}
  \end{equation}
  provided
  \begin{equation}\label{small-condition-tran5}
    \mathcal{E}_0\leq\min\{1, (4C(\tilde{\rho}))^{-\frac{128}{3}}E_0^{-\frac{56}{3}}, (4C(\tilde{\rho},M))^{-\frac{8}{5}}, (4C(\tilde{\rho})(1+E_0))^{-8}\},
  \end{equation}
  and
  \begin{equation}\label{small-condition-tran6}
    \norm{A}_{W^{1,6}}\leq\min\{1,  (4CE_0)^{-\frac{8}{9}}\mathcal{E}_0^{\frac{1}{6}}\}.
  \end{equation}
  Thus, we have completed the proof of Lemma \ref{lem-control-A1-A2}.
\end{proof}

\begin{rem}
  Note that in this lemma, the condition $\norm{\nabla u_0}_{L^2}\leq M$ is only applied in \eqref{A1-estimate-small-T1} through Lemma \ref{lem-small-time-estimate}.
\end{rem}

The following lemma concerns the bound on density $\rho$.

\begin{lem}\label{lem-rho-bound}
  Under the conditions of Proposition \ref{prop-a-priori-estimate}, it holds that for any $(x,t)\in\Omega\times[0,T]$,
  \begin{equation}\label{rho-bound}
    0\leq\rho(x,t)\leq\frac{7\tilde{\rho}}{4},
  \end{equation}
  provided
  \begin{equation}\label{small-assumption4}
    \mathcal{E}_0\leq\epsilon_5=\min\left\{\epsilon_4,\left(\frac{\tilde{\rho}}{2C(\tilde{\rho},M)}\right)^{-\frac{32}{3}},\left(\frac{\tilde{\rho}}{4C(\tilde{\rho})(1+E_0)}\right)^8\right\},
  \end{equation}
  and
  \begin{equation}\label{small-A-assumption2}
    \norm{A}_{W^{1,6}}\leq 1,\quad\norm{A}_{W^{1,\infty}}\leq\mathcal{E}_0^{-\frac{1}{8}}.
  \end{equation}
\end{lem}

\begin{proof}
  First, we rewrite the equation of mass conservation $\eqref{Large-CNS-eq}_1$ as
  \begin{equation}\label{mass-equation}
    D_t\rho=g(\rho)+b'(t),
  \end{equation}
  where
  \begin{equation*}
    D_t\rho=\rho_t+u\cdot\nabla\rho,\quad g(\rho)=-\frac{\rho P}{2\mu+\lambda},\quad b(t)=\frac{-1}{2\mu+\lambda}\int_0^t\rho G.
  \end{equation*}
  Then for $t\in[0,\sigma(T)]$, we obtain from \eqref{essential-energy}, \eqref{a-priori-assumption}, \eqref{GN-inequality2}, \eqref{A1-estimate-small-T3}, Remark \ref{rem-small-A}, and Lemma \ref{lem-small-time-estimate} that for any $0\leq t_1<t_2\leq\sigma(T)$,
  \begin{equation}\label{b-estimate1}
    \begin{aligned}
    &\quad|b(t_2)-b(t_1)|\leq C(\tilde{\rho})\int_0^{\sigma(T)}\norm{G}_{L^{\infty}}
    \leq C(\tilde{\rho})\int_0^{\sigma(T)}\norm{G}_{L^6}^{\frac{1}{2}}\norm{\nabla G}_{L^6}^{\frac{1}{2}}\\
    &\leq C(\tilde{\rho})\int_0^{\sigma(T)}(\norm{\rho\dot{u}}_{L^2}+\norm{A}_{W^{1,6}}\norm{\nabla u}_{L^2})^{\frac{1}{2}}(\norm{\rho\dot{u}}_{L^6}+\norm{A}_{W^{1,6}}\norm{\rho\dot{u}}_{L^2}\\
    &\qquad\qquad+(1+\norm{A}_{W^{1,\infty}})\norm{\nabla u}_{L^2}+\norm{P}_{L^6})^{\frac{1}{2}}\\
    &\leq C(\tilde{\rho})\left(\int_0^{\sigma(T)}t\norm{\nabla\dot{u}}_{L^2}^2\right)^{\frac{1}{4}}
    \left(\int_0^{\sigma(T)}t^{-\frac{1}{3}}\norm{\sqrt{\rho}\dot{u}}_{L^2}^{\frac{2}{3}}\right)^{\frac{3}{4}}
    +C(\tilde{\rho})A_1^{\frac{1}{4}}(\sigma(T))\int_0^{\sigma(T)}t^{-\frac{1}{2}}(t\norm{\nabla\dot{u}}_{L^2}^2)^{\frac{1}{4}}\\
    &\quad+C(\tilde{\rho},M)\int_0^{\sigma(T)}t^{-\frac{1}{2}}(t\norm{\sqrt{\rho}\dot{u}}_{L^2}^2)^{\frac{1}{4}}
    +C(\tilde{\rho})\int_0^{\sigma(T)}((1+\norm{A}_{W^{1,\infty}})\norm{\nabla u}_{L^2}+\norm{P}_{L^6})\\
    &\leq C(\tilde{\rho},M)\left(\int_0^{\sigma(T)}t^{-\frac{2}{3}}(t\norm{\sqrt{\rho}\dot{u}}_{L^2}^2)^{\frac{1}{4}}\right)^{\frac{3}{4}}
    +C(\tilde{\rho},M)A_1^{\frac{1}{4}}(\sigma(T))+C(\tilde{\rho})((1+\norm{A}_{W^{1,\infty}})\mathcal{E}_0^{\frac{1}{2}}+\mathcal{E}_0^{\frac{1}{6}})\\
    &\leq C(\tilde{\rho},M)A_1^{\frac{3}{16}}(\sigma(T))+C(\tilde{\rho})(\mathcal{E}_0^{\frac{1}{6}}+\norm{A}_{W^{1,\infty}}\mathcal{E}_0^{\frac{1}{2}})\\
    &\leq C(\tilde{\rho},M)(\mathcal{E}_0^{\frac{3}{32}}+\mathcal{E}_0^{\frac{1}{6}})\leq C(\tilde{\rho},M)\mathcal{E}_0^{\frac{3}{32}},
    \end{aligned}
  \end{equation}
   provided
   \begin{equation}\label{rho-small-transition1}
     \mathcal{E}_0\leq\epsilon_4,\quad \norm{A}_{W^{1,6}}\leq 1,\quad \norm{A}_{W^{1,\infty}}\leq\mathcal{E}_0^{-\frac{1}{3}}.
   \end{equation}

   Then, by choosing $N_1=0$, $N_0=C(\tilde{\rho},M)\mathcal{E}_0^{\frac{3}{32}}$, and $\zeta_0=\tilde{\rho}$ in Lemma \ref{lem-Zlotnik-inequality}, we have from \eqref{mass-equation} and \eqref{b-estimate1} that
   \begin{equation}\label{rho-bound-small-T}
     \sup_{t\in[0,\sigma(T)]}\norm{\rho}_{L^{\infty}}\leq\tilde{\rho}+C(\tilde{\rho},M)\mathcal{E}_0^{\frac{3}{32}}\leq\frac{3\tilde{\rho}}{2},
   \end{equation}
   provided
   \begin{equation}\label{rho-small-transition2}
     \mathcal{E}_0\leq\left(\frac{\tilde{\rho}}{2C(\tilde{\rho},M)}\right)^{-\frac{32}{3}}.
   \end{equation}

   For $t\in[\sigma(T),T]$ and any $\sigma(T)\leq t_1<t_2\leq T$, we also have
   \begin{equation}\label{b-estimate2}
     \begin{aligned}
     |b(t_2)-b(t_1)|&\leq C\tilde{\rho}\int_{t_1}^{t_2}\norm{G}_{L^{\infty}}\\
     &\leq \frac{\tilde{\rho}P(\tilde{\rho})}{2\mu+\lambda}(t_2-t_1)+C(\tilde{\rho})\int_{\sigma(T)}^{T}\norm{G}_{L^{\infty}}^4\\
     &\leq \frac{\tilde{\rho}P(\tilde{\rho})}{2\mu+\lambda}(t_2-t_1)+C(\tilde{\rho})\mathcal{E}_0^{\frac{1}{8}}(1+E_0),
     \end{aligned}
   \end{equation}
   where we have used Remark \ref{rem-small-A},  \eqref{essential-energy} and \eqref{a-priori-assumption} to get
   \begin{equation*}
     \begin{aligned}
     &\quad\int_{\sigma(T)}^{T}\norm{G}_{L^{\infty}}^4
     \leq\int_{\sigma(T)}^{T}\norm{G}_{L^6}^2\norm{\nabla G}_{L^6}^2\\
     &\leq C\int_{\sigma(T)}^{T}(\norm{\rho\dot{u}}_{L^2}^2+\norm{\nabla u}_{L^2}^2)(\norm{\rho\dot{u}}_{L^6}^2+\norm{\rho\dot{u}}_{L^2}^2+(1+\norm{A}_{W^{1,\infty}}^2)\norm{\nabla u}_{L^2}^2+\norm{P}_{L^6}^2)\\
     &\leq C(\tilde{\rho})\int_{\sigma(T)}^{T}(\norm{\sqrt{\rho}\dot{u}}_{L^2}^2+\norm{\nabla u}_{L^2}^2)(\norm{\nabla\dot{u}}_{L^2}^2+\norm{\sqrt{\rho}\dot{u}}_{L^2}^2+(1+\norm{A}_{W^{1,\infty}}^2)\norm{\nabla u}_{L^2}^2+\norm{P}_{L^6}^2)\\
     &\leq C(\tilde{\rho})(A_1(T)A_2(T)+A_2^2(T)+(1+\norm{A}_{W^{1,\infty}}^2)A_1^2(T)+A_1(T)\mathcal{E}_0^{\frac{1}{3}})\\
     &\quad+C(\tilde{\rho})((1+\norm{A}_{W^{1,\infty}}^2)A_1(T)+\mathcal{E}_0^{\frac{1}{3}})E_0\\
     &\leq C(\tilde{\rho})\mathcal{E}_0^{\frac{1}{8}}(1+E_0),
     \end{aligned}
   \end{equation*}
   provided
   \begin{equation}\label{rho-small-transition3}
     \mathcal{E}_0\leq 1,\quad\norm{A}_{W^{1,6}}\leq 1,\quad\norm{A}_{W^{1,\infty}}\leq\mathcal{E}_0^{-\frac{1}{8}}.
   \end{equation}

   Therefore, by choosing $N_0=C(\tilde{\rho})\mathcal{E}_0^{\frac{1}{8}}(1+E_0)$, $N_1=\frac{\tilde{\rho}P(\tilde{\rho})}{2\mu+\lambda}$ and $\zeta_0=\tilde{\rho}$ in Lemma \ref{lem-Zlotnik-inequality}, we have from Lemma \ref{lem-Zlotnik-inequality}, \eqref{rho-bound-small-T} and \eqref{b-estimate2} that
   \begin{equation}\label{rho-bound-large-T}
     \sup_{t\in[\sigma(T),T]}\norm{\rho}_{L^{\infty}}\leq\frac{3\tilde{\rho}}{2}
     +C(\tilde{\rho})\mathcal{E}_0^{\frac{1}{8}}(1+E_0)\leq\frac{7\tilde{\rho}}{4},
   \end{equation}
   provided
   \begin{equation}\label{rho-small-transition4}
     \mathcal{E}_0\leq\left(\frac{\tilde{\rho}}{4C(\tilde{\rho})(1+E_0)}\right)^8.
   \end{equation}
   Then, combining \eqref{rho-bound-small-T} and \eqref{rho-bound-large-T}, we complete the proof of Lemma \ref{lem-rho-bound}.
\end{proof}

\noindent\textbf{Proof of Theorem \ref{thm-global-CNS-exterior}} In the following, we will prove the main results of this paper. First of all, we derive the time-dependent higher-order estimates of the smooth solution $(\rho, u)$.
From now on, we will always assume that \eqref{small-assumption4} holds and denote the positive constant by $C$ depending on
\begin{equation*}
  T,\,\norm{g}_{L^2},\,\norm{\nabla u_0}_{H^1},\,\norm{\rho_0}_{W^{2,q}},\,\norm{P(\rho)}_{W^{2,q}},
\end{equation*}
for $q\in(3,6)$, as well as $\mu,\lambda,\gamma,a,\tilde{\rho},\Omega,M$ and the matrix $A$, where $g$ is given in \eqref{compatibility-condition}. Here, we only sketch the higher-order estimates in the following lemma, which have been proved in \cite{Cai-Li-Lv2021}.

\begin{lem}\label{lem-higher-order-estimate}
  Under the conditions of Theorem \ref{thm-global-CNS-exterior}, it holds that
  \begin{equation*}
    \begin{aligned}
    &\sup_{t\in[0,T]}\int\rho|\dot{u}|^2+\int_0^T\norm{\nabla\dot{u}}_{L^2}^2\leq C,\\
    &\sup_{t\in[0,T]}(\norm{\nabla\rho}_{L^2\cap L^6}+\norm{\nabla u}_{H^1})+\int_0^T(\norm{\nabla u}_{L^{\infty}}+\norm{\nabla^2u}_{L^6}^2)\leq C,\\
    &\sup_{t\in[0,T]}\norm{\sqrt{\rho}u_t}_{L^2}^2+\int_0^T\norm{\nabla u_t}_{L^2}^2\leq C,\\
    &\sup_{t\in[0,T]}(\norm{\rho}_{H^2}+\norm{P}_{H^2})\leq C,\\
    &\sup_{t\in[0,T]}(\norm{\rho_t}_{H^1}+\norm{P_t}_{H^1})+\int_0^T(\norm{\rho_{tt}}_{L^2}^2+\norm{P_{tt}}_{L^2}^2)\leq C,\\
    &\sup_{t\in[0,T]}\sigma\norm{\nabla u_t}_{L^2}^2+\int_0^T\sigma\norm{\sqrt{\rho}u_{tt}}_{L^2}^2\leq C,\\
    &\sup_{t\in[0,T]}\sigma\norm{\nabla u}_{H^2}^2+\int_0^T(\norm{\nabla u}_{H^2}^2+\norm{\nabla^2u}_{W^{1,q}}^{p_0}+\sigma\norm{\nabla u_t}_{H^1}^2)\leq C,\\
    &\sup_{t\in[0,T]}(\norm{\rho}_{W^{2,q}}+\norm{P}_{W^{2,q}})\leq C,\\
    &\sup_{t\in[0,T]}\sigma(\norm{\nabla u_t}_{H^1}+\norm{\nabla u}_{W^{2,q}})+\int_0^T\sigma^2\norm{\nabla u_{tt}}_{L^2}^2\leq C,
    \end{aligned}
  \end{equation*}
  for $q\in(3,6)$ and $p_0=\frac{9q-6}{10q-12}\in(1,\frac{7}{6})$.
\end{lem}
Thus, combining Proposition \ref{prop-a-priori-estimate} with the higher-order estimates above as well as the local existence in Lemma \ref{lem-local-sol}, we can prove Theorem \ref{thm-global-CNS-exterior} by similar arguments as in \cite{Cai-Li-Lv2021}. Here, we omit the details for simplicity.

\appendix
\section{The mathematical analysis on three terms about density.}
In this appendix, we will give some mathematical analysis on the precise relationship among $(\rho-\bar{\rho})^2$, $G(\rho)$ and $(P(\rho)-P(\bar{\rho}))(\rho-\bar{\rho})$.

\begin{lem}\label{lem-relationship}
  There exists a clear relationship among $(\rho-\bar{\rho})^2$, $G(\rho)$ and $(P(\rho)-P(\bar{\rho}))(\rho-\bar{\rho})$ for any $\rho\in[0,\tilde{\rho}]$ and $\gamma\in(1,\frac{3}{2}]$. If $\bar{\rho}\ll\tilde{\rho}$, then we obtain
  \begin{equation}\label{relationship}
  \begin{aligned}
  &\frac{P(\bar{\rho})}{\bar{\rho}}(\rho-\bar{\rho})^2\leq(P(\rho)-P(\bar{\rho}))(\rho-\bar{\rho}),\\
  &(\rho-\bar{\rho})^2\leq\frac{1}{C_1}\tilde{\rho}\bar{\rho}^{1-\gamma}G(\rho),\\
  &\bar{\rho}G(\rho)\leq(P(\rho)-P(\bar{\rho}))(\rho-\bar{\rho}),
  \end{aligned}
\end{equation}
wihere constant $C_1>0$ depends only on $a$ and $\frac{\tilde{\rho}}{\bar{\rho}}$. In fact, it suffices to assume $\frac{\tilde{\rho}}{\bar{\rho}}\geq 3$ here.
\end{lem}

\begin{proof}
Due to $\bar{\rho}\ll\tilde{\rho}$, it is clear that $\bar{\rho}$ is much smaller than $\tilde{\rho}$. First, we set
\begin{equation*}
  f(\rho)=\frac{(P(\rho)-P(\bar{\rho}))(\rho-\bar{\rho})}{(\rho-\bar{\rho})^2}=\frac{P(\rho)-P(\bar{\rho})}{\rho-\bar{\rho}},
\end{equation*}
and let
\begin{equation*}
  f(\bar{\rho})=\lim_{\rho\rightarrow\bar{\rho}}f(\rho)=P'(\bar{\rho}).
\end{equation*}
Then, a direct calculation yields that
\begin{equation}\label{f-1rd-derivative}
  f'(\rho)=\frac{P'(\rho)(\rho-\bar{\rho})-(P(\rho)-P(\bar{\rho}))}{(\rho-\bar{\rho})^2}
  =(\rho-\bar{\rho})^{-2}\int_{\bar{\rho}}^{\rho}(P'(\rho)-P'(s))ds\geq0,
\end{equation}
which implies
\begin{equation}\label{f-bound}
  f(\rho)\in[f(0),f(\tilde{\rho})]=[\frac{P(\bar{\rho})}{\bar{\rho}},f(\tilde{\rho})].
\end{equation}

Then, we define
\begin{equation*}
  h(\rho)=\frac{G(\rho)}{(\rho-\bar{\rho})^2},
\end{equation*}
and also set
\begin{equation*}
  h(\bar{\rho})=\lim_{\rho\rightarrow\bar{\rho}}h(\rho)=\frac{1}{2}G''(\bar{\rho})=\frac{P'(\bar{\rho})}{2\bar{\rho}}.
\end{equation*}
A similar calculation gives that
\begin{equation}\label{h-1rd-derivative}
  h'(\rho)=\frac{G'(\rho)(\rho-\bar{\rho})-2G(\rho)}{(\rho-\bar{\rho})^3}=\frac{h_1(\rho)}{(\rho-\bar{\rho})^3},
\end{equation}
It is easy to verify that
\begin{equation*}
  h_1(\bar{\rho})=0,
\end{equation*}
and
\begin{equation}\label{h1-1rd-derivative}
  h_1'(\rho)=G''(\rho)(\rho-\bar{\rho})-G'(\rho)=\int_{\bar{\rho}}^{\rho}(G''(\rho)-G''(s))ds\leq0
\end{equation}
due to the decreasing monotonicity of $G''(\rho)=\frac{P'(\rho)}{\rho}$ for $\gamma\in(1,\frac{3}{2}]$.

Then, it holds that
\begin{equation*}
  h_1(\rho)=\begin{cases}
              >0, & \mbox{if }\rho<\bar{\rho}, \\
              <0, & \mbox{if }\rho>\bar{\rho},
            \end{cases}
\end{equation*}
which implies
\begin{equation*}
  h'(\rho)=\begin{cases}
              <0, & \mbox{if }\rho<\bar{\rho}, \\
              =\frac{1}{6}G'''(\bar{\rho})<0, &\mbox{if }\rho=\bar{\rho},\\
              <0, & \mbox{if }\rho>\bar{\rho}.
            \end{cases}
\end{equation*}
This means
\begin{equation}\label{h-bound1}
  h(\rho)\in[h(\tilde{\rho}),h(0)]=[h(\tilde{\rho}),\frac{P(\bar{\rho})}{\bar{\rho}^2}].
\end{equation}
Next, we turn to estimate $h(\tilde{\rho})$. Since $\bar{\rho}\ll\tilde{\rho}$, it holds that\footnote{The focus of this estimate is at the case $\gamma\rightarrow1$, so we need keep some intrinsical relation unchanged during the calculation as $\gamma$ tending to 1.}
\begin{equation}\label{h-bound-transition1}
  \begin{aligned}
  h(\tilde{\rho})&=\left(1-\frac{\bar{\rho}}{\tilde{\rho}}\right)^{-2}\tilde{\rho}^{-1}a\left(\frac{1}{\gamma-1}(\tilde{\rho}^{\gamma-1}-\bar{\rho}^{\gamma-1})
  +\bar{\rho}^{\gamma}(\tilde{\rho}^{-1}-\bar{\rho}^{-1})\right)\\
  &=a(1-B^{-1})^{-2}\left(\frac{1}{\gamma-1}B^{-(2-\gamma)}-\frac{\gamma}{\gamma-1}B^{-1}+B^{-2}\right)\bar{\rho}^{\gamma-2}\\
  &=a(1-B^{-1})^{-2}\left(\frac{1}{\gamma-1}(B^{\gamma-1}-\gamma)B^{-1}+B^{-2}\right)\bar{\rho}^{\gamma-2}\\
  &\geq a(1-B^{-1})^{-2}\left((\ln B-1)B^{-1}+B^{-2}\right)\bar{\rho}^{\gamma-2}\\
  &\geq\frac{a(\ln B-1)}{(1-B^{-1})^2}B^{-1}\bar{\rho}^{\gamma-2}=\frac{a(\ln B-1)}{(1-B^{-1})^2}\frac{\bar{\rho}^{\gamma-1}}{\tilde{\rho}},
  \end{aligned}
\end{equation}
where $B=\frac{\tilde{\rho}}{\bar{\rho}}\gg1$ (actually it suffices to set $B\geq 3$). Then there exists a constant $C_1>0$ depending only on $a$ and $\frac{\tilde{\rho}}{\bar{\rho}}$, such that
\begin{equation}\label{h-bound2}
  h(\rho)\in[C_1\frac{\bar{\rho}^{\gamma-1}}{\tilde{\rho}},\frac{P(\bar{\rho})}{\bar{\rho}^2}].
\end{equation}

Finally, define
\begin{equation*}
  k(\rho)=\frac{(P(\rho)-P(\bar{\rho}))(\rho-\bar{\rho})}{G(\rho)}=\frac{F(\rho)}{G(\rho)}
\end{equation*}
and also set
\begin{equation*}
  k(\bar{\rho})=\lim_{\rho\rightarrow\bar{\rho}}k(\rho)=\frac{2P'(\bar{\rho})}{P'(\bar{\rho})\bar{\rho}^{-1}}=2\bar{\rho}.
\end{equation*}
Then an analogous computation gives that
\begin{equation}\label{k-1rd-derivative}
  k'(\rho)=\frac{F'(\rho)G(\rho)-F(\rho)G'(\rho)}{G^2(\rho)}=\frac{k_1(\rho)}{G^2(\rho)}.
\end{equation}
It is easy to verify that
\begin{equation*}
  k_1(\bar{\rho})=0,
\end{equation*}
and
\begin{equation}\label{k1-1rd-derivative}
  \begin{aligned}
  k_1'(\rho)&=F''(\rho)G(\rho)-F(\rho)G''(\rho)\\
  &=[P''(\rho)(\rho-\bar{\rho})+2P'(\rho)][\frac{1}{\gamma-1}P(\rho)+P(\bar{\rho})(1-\frac{\gamma}{\gamma-1}\frac{\rho}{\bar{\rho}})]\\
  &\quad-(P(\rho)-P(\bar{\rho}))(\rho-\bar{\rho})\frac{P'(\rho)}{\rho}\\
  &=\frac{\gamma}{\gamma-1}P(\bar{\rho})P''(\rho)(\rho-\bar{\rho})(1-\frac{\rho}{\bar{\rho}})+2P'(\rho)G(\rho)\\
  &=\frac{2P'(\rho)}{\rho}[\rho G(\rho)-\frac{\gamma}{2}\frac{P(\bar{\rho})}{\bar{\rho}}(\rho-\bar{\rho})^2]=\frac{2P'(\rho)}{\rho}k_2(\rho).
  \end{aligned}
\end{equation}
Obviously, $k_2(\rho)$ satisfies that
\begin{equation*}
  k_2(\bar{\rho})=k_2'(\bar{\rho})=k_2''(\bar{\rho})=0,
\end{equation*}
and
\begin{equation}\label{k2-derivative}
  \begin{aligned}
  k_2'(\rho)&=\rho G'(\rho)+G(\rho)-P'(\bar{\rho})(\rho-\bar{\rho}),\\
  k_2''(\rho)&=\rho G''(\rho)+2G'(\rho)-P'(\bar{\rho})=2G'(\rho)+P'(\rho)-P'(\bar{\rho}),\\
  k_2'''(\rho)&=2G''(\rho)+P''(\rho)=2\frac{P'(\rho)}{\rho}+P''(\rho)\geq0.
  \end{aligned}
\end{equation}
This means
\begin{equation*}
  \begin{aligned}
  k_2''(\rho)&=\begin{cases}
                <0, & \mbox{if }\rho<\bar{\rho}, \\
                >0, & \mbox{if }\rho>\bar{\rho},
              \end{cases}\\
  k_2'(\rho)&\geq0,\\
  k_2(\rho)&=\begin{cases}
                <0, & \mbox{if }\rho<\bar{\rho}, \\
                >0, & \mbox{if }\rho>\bar{\rho},
              \end{cases}
  \end{aligned}
\end{equation*}
which implies
\begin{equation*}
  k_1(\rho)\geq 0,\quad k'(\rho)\geq0.
\end{equation*}
Thus, we get the bound of $k(\rho)$ as
\begin{equation}\label{k-bound}
  k(\rho)\in[k(0),k(\tilde{\rho})]=[\bar{\rho},k(\tilde{\rho})].
\end{equation}

Finally, we conclude from \eqref{f-bound}, \eqref{h-bound2} and \eqref{k-bound} that
\begin{equation*}
  \begin{aligned}
  &\frac{P(\bar{\rho})}{\bar{\rho}}(\rho-\bar{\rho})^2\leq(P(\rho)-P(\bar{\rho}))(\rho-\bar{\rho}),\\
  &(\rho-\bar{\rho})^2\leq\frac{1}{C_1}\tilde{\rho}\bar{\rho}^{1-\gamma}G(\rho),\\
  &\bar{\rho}G(\rho)\leq(P(\rho)-P(\bar{\rho}))(\rho-\bar{\rho}).
  \end{aligned}
\end{equation*}
and complete the proof of Lemma \ref{lem-relationship}.
\end{proof}

\begin{rem}\label{rem-relationship}
  Indeed, we can still show the following relationship between $(P(\rho)-P(\bar{\rho}))(\rho-\bar{\rho})$ and $G(\rho)$. As in \eqref{k-bound}, for any $\rho\in[0,\tilde{\rho}]$,
  \begin{equation}\label{k-another-bound1}
    (P(\rho)-P(\bar{\rho}))(\rho-\bar{\rho})\leq k(\tilde{\rho})G(\rho).
  \end{equation}
  Here, we aim to determine the upper bound of $k(\tilde{\rho})$. Similar to \eqref{h-bound-transition1}, we have
  \begin{equation}\label{k-another-bound-transition1}
    \begin{aligned}
    k(\tilde{\rho})&=\frac{\bar{\rho}^{\gamma+1}B^{\gamma+1}(1-B^{-\gamma})(1-B^{-1})}{\bar{\rho}^{\gamma}
    A\left(\frac{1}{\gamma-1}(B^{\gamma-1}-\gamma)+B^{-1}\right)}\\
    &=\frac{\bar{\rho}B^{\gamma}(1-B^{-\gamma})(1-B^{-1})}{\frac{1}{\gamma-1}(B^{\gamma-1}-\gamma)+B^{-1}}\\
    &\leq\frac{\bar{\rho}B^{\gamma}(1-B^{-\gamma})(1-B^{-1})}{\ln B-1+B^{-1}}\\
    &=\frac{\tilde{\rho}B^{\gamma-1}(1-B^{-\gamma})(1-B^{-1})}{\ln B-1+B^{-1}},
    \end{aligned}
  \end{equation}
  where $A=\frac{\tilde{\rho}}{\bar{\rho}}\geq3$ and $\gamma\in(1,\frac{3}{2}]$. However, if choosing $B\gg1$ as $\gamma\rightarrow1$, we can get a better estimate as
  \begin{equation}\label{k-another-bound-transition2}
    \begin{aligned}
    k(\tilde{\rho})&=\frac{\bar{\rho}B^{\gamma}(1-B^{-\gamma})(1-B^{-1})}{\frac{1}{\gamma-1}(B^{\gamma-1}-\gamma)+B^{-1}}\\
    &\leq\frac{\bar{\rho}B^{\gamma}(1-B^{-\gamma})(1-B^{-1})}{\frac{1}{3(\gamma-1)}B^{\gamma-1}}\\
    &\leq3(\gamma-1)\tilde{\rho},
    \end{aligned}
  \end{equation}
  under the condition
  \begin{equation}\label{A-restriction}
    B\geq3^{\frac{1}{\gamma-1}},
  \end{equation}
  which implies
  \begin{equation*}
    B^{\gamma-1}\geq3\geq2\gamma,
  \end{equation*}
  for any $\gamma\in(1,\frac{3}{2}]$.

  Thus, we conclude that if $B=\frac{\tilde{\rho}}{\bar{\rho}}\geq3^{\frac{1}{\gamma-1}}$, then for any $\gamma\in(1,\frac{3}{2}]$,
  \begin{equation}\label{k-another-bound2}
    (P(\rho)-P(\bar{\rho}))(\rho-\bar{\rho})\leq 3\tilde{\rho}(\gamma-1)G(\rho),
  \end{equation}
  and if $B=\frac{\tilde{\rho}}{\bar{\rho}}\in[3,3^{\frac{1}{\gamma-1}}]$ with $\gamma\in(1,\frac{3}{2}]$, it holds from \eqref{k-another-bound-transition1} that
  \begin{equation}\label{k-another-bound3}
    (P(\rho)-P(\bar{\rho}))(\rho-\bar{\rho})\leq \frac{3\tilde{\rho}}{\ln3-1}G(\rho).
  \end{equation}
\end{rem}

\medskip

\section*{\bf Data availability}
No data was used for the research described in the article.

\section*{\bf Conflicts of interest}

The authors declare no conflict of interest.

\end{document}